\documentclass[12pt]{article}
\usepackage[centering, left=0.7in,top=0.8in,right=0.7in,bottom=1in]{geometry}

\usepackage{amsmath,amssymb,amsthm,latexsym, amscd,epsfig,epstopdf,enumerate}
\usepackage{epsfig,graphicx,graphics,subfig,pdfpages}

\usepackage[utf8x]{inputenc}
\usepackage{amssymb,pdfsync,epstopdf,verbatim,gensymb}
\usepackage{nccmath}
\usepackage{textcomp}
\usepackage{amsmath}
\usepackage{float}
\usepackage{subfig}
\usepackage{amsfonts}
\usepackage{gensymb}
\usepackage{epsfig,color}
\usepackage{hyperref}
\newtheorem{assumption}{Assumption}[section]
\newtheorem{proposition}{Proposition}[section]
\newtheorem{theorem}{Theorem}[section]
\newtheorem{definition}{Definition}[section]
\newtheorem{lemma}{Lemma}[section]

\newtheorem{remark}{Remark}[section]

\graphicspath{{./pics/}}
\newcommand{\overbar}[1]{\mkern 1.5mu\overline{\mkern-1.5mu#1\mkern-1.5mu}\mkern 1.5mu}

\newcommand{\blue}{\color{blue}}

\title{How to place an obstacle having a dihedral symmetry centered at a given point inside a disk so as to optimize the fundamental Dirichlet eigenvalue}

\author{
Anisa M.H. Chorwadwala{\footnote{
anisa@iiserpune.ac.in, Indian Institute of Science Education and Research Pune,
Dr. Homi Bhabha Road, Pashan, Pune 411008,
India,
Tel: +91(20)25908218.}}
\and
Souvik Roy{\footnote{
souvik.roy@mathematik.uni-wuerzburg.de, Institut f\"ur Mathematik,
Universit\"at W\"urzburg,
Emil-Fischer-Strasse 30,
97074 W\"urzburg,
Germany,
Tel: +49 15213647226.}}
}

\date{}

\begin{document}
\maketitle

\begin{abstract} 
 A generic model for the shape optimization problems we consider in this paper is the optimization of the Dirichlet eigenvalues of the Laplace operator with a volume constraint. We deal with an obstacle placement problem which can be formulated as the following eigenvalue optimization problem: Fix two positive real numbers $r_1$ and $A$. We consider a disk $B\subset \mathbb{R}^2$ having radius $r_1$. We want to place an obstacle $P$ of area $A$ within $B$ so as to maximize or minimize the fundamental Dirichlet eigenvalue $\lambda_1$ for the Laplacian on $B\setminus P$. That is, we want to study the behavior of the function $\rho \mapsto \lambda_1(B\setminus\rho(P))$, where $\rho$ runs over the set of all rigid motions of the plane fixing the center of mass for $P$ such that $\rho(P)\subset B$. In this paper, we consider this obstacle placement problem for the case where (i) the obstacle $P$  is invariant under the action of a dihedral group $\mathbb{D}_n,~ n ≥ 3,$ $n$ even, (ii) $P$ and $B$ have distinct centers, and (iii) the boundary $\partial P$ of $P$ satisfy certain monotonicity condition between each pair of consecutive axes of symmetry of $P$. The extremal configurations correspond to the cases where an axis of symmetry of $P$ coincide with an axis of symmetry of $B$. We also characterize the maximizing and the minimizing configurations in our main result, viz., Theorem \ref{max_min}. Equation (\ref{even_lambda}), Propositions \ref{critical_points} and \ref{complete_critical_points} imply Theorem \ref{max_min}. We give many different generalizations of our result.
At the end, we provide some numerical evidence to validate our main theorem for the case where the obstacle $P$ has $\mathbb{D}_4$ symmetry. 

{ For the $n$ odd case, we identify some of the extremal configuration for $\lambda_1$. We prove that equation (\ref{even_lambda}) and Proposition \ref{critical_points} hold true for $n$ odd too. We highlight some of the difficulties faced in proving Proposition \ref{complete_critical_points} for this case. We provide numerical evidence for $n=5$ and conjecture that Theorem \ref{max_min} holds true for $n$ odd too.}
 \end{abstract}
 


\section{Introduction}
We start with a motivation for studying what is known as the shape optimization problems. We borrow this motivation and the introduction from \cite{anisa-aithal1}. Questions of the following type arise quite naturally. Why are small water droplets and bubbles that float in air approximately spherical? Why does a herd of reindeer form a circle if attacked by wolves? Why does a cat fold her body to form almost a round shape on a cold night? Can we hear the shape of a drum? Of all geometric figures having a certain property, which one has the greatest area or volume? And of all figures having a certain property, which one has the least perimeter or surface area? Mathematician have been trying to answer such questions via what is known as studying the shape optimization problems. A shape optimization problem typically deals with finding a shape which is optimal in the sense that it minimizes a certain cost functional among all shapes satisfying some given constraints. Mathematically speaking, it is to find a domain $\Omega$ that minimizes a cost functional $J(\Omega)$ possibly subject to a constraint of the form $G(\Omega)=0$. In other words, it is about minimizing a functional $J(\Omega)$ over a family $\mathcal{F}$ of admissible domains $\Omega$. That is, to find an optimal domain, $\Omega^*$ say, in $\mathcal{F}$ such that $\displaystyle J(\Omega^*)=\min_{\Omega \in \mathcal{F}}J(\Omega)$. In many cases, the functional being minimized depends on a solution of a given partial differential equation defined on a varying domain. The classical isoperimetric problem and its variants are examples of shape optimization problems.

Shape optimization problems arise naturally in different areas of science and engineering. In the context of spectral theory, these problems usually involve the study of eigenvalues of elliptic differential operators. Analysis of such problems is crucial in many physical applications which include designing of musical instruments so as to produce a desired sound \cite{kinsler, osher}, building of structures which are non-resonant to force \cite{tisseur}, analyzing the static equilibrium of a nonrigid water tank containing obstacles \cite{bonnans}, and designing of the optimal accelerator cavities \cite{accelerator}.

A generic model for such shape optimization problems is the optimization of the Dirichlet eigenvalues of the Laplace operator with a volume constraint. The origin of such problems dates back to 1800s when Rayleigh conjectured the famous isoperimetric inequality \cite{rayleigh}, which was proved by Faber \cite{faber} in 1923 and by Krahn \cite{krahn} in 1925, independently. Since then, there have been numerous notable research on the eigenvalue optimization problems involving various constraints. For a review of such results please refer to \cite{ashbaugh1, ashbaugh2, henrot, osserman}.
For a mini review of the kind of shape optimization problems that one of the authors along with her collaborators have worked on one may also refer to \cite{AnisaCS}.

The problem of the placement of an obstacle inside a given planar domain was first studied by Hersch \cite{hersch}. In the problem considered by him, the optimal configuration for the fundamental Dirichlet eigenvalue $\lambda_1$ for the Laplacian was characterized for the case where a circular obstacle is placed inside a disk. See also Ramm and Shivakumar \cite{Ramm-Shivakumar} for this case. Their results were subsequently extended to higher dimensional Euclidean spaces by Kesavan, and Harell et al., cf. \cite{kesavan, harrel_kurata}. %, wherein the extremum configurations of the domains for fundamental Dirichlet eigenvalue of the Laplacian were characterized. 
In \cite{harrel_kurata}, the case of multiple circular obstacles of possibly different sizes was also considered. In all these results the obstacles were balls in $\mathbb{E}^n$ and thus only translation of the obstacle/s affect the eigenvalues. Therefore, these obstacle placement problems reduce to just positioning of the center/s of the obstacle/s inside the outer disk. These results were further extended from the Euclidean case to all the three space forms in \cite{anisa_aithal} and later to all rank one symmetric spaces of non-compact type in \cite{anisa_vemuri}. The mini review article \cite{AnisaCS} gives a brief explanation of the difficulties faced in proving these generalizations and about how the respective authors overcame these difficulties. 

In \cite{elsoufi_kiwan}, an obstacle placement problem inside a planar domain was investigated for the case where (o) the obstacle $P_1$ and the domain $P_2$ had fixed areas, (i) the obstacle $P_1$ and the domain $P_2$ both were invariant under the action of the same dihedral group $\mathbb{D}_n$ $n ≥ 3$, (ii) the obstacle $P_1$ and the domain $P_2$ were concentric, (iii) the boundaries of $P_1$ and $P_2$ were simple closed $\mathcal{C}^2$ curves, (iv) between each pair of consecutive axes of symmetry of the obstacle $P_1$, a monotonicity assumption was made on its boundary $\partial P_1$, and (v) between each pair of consecutive axes of symmetry of the domain $P_2$ %(which coincide with the axis of symmetry of $P_2$ because of condition (i)), 
a monotonicity assumption was made on its boundary $\partial P_2$. For such pairs $P_1$ and $P_2$, they considered a family $\mathcal{F}$ of domains of the type $P_2 \setminus \overline{P_1}$.  Among $\mathcal{F}$, the extremal configurations for the fundamental Dirichlet eigenvalue $\lambda_1$ for the Laplacian were obtained by rotating the obstacle around its fixed center. The extremal configurations for $\lambda_1$ correspond to the cases where the axes of symmetry of the obstacle $P_1$ coincide with those of the domain $P_2$. In such configurations this common axis of symmetry of $P_1$ and $P_2$ then becomes the axis of symmetry of the $P_2 \setminus \overline{P_1}$.  Further, the characterizations of both the minimizing and the maximizing configurations for $\lambda_1$ are also obtained in \cite{elsoufi_kiwan}. 

In this paper, we prove a variant of the obstacle placement problem considered in \cite{elsoufi_kiwan}. We consider the case where the planar obstacle $P$ is invariant under the action of a dihedral group $\mathbb{D}_n$ $n ≥ 3$, $n$ even. It follows that the axes of symmetry of $P$ intersect in a unique point in the interior of $P$. We call this point the center of $P$ and denote it by $\underline{o}$. Let $B$ be a disk in $\mathbb{E}^2$ containing $\underline{o}$ away from its center. We place the obstacle $P$ centered at the fixed point $\underline{o}$ inside $B$. %containing $\underline{o}$ away from its center. %in $\mathbb{E}^2$ 
That is, the centers of $P$ and $B$ are distinct. In accordance with the notations of the previous paragraph, $P_1=P$ and $P_2 =B$ for us. The disk $B$ obviously is invariant under the action of dihedral groups $\mathbb{D}_n$, for each  $n ≥ 3$. Therefore, in our case, condition (i) of the above paragraph holds for some $n$, $n$ even, while condition (ii) does not hold. We, of course, assume the smoothness condition (iii) on both the boundaries and also assume the volume constraint (o) on $P$ and $B$ both. We further assume the monotonicity condition (iv) on the boundary $\partial P$ of the obstacle $P$ as in the previous paragraph. %Please note that an axis of symmetry of $P$ need not coincide with that 
We derive certain monotonicity condition on the boundary of the disk $B$ in Lemma \ref{bound_monot}. Therefore, for us condition (v) of the above paragraph for $P_2 =B$ is replaced by the statement of Lemma \ref{bound_monot}. In this setting, we investigate the extremal configurations of the obstacle $P$ with respect to the disk $B$ for the fundamental Dirichlet eigenvalue $\lambda_1$ for the Laplacian by rotating $P$, inside $B$, about the fixed center $\underline{o}$ of $P$. Such problems apply naturally, for example, to the designing of some musical instruments, where one usually has an asymmetric structure of the obstacle with respect to the domain. 

The proof in \cite{elsoufi_kiwan} relies mainly on the Hadamard perturbation formula and the reflection technique as in \cite{serrin}. Since both, the obstacle and the domain, had a dihedral symmetry and were concentric, it was enough for the authors to study the behavior of $\lambda_1$ with respect to the rotations of the obstacle by angle $\theta \in (0, \pi/n)$ where $\pi /n$ is nothing but the angle between two consecutive axes of symmetry of the obstacle $P$. The proof in \cite{elsoufi_kiwan} %in \cite{elsoufi-kiwan}
works for obstacles with $\mathbb{D}_n$ symmetry for any $n \geq 3$, odd as well as even. 
%{ We might have to remove this sentence later.} 

{ In this current work, because of the lack of such a symmetry, as $P$ and $B$ are not concentric, the analysis of the behavior of $\lambda_1$ is more challenging. Recall that $P$ has a $\mathbb{D}_n$ symmetry. We prove our main  theorem, viz. Theorem \ref{max_min} for $n$ even and {highlight some of the difficulties faced in proving the result for $n$ odd.} %Proposition \ref{complete_critical_points}indicate briefly why a few similar ideas don't work for $n$ odd. 

For the $n$ even case, we analyze the behavior of $\lambda_1$  in two different hemispheres of the disk $B$ separately. We perform this analysis using an appropriate domain reflection technique. %We will call this domain reflection technique as `Sector Reflection Technique'. 
Since the obstacle $P$ we consider has a $\mathbb{D}_n$ symmetry, if we take $n$ to be even, $n\geq 3$, the axes of symmetry of $P$ divide $B$ in even number of sectors in each of these hemispheres. This helps in pairing up two consecutive sectors in each of these hemispheres. We then reflect the smaller sector of the two into the larger one using the reflection about the axis of symmetry separating these two sectors. It makes sense to call this domain reflection technique as sector reflection technique. %{ Remove the following sentence.} Therefore, sector reflection technique will suffice when $n$ is even. 

For the $n$ odd case, %Since the obstacle $P$ we consider has a $\mathbb{D}_n$ symmetry, if we take $n$ to be {\blue odd}, $n\geq 3$, 
the axes of symmetry of $P$ divide $B$ in {\blue odd} number of sectors in each of these hemispheres. Therefore, it's not possible to find a complete pairing of consecutive sectors within each of the hemispheres, and hence the sector reflection technique mentioned above doesn't work.}  %{ inversion technique doesn't seem to work.}}

In the next section, in order to introduce the family of domains over which we are going to carry out the eigenvalue optimization analysis, we list the assumptions made on them. We also give a few definitions so as to identify the various different configurations in the family of domains under consideration. 

In section \ref{Prel}, we prove a monotonicity property on the boundary of an arbitrary disk $B$, see Lemma \ref{bound_monot}, using the representation $B$ in polar coordinates with respect to a point other than its center. We then consider a planar simply connected bounded domain $K$ and represent it in polar co-ordinates with respect to the origin in $\mathbb{R}^2$. We consider the unit outward normal vector field to $K$ on its boundary $\partial K$. We call this vector field $\eta$. We derive an expression for $\eta$ in the polar co-ordinates. We then consider a smooth vector field $v$ in $\mathbb{R}^2$ that rotates the domain $K$ by a right angle about the origin in the anticlockwise direction. We then derive the expression, in polar coordinates, for the inner product of these two vector fields evaluated at a boundary point. The lemmas of section \ref{Prel} are useful in proving our main theorem, viz., Theorem \ref{max_min}. 

In Section \ref{stmnt}, we state our main theorem, viz., Theorem \ref{max_min} describing the extremal configurations for $\lambda_1$ over the family of admissible domains. This theorem also characterizes the maximizing and the minimizing configurations for $\lambda_1$.

In section \ref{proof}, we give a proof of Theorem \ref{max_min} { for $n$ even, $n \geq 3$}. We first justify that the fundamental Dirichlet eigenvalue $\lambda_1$ of the Laplacian for the family of domains under consideration is a function of just one real variable and that it is an even periodic function of period $2 \pi$. Therefore, in order to determine the extremal configuration/s for $\lambda_1$ we study the behavior of its derivative. The Hadamard perturbation formula (\ref{Hadamard}) becomes useful in this analysis. We identify some of critical points for $\lambda_1$ in Proposition \ref{critical_points}. In view of equation (\ref{even_lambda}) Propositions \ref{critical_points} and \ref{complete_critical_points} imply that (a) the critical points listed in Proposition \ref{critical_points} are the only critical points for $\lambda_1$ and that (b) between every pair of consecutive critical points, $\lambda_1$ is a strictly monotonic function of the argument. { %We introduce a `sector reflection technique', similar to the domain refection technique; and also introduce a `rotating plane method', similar to the moving plane method. For the $n$ odd case, we identify some of the extremal configuration for $\lambda_1$. 
We prove that equation (\ref{even_lambda}) and Proposition \ref{critical_points} hold true for $n$ odd too. We highlight some of the difficulties faced in proving Proposition \ref{complete_critical_points} for this case.}
%For odd $n$, we use a combination of the sector reflection technique and a sector inversion technique followed by a rotating plane method. { inversion technique doesn't seem to work.}} %We use a domain reflection technique, rather a sector reflection technique; and a moving plane method, rather a rotating plane method, in the proof. %Theorems \ref{critical_points} and \ref{complete_critical_points} then immediately imply our main theorem, viz., Theorem \ref{max_min}. 

 In Section \ref{sec:remarks}, we talk about generalizations of Theorem \ref{max_min} to differential equations involving Schr\"{o}dinger-type operators. The result is still valid if instead of a hard obstacle we consider soft obstacles or wells. A theorem similar to Theorem \ref{max_min} also holds for the energy functional associated with the stationary Dirichlet boundary value problem (\ref{stationary}). We then generalize the result to planar obstacles with non-smooth polygonal boundary. We then talk about some generalizations from the Euclidean case to some other Riemannian manifolds of dimension 2 known as space forms, i.e., complete simply connected Riemannian manifolds having constant sectional curvature.  

In Section \ref{sec:num_results}, we provide some numerical evidence supporting Theorem \ref{max_min} { for $n$ even. We also provide numerical evidence for $n=5$ and conjecture that Theorem \ref{max_min} holds true for $n$ odd too.}%We end this paper with a section on concluding remarks.  
%At the end, we talk about the extension of our results to some two-dimensional manifolds. We also talk about the obstacle placement problems with soft obstacles and wells for a Schr\"{o}dinger-type operator.

\section{The family of admissible domains and various configurations}\label{sec:prelim}
In this section, in order to introduce the family of domains over which we are going to carry out the eigenvalue optimization analysis, we list the assumptions made on them. We also give a few definitions so as to identify the various different configurations in the family of domains under consideration. {In this section, $n$ is a positive integer, $n \geq 3$, even or odd.}

\subsection{The family of admissible domains}\label{TheDomain}
Let $n$ be a positive integer, $n \geq 3$. Consider the dihedral group $\mathbb{D}_n$ generated by a rotation $r$ of order $n$ and %$\rho_{2\pi/n}$ and a reflection $R$. Dn is generated by a rotation r of order n and
 a reflection $s$ of order 2 such that
$\displaystyle s r s=r^{-1}$. Here, $r $ is a rotation by an angle $2 \pi/n$. Fix $A>0$. Let $P$ denote a compact simply connected subset of the Euclidean plane $\mathbb{E}^2$ satisfying the following assumptions:
\begin{assumption}. \label{assumption_A}
\begin{enumerate}[(a)]
\item the boundary $\partial P$ of $P$ is a simple closed $\mathcal{C}^2$ curve in $\mathbb{R}^2$,
\item  $P$ has a $\mathbb{D}_{n}$ symmetry for some $n \geq 3$, $n$ even, i.e., $P$ is invariant under the action of a dihedral group $\mathbb{D}_{n}$ for some $n \geq 3$, % {\blue even or odd},
\item the area of $P$ is $A$.
\end{enumerate}
\end{assumption}
\begin{figure}[H]\centering
\subfloat{\includegraphics[width=0.4\textwidth]{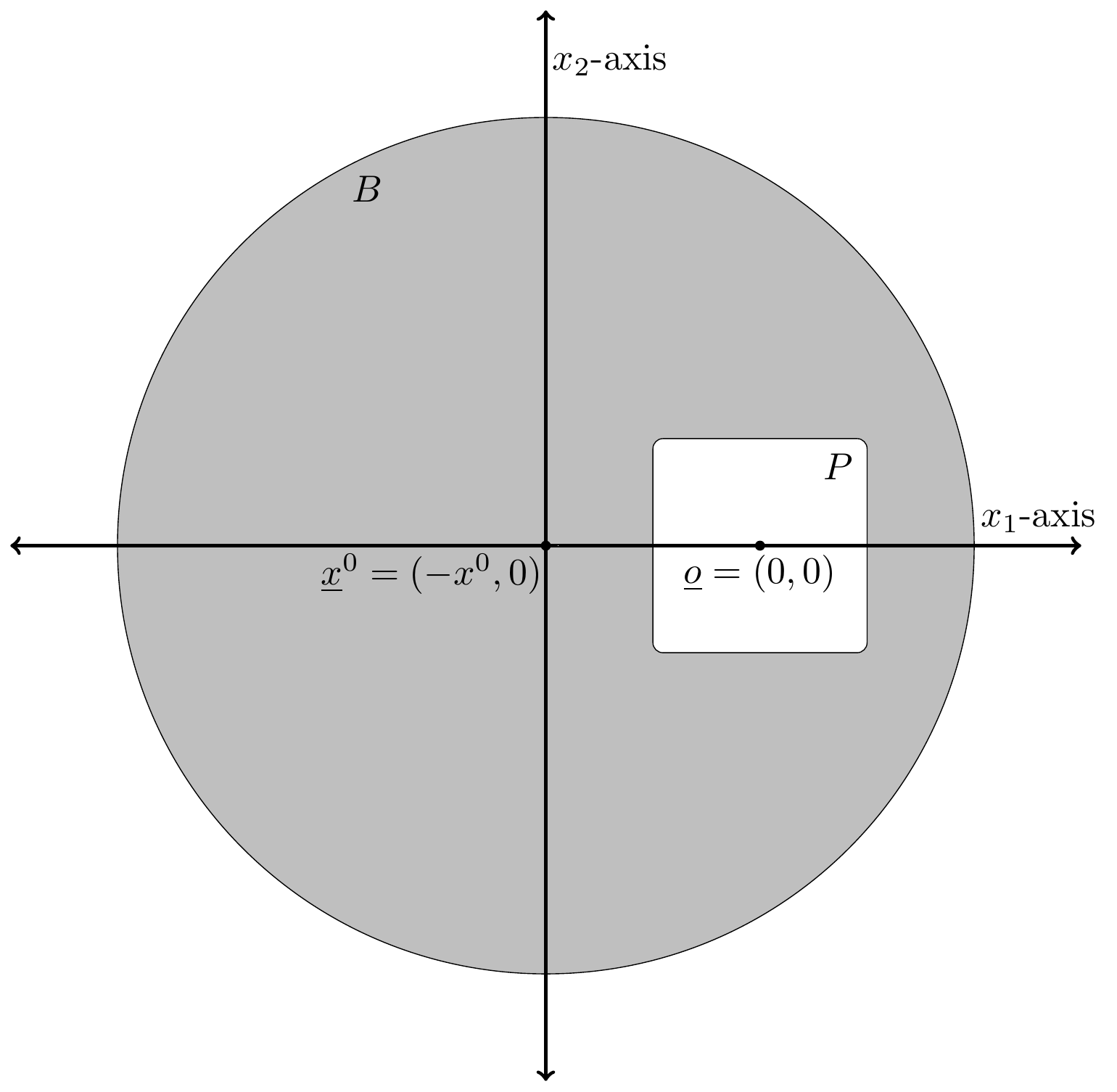}    }\hspace{10mm}
\subfloat{\includegraphics[width=0.4\textwidth]{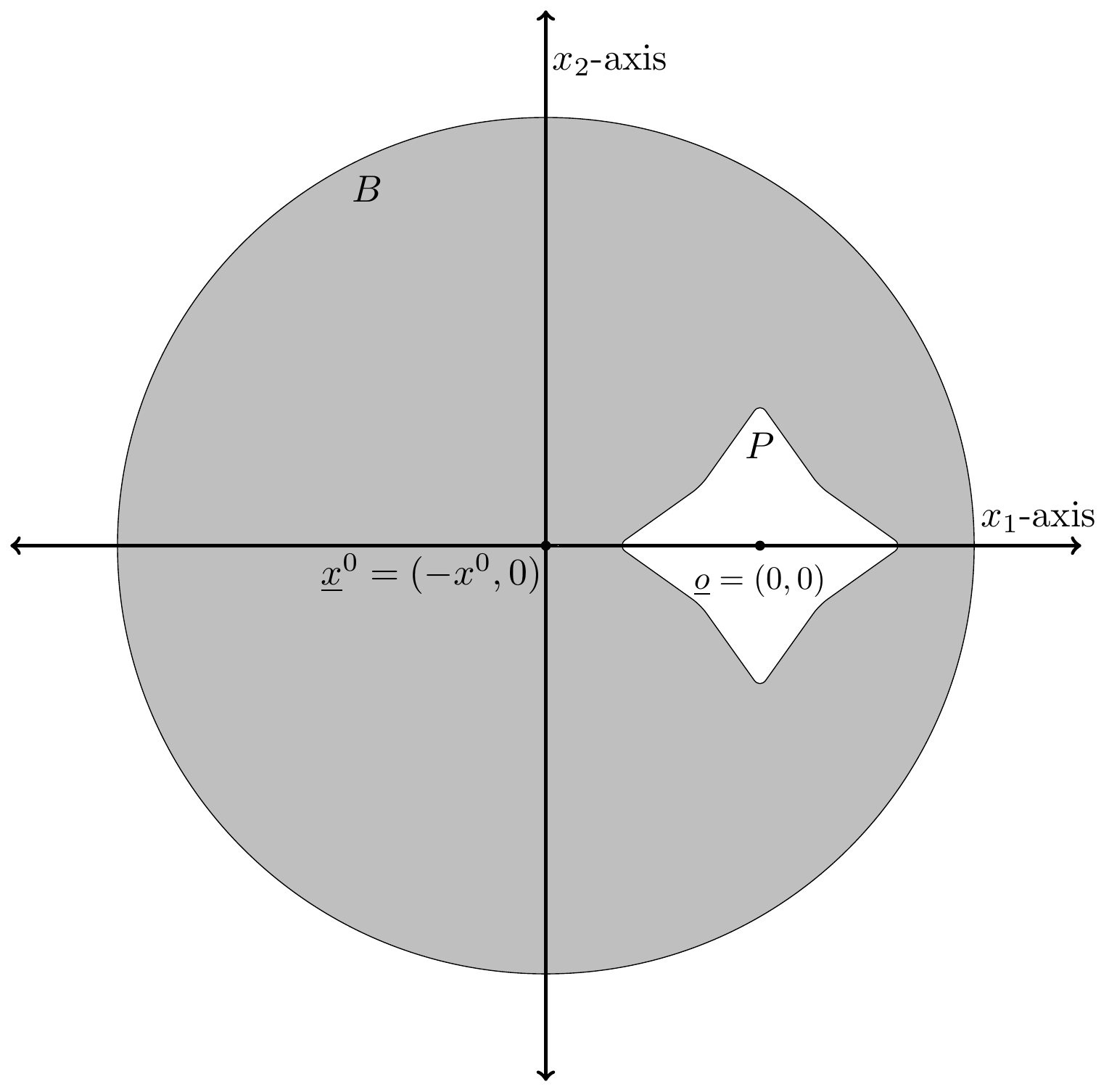}    }
\caption{Obstacles having $\mathbb{D}_4$ symmetry }\label{fig:domain}
\end{figure}
\begin{figure}[H]\centering
\subfloat{\includegraphics[width=0.4\textwidth]{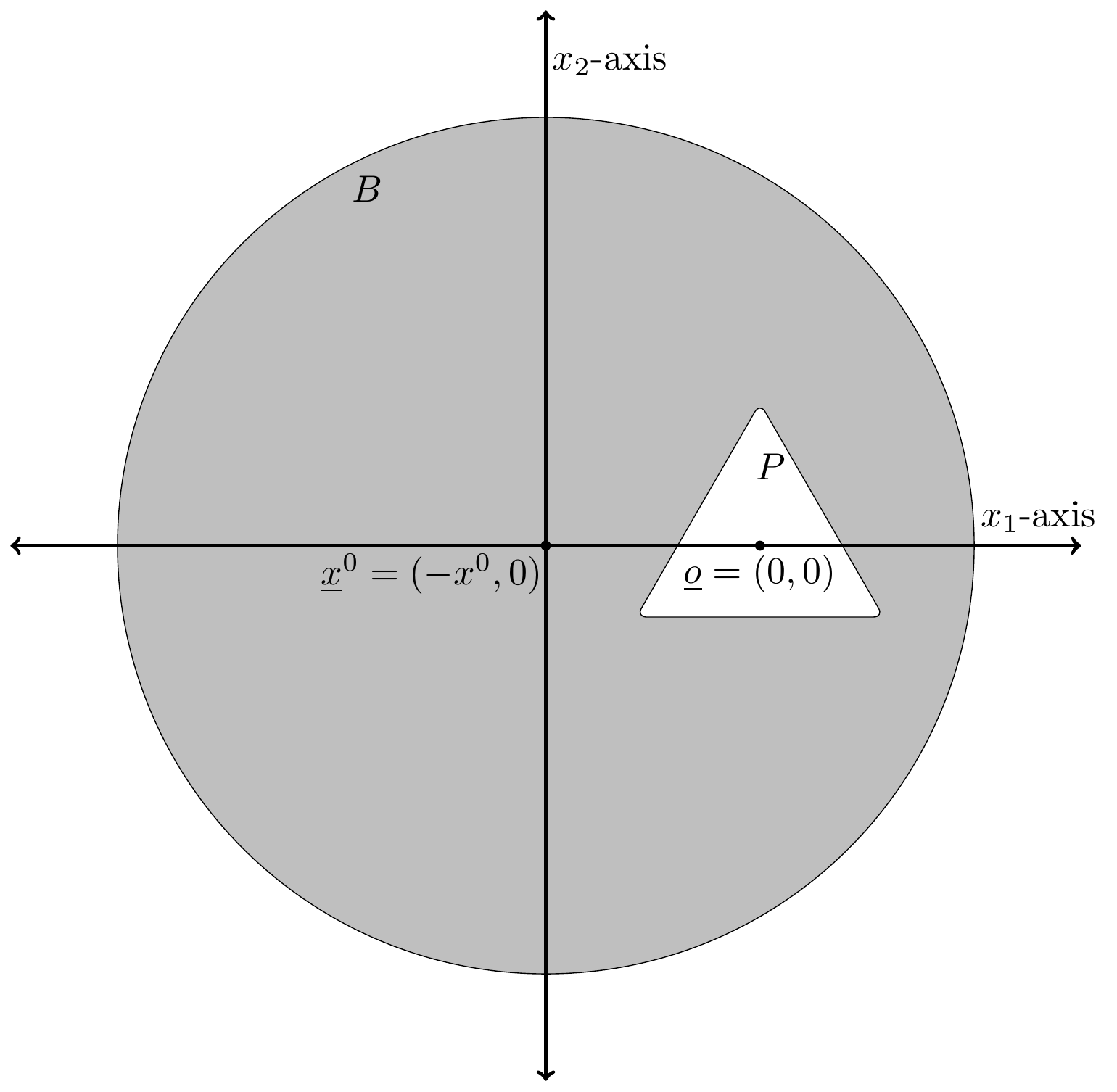}    }\hspace{10mm}
\subfloat{\includegraphics[width=0.4\textwidth]{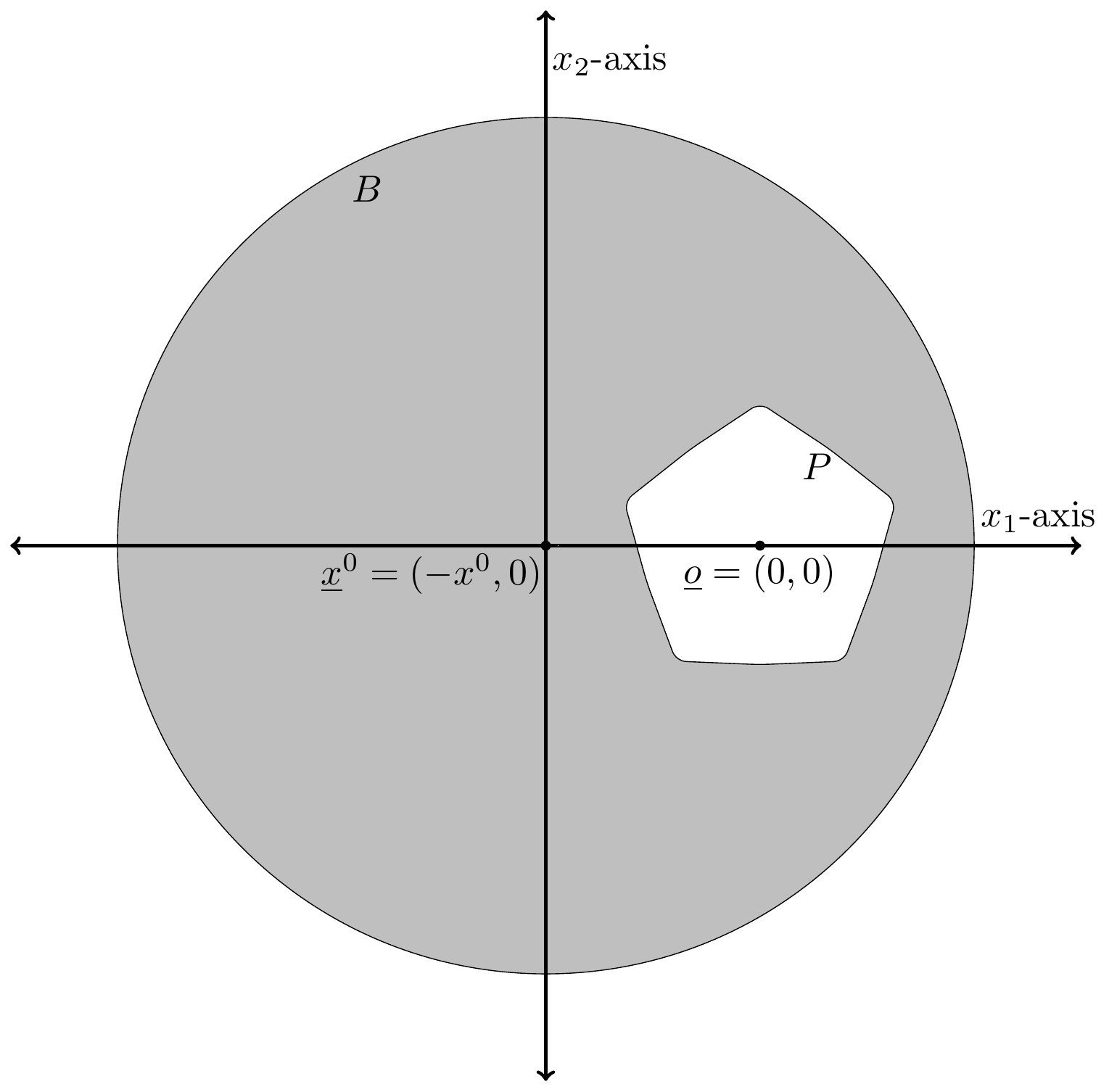}    }
\caption{Obstacles having $\mathbb{D}_3$ and $\mathbb{D}_5$ symmetry, respectively.%, $n$ odd
}\label{fig:domain_odd}
\end{figure}

It follows from the above conditions that the axes of symmetry of $P$ intersect in a unique point in the interior of $P$. We call this point the center $\underline{o}$ of $P$. Without loss of generality we assume that $\underline{o}$ is the origin $(0,0)$ of $\mathbb{R}^2$. The axes of symmetry of $P$ divide $\mathbb{R}^2$ in $2n$ components. We call each of these $2n$ components as sectors, and denote them by $S_i$, $1 \leq i \leq 2n$. We further make the following assumption:

\begin{assumption}. \label{assumption_B}
\begin{enumerate}[(d)]
\item the monotonicity of the boundary $\partial P$, that is, the distance $d(\underline{o},x)$, between the center $\underline{o}$ of $P$ and the point $x$ on the boundary $\partial P$ of $P$, is monotonic as a function of the argument $\phi$ in a sector delimited by two consecutive axes of symmetry of $P$. 
\end{enumerate}
\end{assumption}
We note that assumptions \ref{assumption_A} and \ref{assumption_B} imply that $P$ is a star-shaped domain with respect to its center $\underline{o}$. 
\begin{definition}[Incircle and circumcircle]
Let $P$ be a compact simply connected subset of $\mathbb{R}^2$ satisfying assumptions  \ref{assumption_A}, \ref{assumption_B} and centered at $\underline{o}$. By an incircle of $P$ we mean the largest circle in $\mathbb{R}^2$ centered at $\underline{o}$ that fits completely in $P$ and which is tangent to $\partial P$ in each of its $2n$ sectors. By a circumcircle of $P$ we mean the smallest circle in $\mathbb{R}^2$ centered at $\underline{o}$ that contains $P$ and which is tangent to $\partial P$ in each of its $2n$ sectors. Let $C_1(P)$ (resp. $C_2(P)$) denote the incircle (resp. the circumcircle) of $P$. When the set $P$ is fixed, we will simply refer to the incircle as $C_1$ and the circumcircle as $C_2$. Please note here that $C_1(\rho(P)) =C_1(P)$ and $C_2(\rho(P))= C_2(P)$ for each $\rho \in \mathbb{D}_n$. 
\end{definition}

\par  Let $co(A)$ denote the convex hull of a subset $A$ in $\mathbb{R}^2$ and let $\overline{co(A)}$ denote its closure. Clearly, for a compact simply connected subset $P$ of the Euclidean plane $\mathbb{E}^2$ satisfying Assumptions \ref{assumption_A} and \ref{assumption_B} we have, $P \subset \overline{co(C_2(P))}$ and hence $\rho(P) \subset \overline{co(C_2(P))}$ for each $\rho \in \mathbb{D}_n$. We now take an open disk $B$ in $\mathbb{R}^2$ with radius $r_1 >0$ such that $B \supset \overline{co(C_2(P))}$.
\begin{figure}[H]\centering

\subfloat[]{
\includegraphics[width=0.5\textwidth]{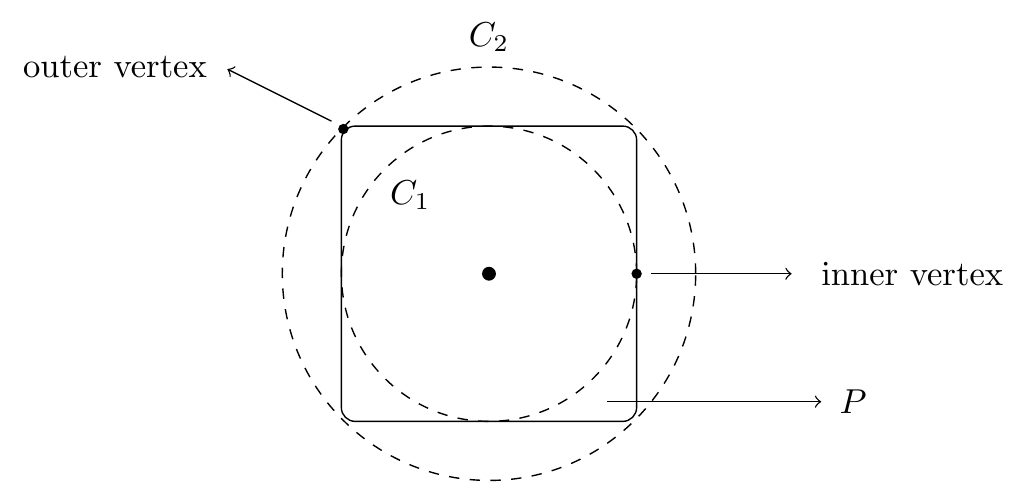}}
\subfloat[]{
\includegraphics[width=0.5\textwidth]{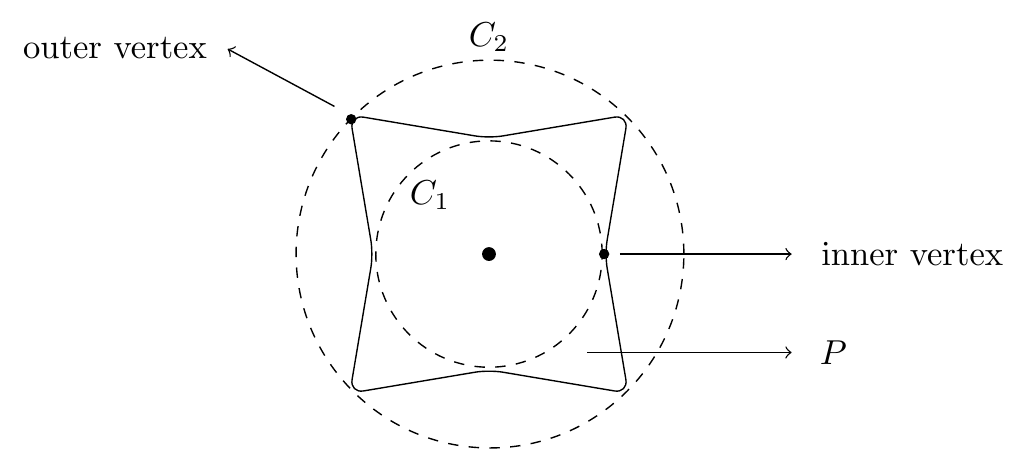}}\\
\subfloat[]{
\includegraphics[width=0.5\textwidth]{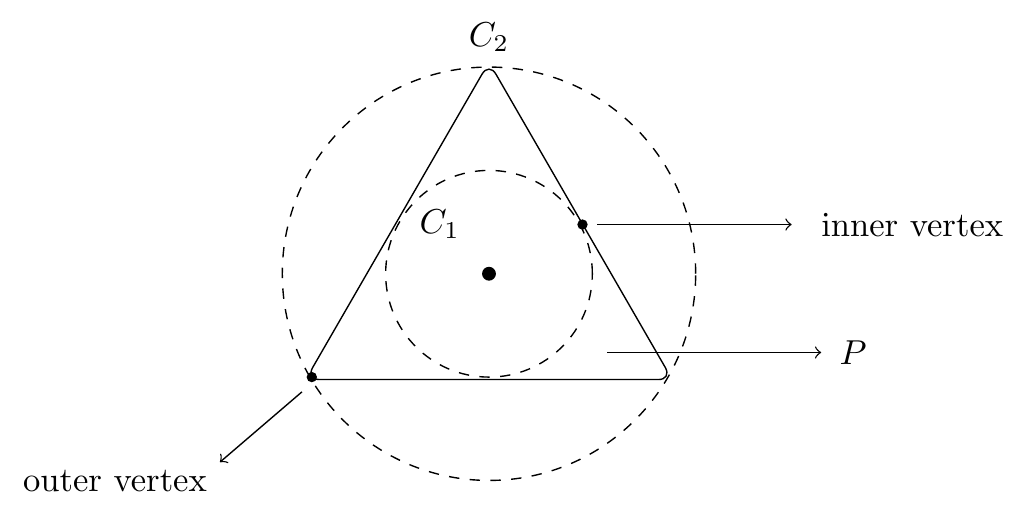}}
\subfloat[]{
\includegraphics[width=0.5\textwidth]{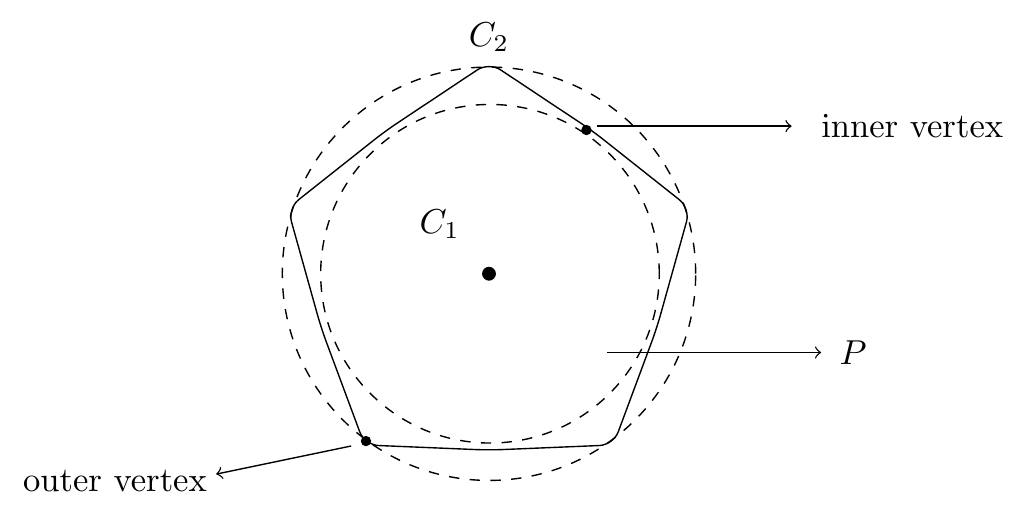}}
\label{Vertex}
\caption{Vertices of $P$ }
\end{figure}

%$B \supset C_2(P) \supset P$. %We have the following representation of $B$ in polar co-ordinates.
%\begin{enumerate}
% \item\label{1} (Representation of $B$)- 
%Fix $r_1>0$. Let $B$ be an open disk in $\mathbb{R}^2$ with radius $r_1$ such that $C_2 \subset B$. 
%In polar co-ordinates, $B $ can be represented as the set $\lbrace{re^{i\phi}: \phi\in[0,2\pi), 0\leq r < g(\phi)\rbrace}$. Remember here that $(r,\phi)$ is measured with respect to $\underline{o}=(0,0)$.
%\item\label{2} (Monotonicity of the boundary $\partial B$)- The distance, $d(\underline{o},s)$, between the center $\underline{o}$ of $P$ and the point $s$ on the boundary $\partial B$ of $B$ is monotonic as a function of the argument $\phi$ in a sector delimited by two consecutive axes of symmetry of $P$.
%\end{enumerate}
%Please note we will see in Lemma \ref{bound_monot} that condition 2 above is in fact redundant.
\subsection{The OFF and the ON positions}
Let $n$ to be a positive integer, $n \geq 3$. For $P$,  a compact simply connected subset of $\mathbb{R}^2$ satisfying assumptions  \ref{assumption_A} and \ref{assumption_B}, recall that $C_1$ and $C_2$ denote the incircle and the circumcircle of $P$ respectively. We define the {\it inner vertex set} $V_{in}$ and the {\it outer vertex set}  $V_{out}$ of $P$ as follows:  $$V_{in}:= \partial P \cap C_1 ~~~ \mbox{ and } ~~~~~ V_{out}:= \partial P \cap C_2.$$ By a vertex set $V$ we simply mean $V_{in} \cup V_{out}$. Elements of $V_{in}$ (resp. $V_{out}$) will be called {\it inner vertices} (resp. {\it outer vertices}) {\it of} $P$. Elements of $V$ will simply be referred to as {\it vertices of} $P$. A radial segment of the incircle $C_1$ of $P$ containing an inner vertex will be referred to as an inradius of $P$, and likewise, a radial segment of the circumcircle $C_2$ of $P$ containing an outer vertex of $P$ will be referred to as a circumradius of $P$.

\par  %Let $co(A)$ denote the convex hull of a subset $A$ in $\mathbb{R}^2$ and let $\overline{co(A)}$ denote its closure. 
As described in section \ref{TheDomain}, let $P$ be a compact simply connected subset of $\mathbb{R}^2$ satisfying assumptions \ref{assumption_A}, \ref{assumption_B}; and let $B$ be an open disk in $\mathbb{R}^2$ of radius $r_1$ such that $B \supset \overline{co(C_2(P))}$. {Since $\lambda_1$ is invariant under isometries of $\mathbb{R}^n$, without loss of generality we make the following assumptions: (a) The centers of $B$ and $P$ are on the $x_1$-axis, (b) the center of $P$ is at the origin, and (c) the center of $B$ is on the negative $x_1$-axis. We say that $P$ is in an OFF position with respect to $B$ if an inner vertex of $P$ is on the negative $x_1$-axis and that $P$ is in an ON position if an outer vertex of $P$ is on the negative $x_1$-axis.} %Clearly, $P \subset \overline{co(C_2(P))}$. 

{%For a vertex $V$ of $P$ we say that $V^\prime$ is the opposite vertex of $V$ 
If two vertices of $P$ lie on the same axis of smmetry of $P$ then they are called opposite vertices of each other. Note here that, if a vertex of $P$ is on the negative $x_1$-axis then the corresponding opposite vertex of $P$ is going to be on the positive $x_1$-axis. For $n$ even, the vertex opposite to an inner vertex is also an inner vertex. Whereas, for $n$ odd, the vertex opposite to an inner vertex is going to be an outer vertex and vice versa. Therefore, for $n$-odd, we can say that $P$ is in an OFF position with respect to $B$ if an outer vertex of $P$ is on the positive $x_1$-axis and that $P$ is in an ON position if an inner vertex of $P$ is on the positive $x_1$-axis. But this isn't true for $n$ even.}%the OFF and ON configurations coincide.
%\begin{equation}\label{t_configuration}
%\partial P_t : =  \{f(\phi-t) e^{i \phi}\, |\, \phi \in [0, 2\pi)\}.
%\end{equation}
\begin{figure}[H]\centering
\subfloat[OFF configuration]{
\label{Off position}
\includegraphics[width=0.18\textwidth]{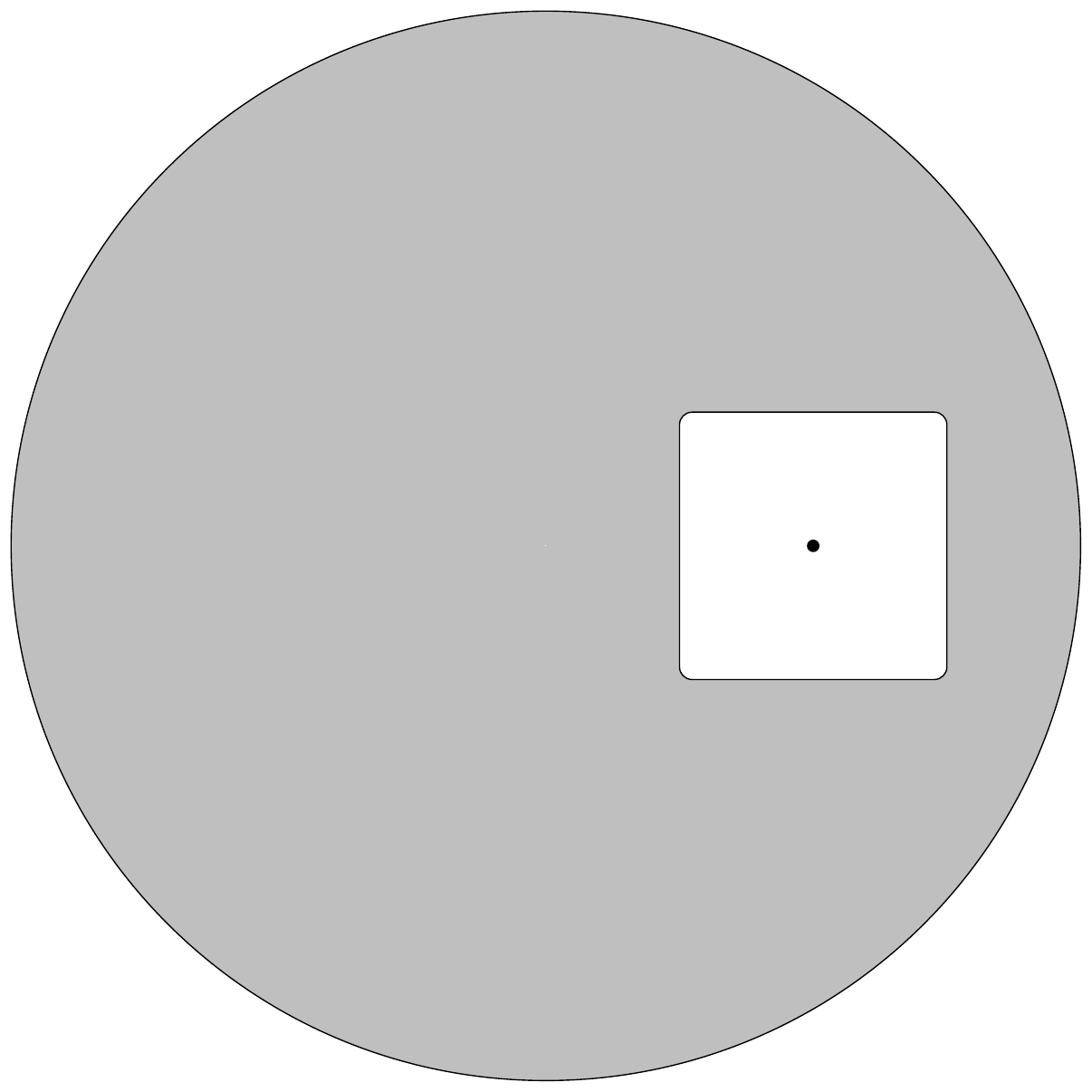}}
\hspace{12mm}
\subfloat[ON configuration]{
\label{On position}
\includegraphics[width=0.18\textwidth]{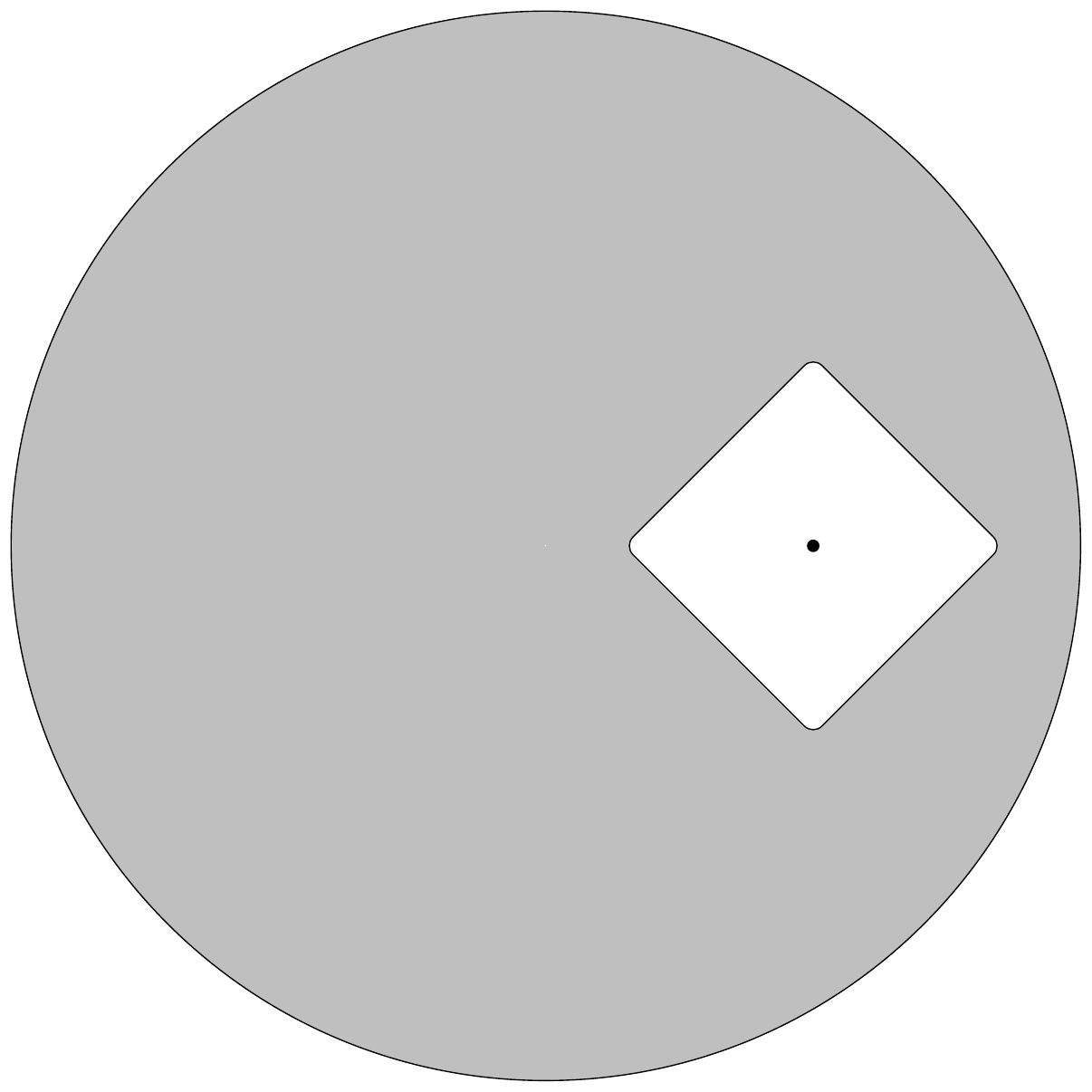}}
\hspace{14mm}
\subfloat[OFF configuration]{
\label{Off position2}
\includegraphics[width=0.18\textwidth]{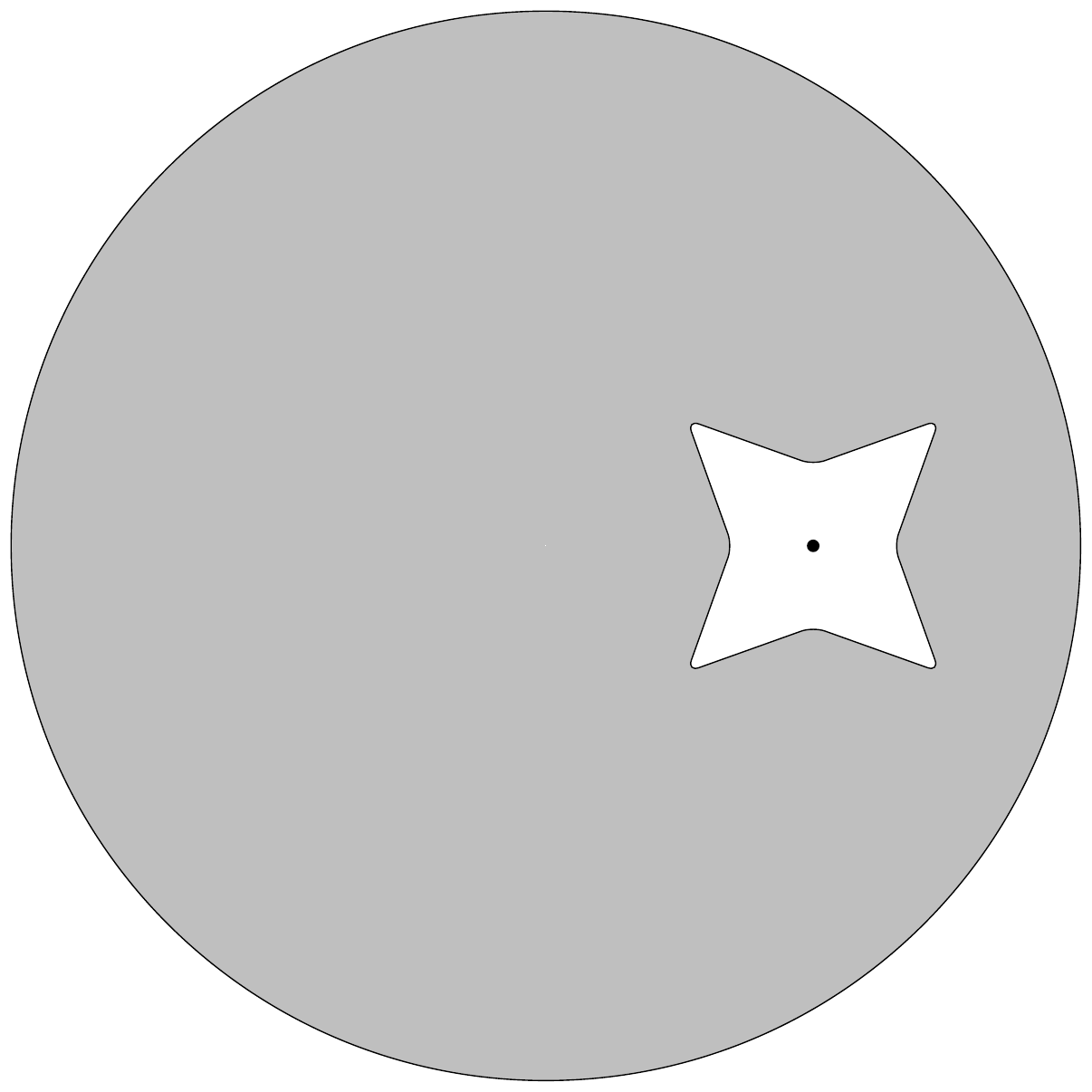}}
\hspace{12mm}
\subfloat[ON configuration]{
\label{On position2}
\includegraphics[width=0.18\textwidth]{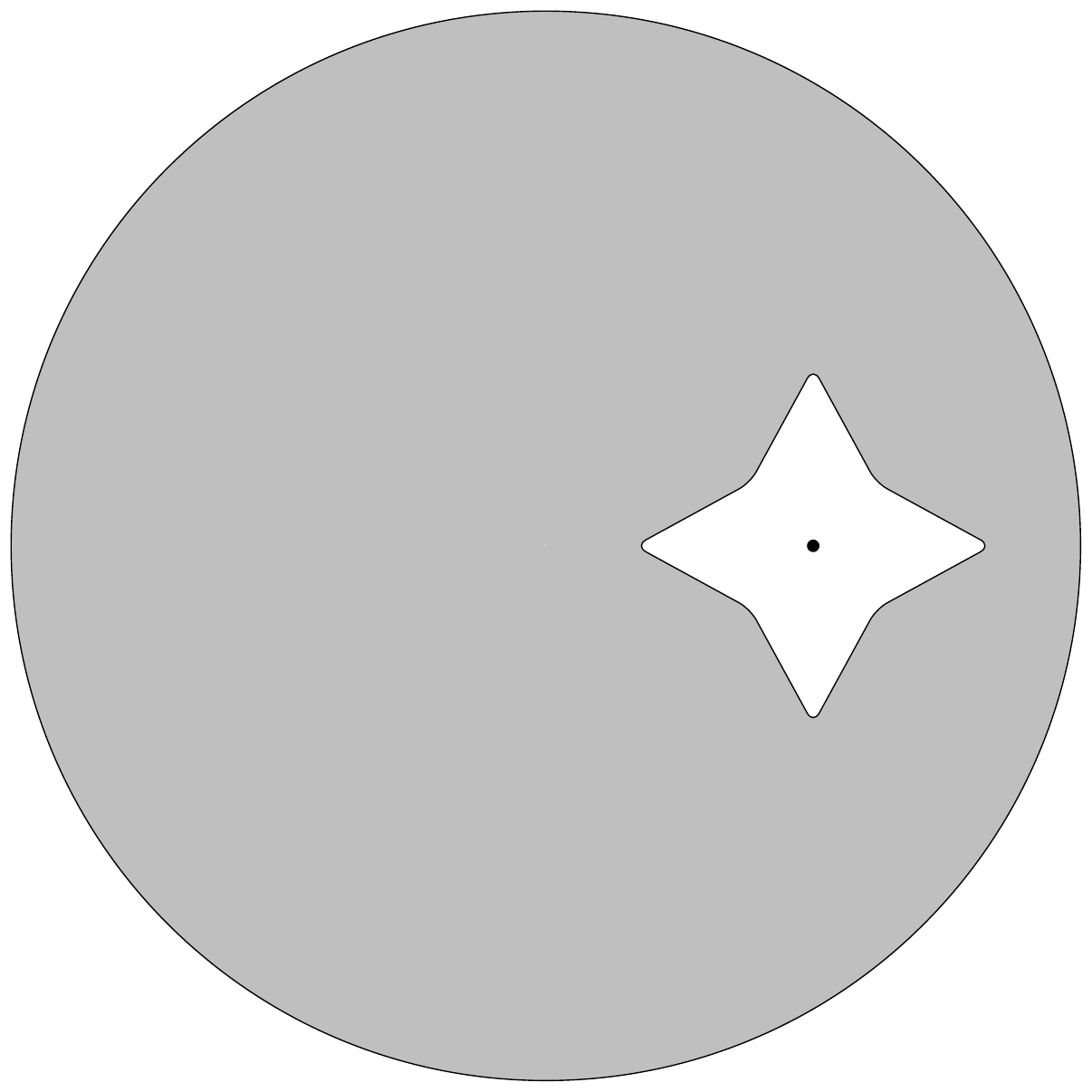}}
\caption{OFF and ON configurations for obstacles having $\mathbb{D}_4$ symmetry}\label{fig:on_off}
\end{figure}

\begin{figure}[H]\centering
\subfloat[OFF configuration]{
\label{Off position_triangle}
\includegraphics[width=0.18\textwidth]{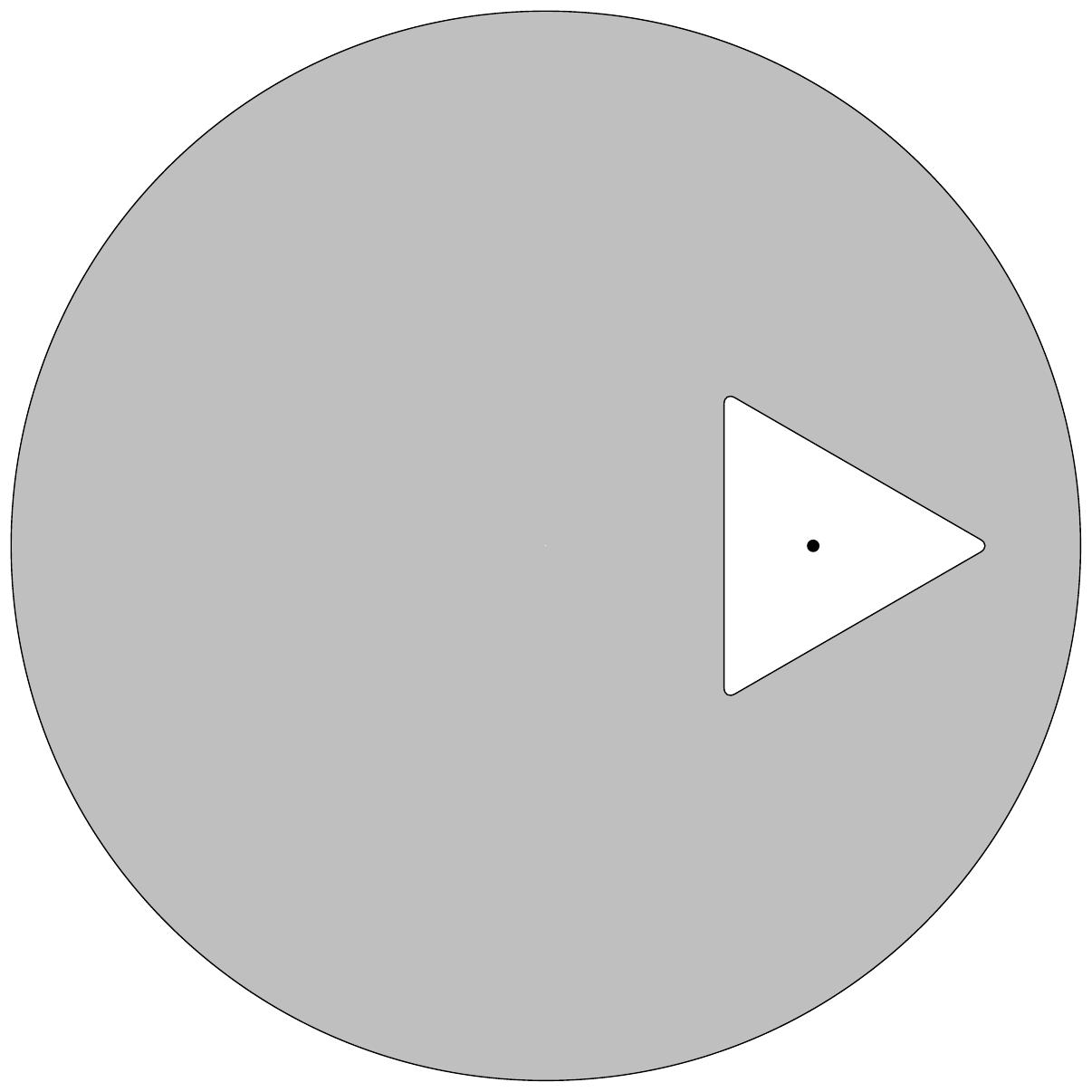}}
\hspace{12mm}
\subfloat[ON configuration]{
\label{On position_triangle}
\includegraphics[width=0.18\textwidth]{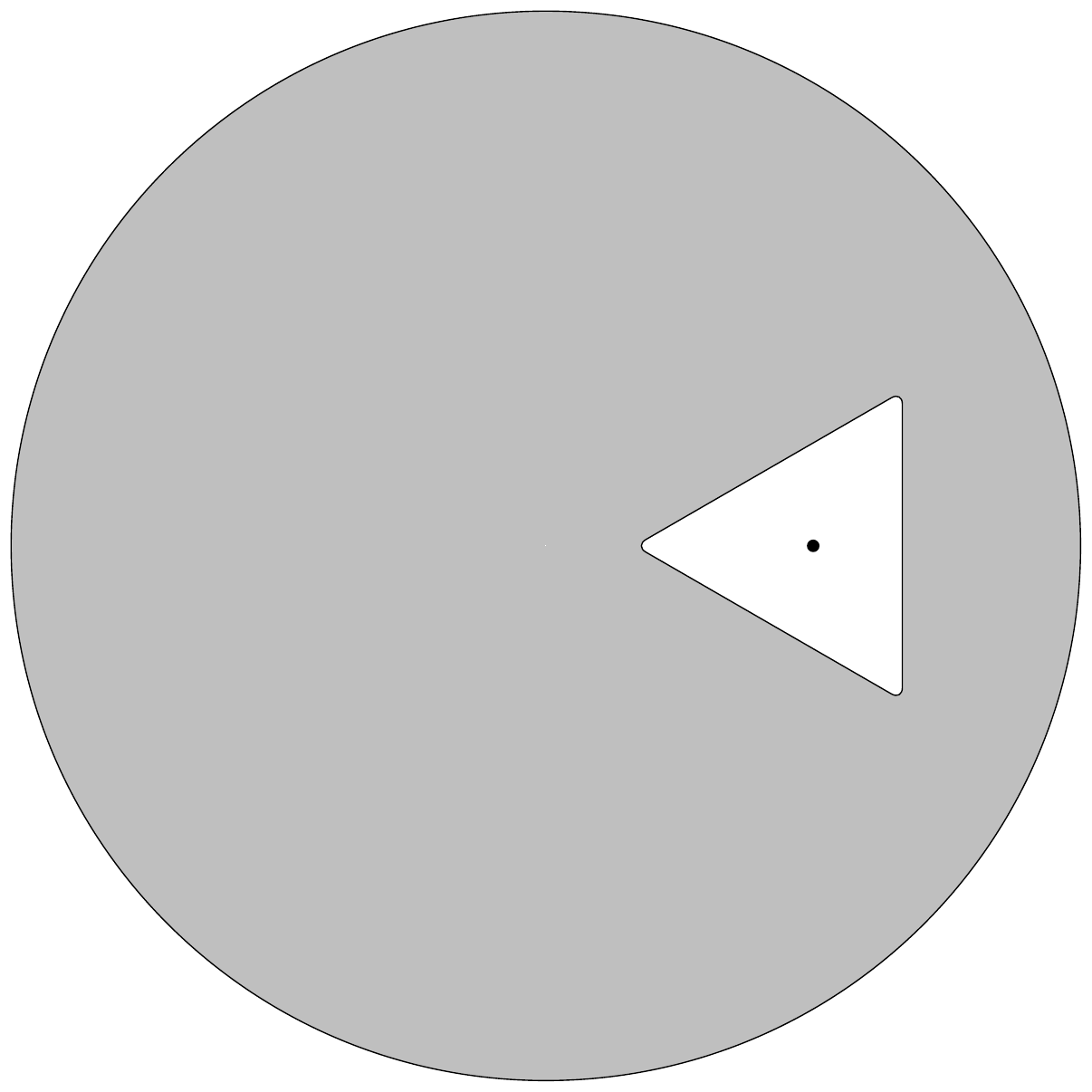}}
\hspace{14mm}
\subfloat[OFF configuration]{
\label{Off position_pentagon}
\includegraphics[width=0.18\textwidth]{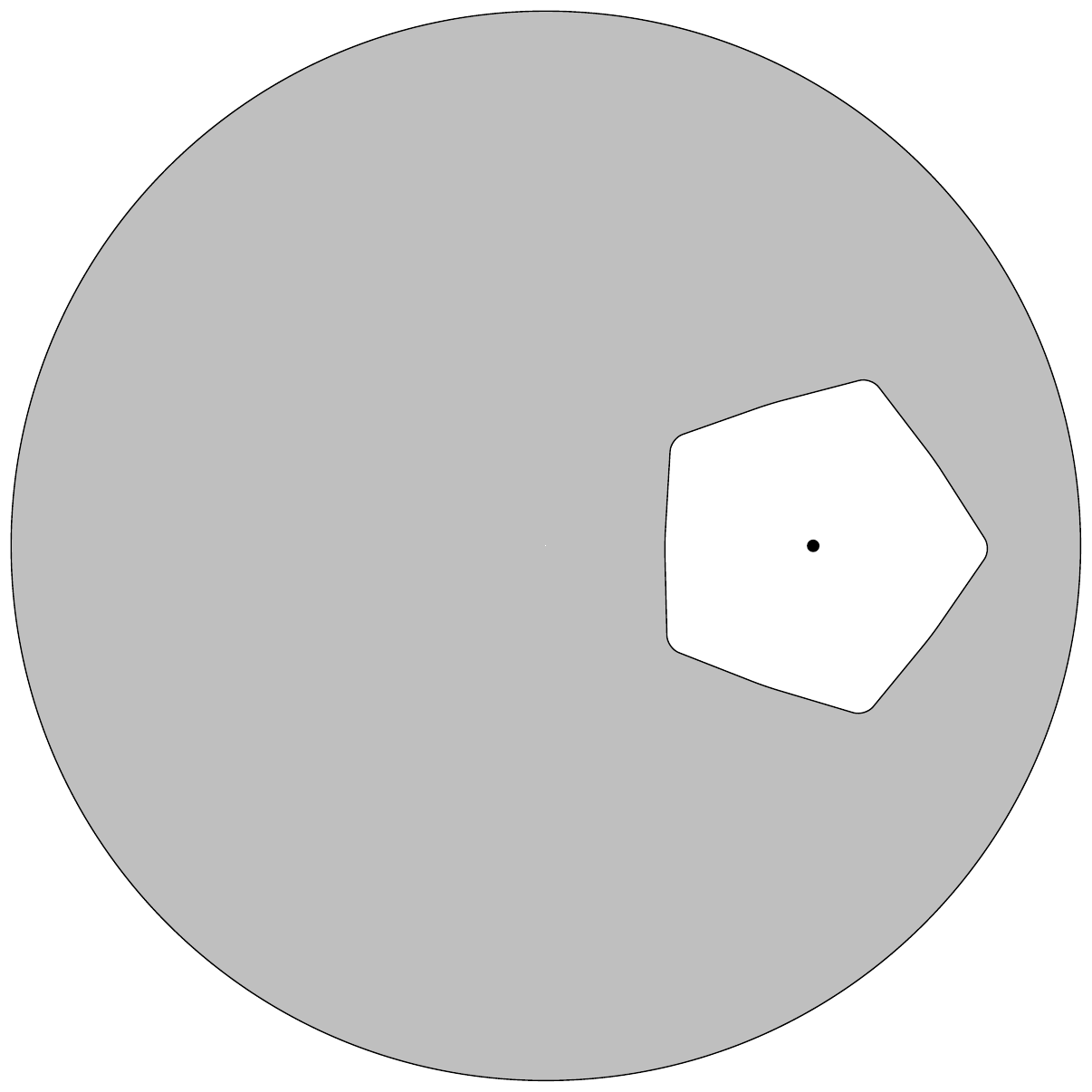}}
\hspace{12mm}
\subfloat[ON configuration]{
\label{On position_pentagon}
\includegraphics[width=0.18\textwidth]{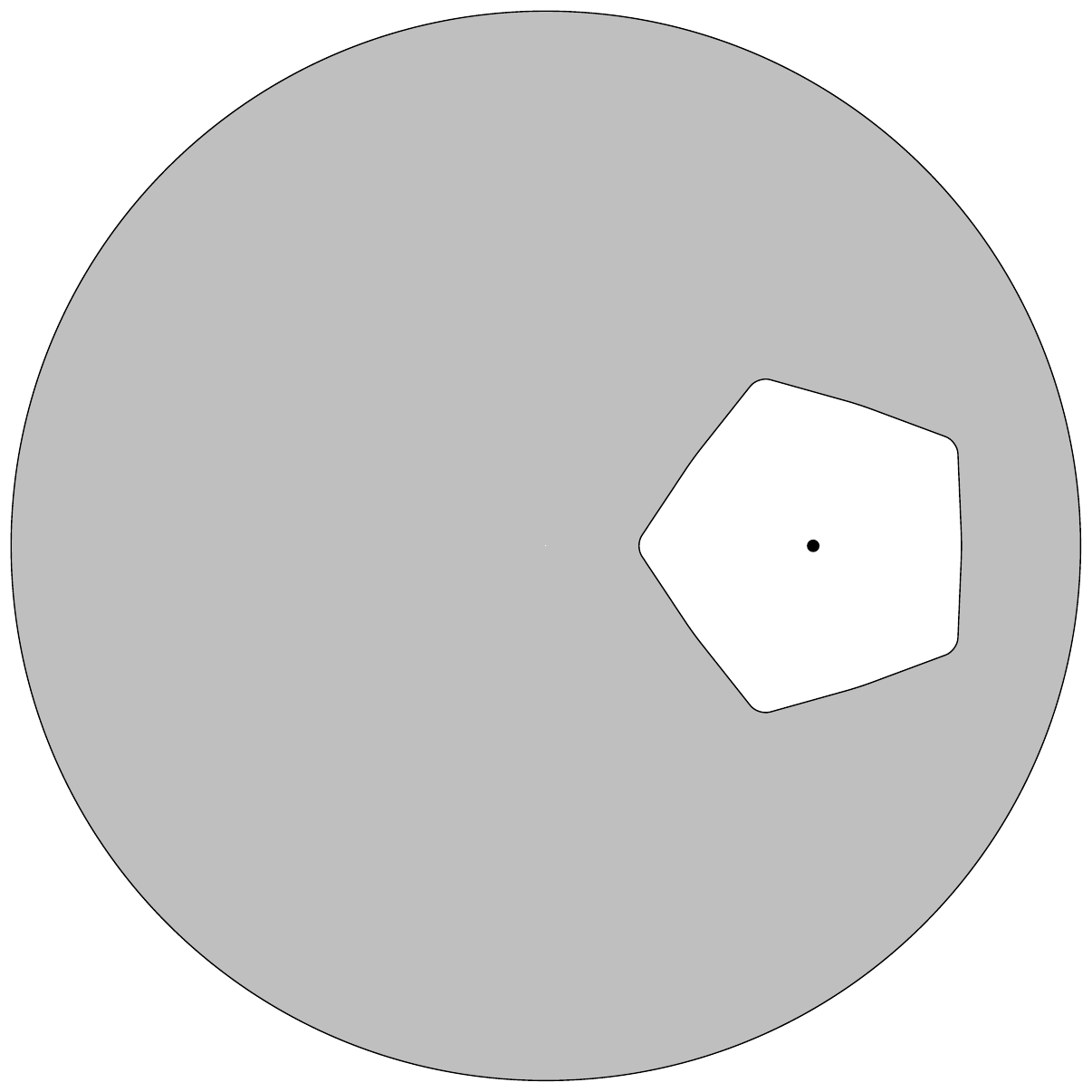}}
\caption{OFF and ON configurations for obstacles having $\mathbb{D}_n$ symmetry, $n$ odd}\label{fig:on_off_odd}
\end{figure}

%\end{enumerate}
\section{Auxiliary results}\label{Prel}
The lemmas proved in this section, viz., Lemmas \ref{bound_monot} and \ref{properties_eta}, are useful {in proving Propositions \ref{critical_points} and \ref{complete_critical_points}, and hence,} in proving our main theorem, viz., Theorem \ref{max_min}. %{\blue This section does not depend on the choice}
\subsection{Certain monotonicity property on the boundary of a disk $B$}
In Lemma \ref{bound_monot}, we prove a monotonicity property on the boundary of an arbitrary disk $B$ using the representation of $B$ in polar coordinates with respect to a point other than its center.
\begin{lemma}\label{bound_monot} Let $B((-x_0,0), r_1)$ be a disk in $\mathbb{R}^2$ with center at $(-x_0, 0)$ and radius $r_1>0$ such that $0< x^0<r_1$. Let $\lbrace{re^{i\phi}: \phi\in[0,2\pi), 0\leq r < g(\phi)\rbrace}$ be a  representation $B$ in polar co-ordinates, where $g: [0,2 \pi] \rightarrow [0, \infty)$ is a $\mathcal{C}^2$ map with $g(0)= g(2 \pi)$. Here, the polar coordinates $(r,\phi)$ are measured with respect to the origin $(0,0)$ and the positive $x_1$-axis of $\mathbb{R}^2$. %which belongs to $B$. 
Then, the distance $d(\phi)$ of a point $g(\phi) \, e^{i \phi}$ on $\partial B$ from $(0,0)$ is a strictly increasing function of $\phi$ in $[0,\pi]$, and is a strictly decreasing function of $\phi$ in $[\pi,2\pi]$. \end{lemma}
%{ May be draw a figure with $n$ odd too.} {\blue Done.}
\begin{proof}
\begin{figure}[H]\centering
\includegraphics[width=0.6\textwidth]{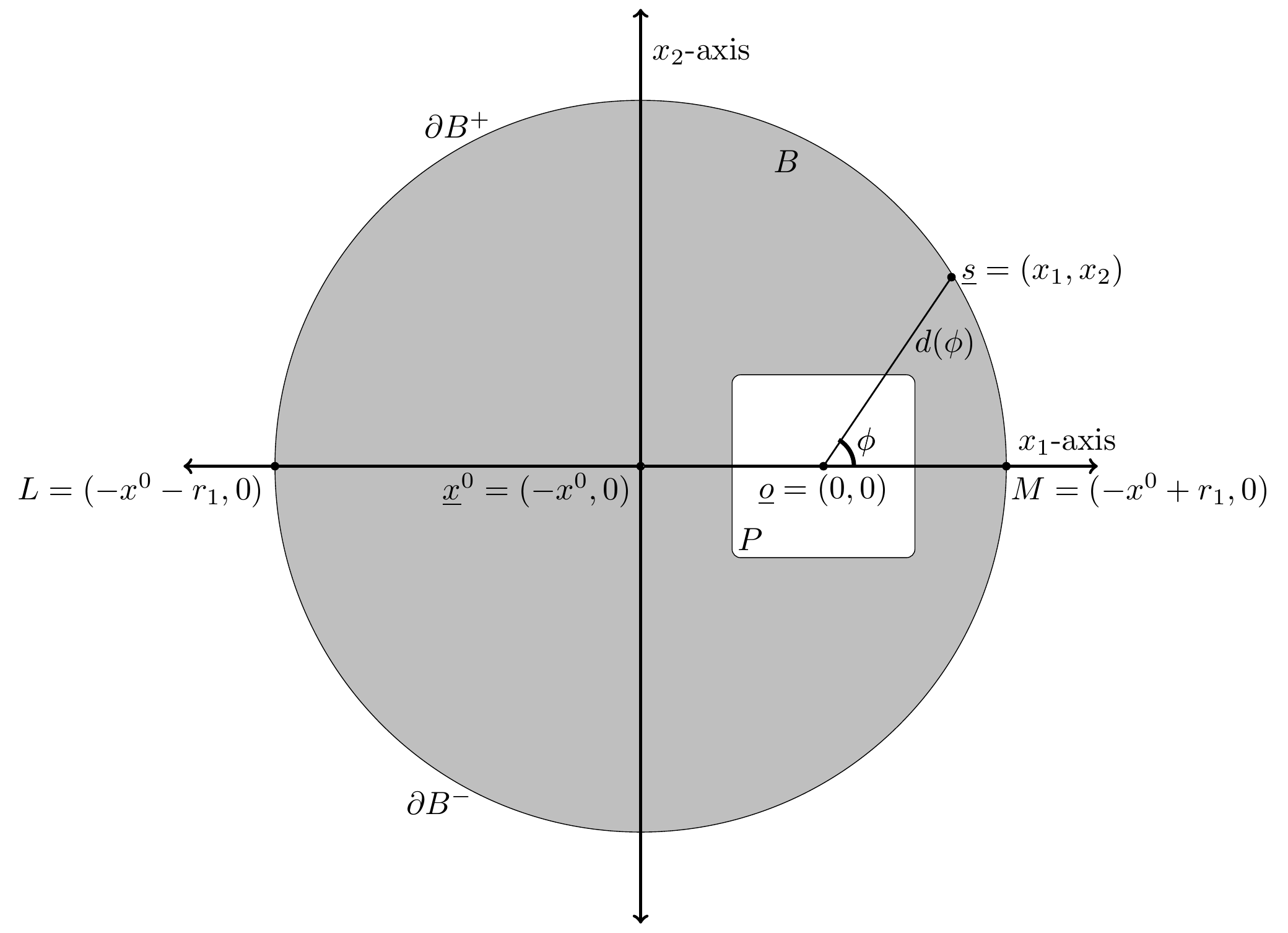}   \\
\includegraphics[width=0.6\textwidth]{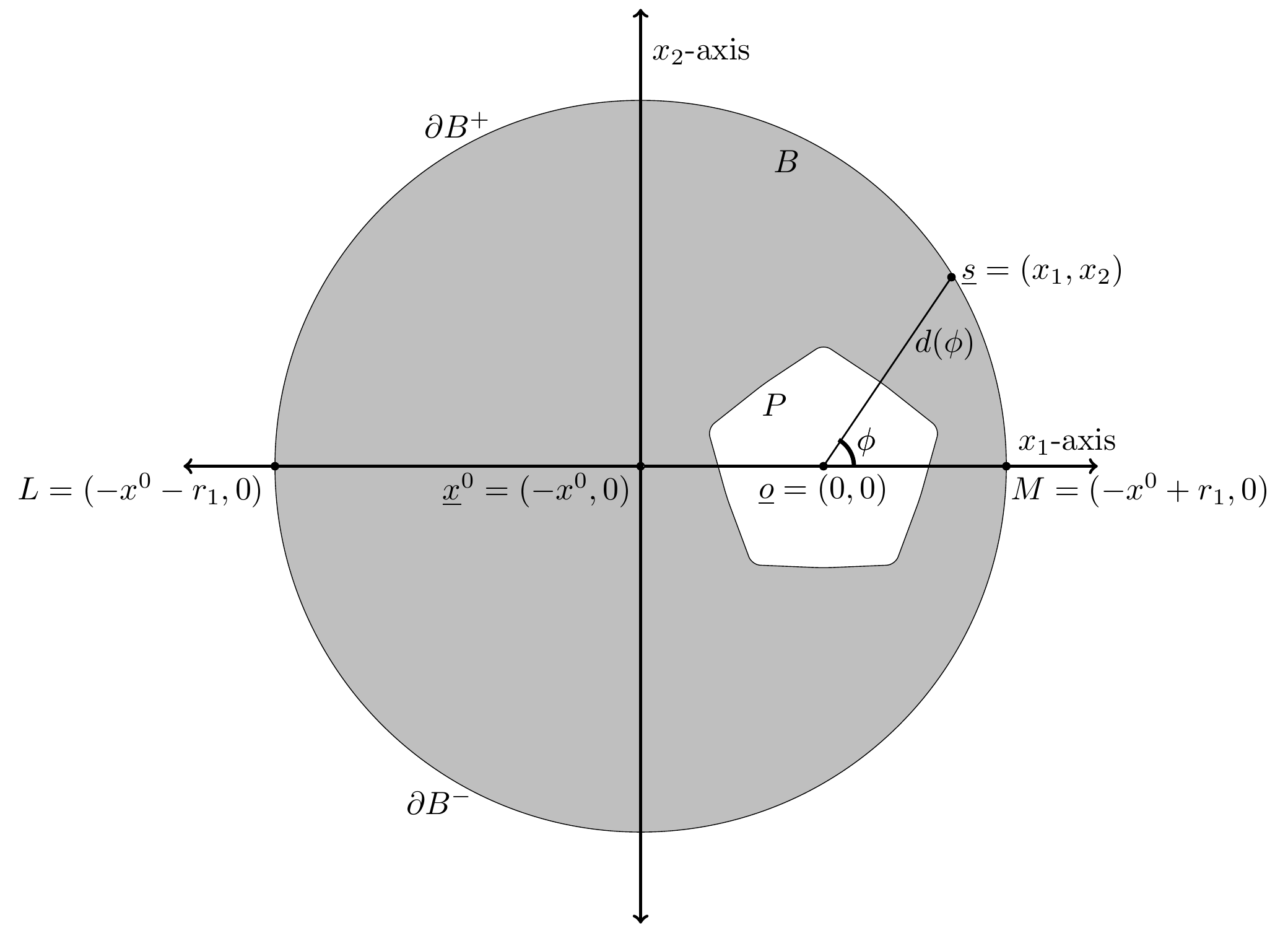}    
    
\caption{The distance function $d$ for the boundary points}\label{fig:geom}
\end{figure}

 Let $\partial{B}^+$ be defined as $\{ g(\phi)\, e^{i \phi} \in \partial B \, | \, \phi \in [0,\pi]\} \subset \partial B$. Similarly, we define $\partial{B}^-$ as the set $ \{ g(\phi)\, e^{i \phi} \in \partial B \, | \, \phi \in [\pi, 2 \pi) \}$. We will prove that $d(\phi)$ is a strictly increasing function of $\phi$ in $[0,\pi]$. The proof for $\phi \in [\pi, 2\pi]$ is similar. 

Let $(x_1,x_2)$ denote the Cartesian coordinate of a point $g(\phi) \, e^{i \phi} \in \partial B^+$ as shown in Figure \ref{fig:geom}. Then, $x_2 \geq 0$ and $(x_1+x^0)^2 + x_2^2 =r_1^2$. We will first show that the Euclidean norm of the point $(x_1,x_2) \in \partial B^+$, is a monotonic function of $x_1$ for all $(x_1,x_2) \in \partial B^+$. Here $x_1 \in [ -x^0-r_1, -x^0+ r_1]$.
We thus consider 
%\begin{equation}
%\begin{aligned}&
$\|(x_1, x_2)\|= d((x_1,x_2), (0,0))$ subject to $(x_1+x^0)^2+x_2^2=r_1^2$.
%\end{aligned}
%\end{equation}
Now, $\|(x_1, x_2)\|=(x_1^2 + x_2^2)^{\frac{1}{2}} = (x_1^2 + r_1^2- (x_1+x^0)^2)^{\frac{1}{2}}= ( r_1^2- 2\, x_1 \, x^0 - (x^0)^2)^{\frac{1}{2}}=: h(x_1) >0$. Therefore, $h'(x_1) =\dfrac{-x^0}{h(x_1)} <0$ for $(x_1,x_2) \in \partial B^+$. Hence, $h$ is a strictly decreasing function of $x_1$ for $(x_1,x_2) \in \partial B^+$. We also note that $h(x_1) = \|(x_1, x_2)\|=|g (\phi)|=d(\phi)$ for $(x_1, x_2) = g(\phi) e^{i \phi} \in \partial B^+$, $ \phi \in [0,\pi]$.  

Next we show that $x_1 = x_1(\phi)$ is a monotonic decreasing function of $\phi$. We have $x_1 = \|(x_1, x_2)\| \, \cos \phi = h(x_1) \, \cos \phi$.
%We can see from Fig. \ref{fig:geom}, $x_1 =$ 
Hence, %\begin{equation}\label{cosex}
$\cos(\phi) = \frac{x_1}{h(x_1)}$.
%\end{equation}
%$ since $y^2=1-(x+s)^2$ on the boundary of the disk.
Consider $\phi : (-x^0-r_1,-x^0+r_1)\rightarrow (0, \pi)$. Then, $$\dfrac{d \phi}{d x_1} = -\dfrac{h(x_1)^2+ x^0 \, x_1}{h(x_1)^3} \, \dfrac{1}{\sin \phi} = -\dfrac{h(x_1)^2+ x^0\, x_1}{x_2 \, h(x_1)^2 } = -\dfrac{r_1^2- x^0\, x_1-(x^0)^2}{x_2 \, h(x_1)^2}= -\dfrac{r_1^2- x^0\,( x_1+x^0)}{x_2 \, h(x_1)^2}.$$ 
Since $|x_1+x^0| < r_1$ and $0<x^0<r_1$ we get, $-r_1^2 < x^0 (x^0 +x_1) < r_1^2$. This implies that $\dfrac{d \phi}{d x_1} <0 $ on  $(-x^0-r_1,-x^0+r_1)$. % for $x_1>0$, $\phi \in (0, \pi)$.
% and 
%\phi = \arccos \left( \frac{x_1}{f(x_1)} \right)
%$$
%as a function of $x$. Differentiating with respect to $x$ in $(-s-1,-s+1)$, we get
%$$
%\phi'(x)  = \frac{ -1+s(s+x)}{(1 -s^2 - 2sx)^{3/2}}\cdot \sqrt{\frac{(1-s-x)(1+s+x)}{1 - s^2 - 2 s x}}< 0,~ \mbox{ since } |x+s|<1, |s|<1.$$ 
Thus, $\phi$ as a function of $x_1$ is strictly decreasing and hence injective on $ (-x^0-r_1,-x^0+r_1)$. 

Finally, we show that $\phi : (-x^0-r_1,-x^0+r_1) \rightarrow (0, \pi)$ is surjective. Let $\theta \in (0, \pi)$, define $x_1 = g(\theta) \cos \theta \in  (-x^0-r_1,-x^0+r_1)$, by the definition of $g$. 

%Suppose there exists a $\theta \in (0,\pi)$ such that there is no $x_1 \in (-x_0-r_1,-x_0+r_1)$ such that $\phi(x_1) = \theta$. % That is, in that case $|x_1+x_0|\geq r_1$. 
 %In this case, $0 < -x_0+r_1 \leq x_1 $ or $x_1 \leq -x_0 -r_1 <0$. That is, either $x_1 + x_0 \geq r_1 >0$ or $x_1+x_0 \leq -r_1 <0$. That is, either $x_1 > 0 > -x_0$ or $x_1 < -x_0 <0$. This gives, either $\frac{x_1}{-x_0} >1$ or $\frac{x_1}{-x_0} < -1$. %Suppose quaring (\ref{cosex}), we get,
%$$
%\cos^2(\theta) = \frac{x_1^2}{f(x_1)^2} = \frac{x_1^2}{r_1^2 -2 x_0 x_1 -x_0^2}= \frac{x_1^2}{r_1^2 - x_0 ( x_1 +x_0) -x_0 x_1}>  \frac{x_1^2}{ - x_0 x_1} = \frac{x_1}{ - x_0 }. 
%$$
%If for such a value of $\phi$, we have $|x+s| >1$, then
%$$\cos^2(\phi) = \frac{x^2}{ x^2 + 1-(x+s)^2 } >  \frac{x^2}{x^2 }=1,$$ since $1-(x+s)^2 < 0$. This gives $ \cos^2(\phi)  > 1$, which is a contradiction. Thus, $|x+s| < 1$ and $x$ can be explicitly found by solving a quadratic equation.
Hence, $\phi : (-x^0-r_1,-x^0+r_1) \rightarrow (0, \pi)$ is a bijective and strictly decreasing function of $x_1$. Since the distance function $d(\phi)$ is decreasing with respect to $x_1$, it is increasing with respect to $\phi$. This proves the lemma. \end{proof}
\subsection{About a planar simply connected bounded domain $K$ }\label{K}
In this section, we consider a planar simply connected bounded domain $K$ and represent it in polar co-ordinates with respect to the origin of $\mathbb{R}^2$. We consider the unit outward normal vector field to $K$ on its boundary $\partial K$. Call this vector field $\eta$. We derive an expression for $\eta$ in the polar co-ordinates. We then consider a smooth vector field $v$ in $\mathbb{R}^2$ that rotates the domain $K$ by a right angle about the origin in the anticlockwise direction. We then derive an expression, in polar coordinates, for the inner product of these two vector fields evaluated at a boundary point. All these expressions are put together in the form of Lemma \ref{properties_eta}.
 
Now, in polar co-ordinates, the planar simply connected  bounded domain $K$ can be given by
%
%Let $K$ denote a planar, bounded and simply connected domain defined in polar coordinates by
 $
K = \lbrace{re^{i\phi} %\in \mathbb{C}
: \phi\in[0,2\pi), 0\leq r < h(\phi)\rbrace} \subset \mathbb{R}^2,
$
where $h$ is a positive, bounded and $2\pi$-periodic function of class $\mathcal{C}^2$. Let $v \in \mathcal{C}_0^\infty(\mathbb{R}^2)$ be a smooth vector field whose restriction to $\partial K$ is given by 
$v(x_1,x_2) = (-x_2,x_1) \; \forall (x_1,x_2) \in \partial K.
$
This implies 
$v \left( h(\phi) \left(\cos \phi, \sin \phi \right) \right)= h(\phi) \left(-\sin \phi, \cos \phi \right) \; \forall \phi \in [0, 2 \pi).
$
Treating $\mathbb{R}^2$ as the complex plane $\mathbb{C}$, one can write $v$ as 
$v(\zeta) = \textbf{i}\zeta \;\forall \zeta = h(\phi)e^{\textbf{i}\phi} \in \partial K, 
$
which is equivalent to saying that
$v(\phi):= v\left( h(\phi)e^{\textbf{i}\phi}\right)= \textbf{i} h(\phi) \, e^{\textbf{i} \phi} \; \forall \phi \in \mathbb{R}.
$

Denote by $\eta$ the unit outward normal vector field to $K$ on $\partial K$. For $\alpha \in [0, 2 \pi]$, let $z_\alpha := \{ r e^{i \alpha} \,| \, r \in \mathbb{R}\}$ denote the line in $\mathbb{R}^2$ corresponding to angle $\phi=\alpha$ represented in polar co-ordinates. Clearly, $z_\alpha = z_{\alpha+ \pi}$ for each $\alpha \in [0, 2\pi]$ where the addition is taken modulo $2\pi$. 

We now prove the following auxiliary lemma.
\begin{lemma}\label{properties_eta}
Let $K, h, v, \eta$ and $z_\alpha$ be as defined above. Then at any point $ h(\phi)e^{\textbf{i}\phi}$ of $\partial K$%where $\eta$ is defined
, we have the following:
\begin{enumerate}[i)]
\item $\eta(\phi) := \eta(h(\phi)e^{\textbf{i}\phi}) = \dfrac{h(\phi)e^{\textbf{i}\phi}-\textbf{i}h^\prime(\phi)e^{\textbf{i}\phi}}{\sqrt{h^2(\phi)+(h^\prime(\phi))^2}}~\; \forall \phi \in \mathbb{R}$,
\item $\left<\eta , v\right>(\phi) := \left<\eta , v\right> (h(\phi)e^{\textbf{i}\phi})  = \dfrac{-h(\phi)h^\prime(\phi)}{\sqrt{h^2(\phi)+(h^\prime(\phi))^2}}~ \; \forall \phi \in \mathbb{R}$. Hence $\left<\eta , v\right>$ has a constant sign on an interval $I \subset \mathbb{R}$ iff $h$ is monotonic in $I$.
\item If for some $\alpha \in [0, 2 \pi)$, the domain $K$ is symmetric with respect to the axis $z_{\alpha}$ then, for each $\theta \in [0, \pi]$,
$
\left<\eta , v\right>(\alpha+\theta) = -\left<\eta , v\right>(\alpha-\theta).
$
\end{enumerate}
\end{lemma}
\begin{proof}
\begin{enumerate}[i)]
\item  Let $\gamma: [0, 2 \pi) \rightarrow \mathbb{R}^2$ be defined as $\gamma(\phi)=h(\phi)e^{\textbf{i}\phi}$. That is, $\gamma$ is a parametrization of the boundary curve $\partial K$. Then, the tangent vector field to the boundary $\partial K$ %:=\{ \gamma(\phi):=h(\phi)e^{\textbf{i}\phi} \, |\, \phi \in [0, 2 \pi) %\mathbb{R}\} $ 
is given by 
$\gamma^\prime(\phi)=%(-h(\phi) \sin \phi + h^\prime(\phi) \cos \phi, h(\phi) \cos \phi + h^\prime(\phi) \sin \phi)= 
\left( h^\prime(\phi)%e^{\textbf{i}\phi}
+ \textbf{i}h(\phi) \right) e^{\textbf{i}\phi}.$
Thus, the outward unit normal to $K$ at a point $\gamma(\phi) \in \partial K$ is given by $$\eta(\phi)=\dfrac{ \left( h(\phi)%e^{\textbf{i}\phi}
-\textbf{i}h^\prime(\phi) \right) e^{\textbf{i}\phi}}{\sqrt{h^2(\phi)+(h^\prime(\phi))^2}}.$$
\item Therefore,
%Let $v(\phi) = v(h(\phi)e^{\textbf{i}\phi}).$
$$
\left<\eta, v\right> (\phi) = \dfrac{h^2(\phi)\left<e^{\textbf{i}\phi}, \textbf{i}e^{\textbf{i}\phi}\right>-h(\phi)h^\prime(\phi)|\textbf{i}e^{\textbf{i}\phi}|^2}{\sqrt{h^2(\phi)+(h^\prime(\phi))^2}}=- \dfrac{h(\phi)h^\prime(\phi)}{\sqrt{h^2(\phi)+(h^\prime(\phi))^2}}.
$$
\item Since $K$ is symmetric with respect to the axis $z_\alpha$, the function $h$ satisfies 
$
h(\alpha+\theta) = h(\alpha-\theta)$ %~~~  \mbox{
 for each %} ~~
$\theta \in [0, \pi].
$ Moreover, $h^\prime(\alpha -\theta) = - h^\prime (\alpha +\theta)$ %~~~  \mbox{
 for each %} ~~
$\theta \in [0, \pi] .$
Using (ii), we then have
$
\left<\eta , v\right>(\alpha+\theta) = -\left<\eta , v\right>(\alpha-\theta).
$
\end{enumerate}
\end{proof}
\begin{remark}
%\item[(a)] 
We note here that since $h %\mathbb{R} \rightarrow (0, \infty)
$ is a $2 \pi$-periodic function on $\mathbb{R}$, so are the functions $v, \eta$ and $\left<v, \eta \right>$. %: \mathbb{R} \rightarrow$
%\item[(b)] In the rest of this article, we are interested in the outward unit normal to the domain $\Omega_t$ at points on the boundary $\partial P_t := \{f(\phi) e^{i \phi} \, | \, \phi \in [0, 2 \pi]\}$ of the puncture $P_t$. Therefore, the outward unit normal with respect to the punctured domain $\Omega_t$ at a point $f(\phi) e^{i \phi}$ on $\partial P_t$ will be the negative of the vector field $\eta(f(\phi) e^{i \phi})$ described above. 
\end{remark}

\section{The main theorem} \label{stmnt}%\label{sec:initconfig}
%In this section, we state our main theorem, viz., Theorem \ref{max_min}. 
%\subsection{The Statement}\label{stmnt}
We recall here that $P$ is a compact simply connected subset of $\mathbb{R}^2$ satisfying assumptions \ref{assumption_A}, \ref{assumption_B} and that $B$ is an open disk in $\mathbb{R}^2$ of radius $r_1$ such that $B \supset \overline{co(C_2(P))}$. For $t\in\mathbb{R}$, let $\rho_t \in SO(2)$ denote the rotation in $\mathbb{R}^2$ about the origin $\underline{o}$ in the anticlockwise direction by an angle $t$, i.e., for $\zeta \in \mathbb{C} \cong \mathbb{R}^2  $, we have $\rho_t \zeta :=e^{\textbf{i} t} \zeta$. Now fix $t \in [0, 2\pi)$. %Let $P_t:=\rho_t(P)$. Then, in polar co-ordinates,
Let $\Omega_t:= B \setminus \rho_t(P)$ and $\mathcal{F}:= \{\Omega_t \, |\, t \in [0,2 \pi)\}$. 

%Let $\lambda_1(t)$ denote the fundamental Dirichlet eigenvalue of the Laplacian on $\Omega_t$ i.e., $\lambda_1(t):= \lambda_1(\Omega_t)$. Let $y_1(t)$ denote the unique positive unit norm principal Dirichlet eigenfunction for the Laplacian on $\Omega_t$, i.e., $y_1(t)$ is the eigenfunction corresponding to $\lambda_1(t)$ on $\Omega_t$ satisfying 
%\begin{equation}\label{laplace_equation}
%\begin{aligned}
%-\Delta u &=\lambda_1(t) \, u & \mbox{ in } & \Omega_t,\\
%&u =0 & \mbox{ on } & \partial \Omega_t,\\
%\int_{\Omega_t} &u^2(x) \, dx =1, &  & \\
%u & > 0 & \mbox{ in } & \Omega_t.
%\end{aligned}
%\end{equation}
% and B satisfies (1) and (2). 

We now state our main theorem {for $n$ even, $n \geq 3$}:
\begin{theorem}[Extremal configurations]\label{max_min}
%Let $P \subset \mathbb{R}^2$ satisfy assumptions (a), (b), (c) and (d). Assume further that $\overline{co(C_2(P))} \subset B$. %$\rho(P) \subset B~ \forall \rho \in SO(2)$
The fundamental Dirichlet eigenvalue $\lambda_1(\Omega_t)$ for $\Omega_t \in \mathcal{F}$ is optimal precisely for those $t \in [0,2\pi)$ for which an axis of symmetry of $P_t$ coincides with a diameter of $B$. Among these optimal configurations, the maximizing configurations are the ones corresponding to those $t \in [0,2\pi)$ for which $P_t$ is in an ON position with respect to $B$; and the minimizing configurations are the ones corresponding to those $t \in [0, 2 \pi)$ for which $P_t$ is in an OFF position with respect to $B$.
\end{theorem}

{Equation (\ref{even_lambda}), Propositions \ref{critical_points} and \ref{complete_critical_points} imply Theorem \ref{max_min} for $n$ even, $n \geq 3$. For the $n$ odd case, we identify some of the extremal configuration for $\lambda_1$. We prove that equation (\ref{even_lambda}) and Proposition \ref{critical_points} hold true for $n$ odd too. %We highlight some of the difficulties faced in proving Proposition \ref{complete_critical_points} for this case. 
We provide numerical evidence for $n=5$ and conjecture that Proposition \ref{complete_critical_points}, and hence, Theorem \ref{max_min} hold true for $n$ odd too.}

\section{Proof of the main theorem}\label{proof}
In this section, we prove our main theorem, viz., Theorem \ref{max_min} {for $n \geq 3$, $n$ even. We prove that equation (\ref{even_lambda}) and Proposition \ref{critical_points} hold true for any $n\geq 3$, even or odd.}

We first justify that, {for any $n \geq 3$, even or odd}, the fundamental Dirichlet eigenvalue $\lambda_1$ of the Laplacian for the family of domains under consideration is a function of just one real variable, and that it is an even periodic function of period $2 \pi /n$. Therefore, in order to determine the extremal configuration/s for $\lambda_1$ we study its behavior on the interval $[0, \frac{\pi}{n}]$. The Hadamard perturbation formula (\ref{Hadamard}) becomes useful in this analysis. We identify some of critical points of $\lambda_1$ in Proposition \ref{critical_points} {for $n\geq 3$, even or odd.}  

We prove Proposition \ref{complete_critical_points} {for $n$ even, $n \geq3$}. In view of equation (\ref{even_lambda}) Propositions \ref{critical_points} and \ref{complete_critical_points} imply that, {for $n$ even, $n \geq3$}, (a) these are the only critical points for $\lambda_1$, and that, (b) between every pair of consecutive critical points, $\lambda_1$ is a strictly monotonic function of the argument. {We introduce and use a `sector reflection technique' which is similar to the domain reflection technique. We also introduce and use a `rotating plane method' which is similar to the moving plane method.}% For odd $n$, we use a combination of the sector reflection technique and a sector inversion technique followed by a rotating plane method.} { inversion technique doesn't seem to work.}%We use a sector reflection technique and a rotating plane method in the proof. %Theorems \ref{critical_points} and \ref{complete_critical_points} then immediately imply our main theorem, viz., Theorem \ref{max_min}. 

Let $\lambda_1(t)$ denote the fundamental Dirichlet eigenvalue of the Laplacian on $\Omega_t$ i.e., $\lambda_1(t):= \lambda_1(\Omega_t)$. %Let $y_1(t)$ denote the unique positive unit norm principal Dirichlet eigenfunction for the Laplacian on $\Omega_t$, i.e., $y_1(t)$ is the eigenfunction corresponding to $\lambda_1(t)$ on $\Omega_t$ satisfying 
%\begin{equation}\label{laplace_equation}
%\begin{aligned}
%-\Delta u &=\lambda_1(t) \, u & \mbox{ in } & \Omega_t,\\
%&u =0 & \mbox{ on } & \partial \Omega_t,\\
%\int_{\Omega_t} &u^2(x) \, dx =1, &  & \\
%u & > 0 & \mbox{ in } & \Omega_t.
%\end{aligned}
%\end{equation}
Then, by Proposition 3.1 in \cite{anisa_aithal}, the map $t \longmapsto \lambda_1(t)$ is a $\mathcal{C}^1$ map in $\mathbb{R}$ from a neighborhood of $0$ in $\mathbb{R}$. The same can be said about $\lambda_1(t_0 +t)$ for a fixed $t_0 \in \mathbb{R}$. %: \mathbb{R} \rightarrow (0,\infty)$ is a . %and the derivative $\lambda_1^\prime(t)$ of $\lambda_1$ at a point $t \in \mathbb{R}$ is given by the Hadamard perturbation formula (\ref{Hadamard}).
Therefore, to prove Theorem \ref{max_min}, we first need to characterize the critical points of $\lambda_1(t)$. 
\subsection{Sufficient condition for the critical points of $\lambda_1(B \setminus P_t),$ $t \in [0, 2 \pi)$} 
{ Fix $n \geq3$, even or odd. }In this section, we establish a sufficient condition for the critical points of the $\mathcal{C}^1$ function $\lambda_1: \mathbb{R} \rightarrow (0, \infty) $. %In this section, we characterize the critical points of $\lambda_1(\Omega_t)$.

In polar co-ordinates, the open disk $B $ can be represented as the set $\lbrace{re^{i\phi}: \phi\in[0,2\pi), 0\leq r < g(\phi)\rbrace}$, where $g: [0,2 \pi] \rightarrow [0, \infty)$ is a $\mathcal{C}^2$ map with $g(0)= g(2 \pi)$. Here, $(r,\phi)$ is measured with respect to the origin $\underline{o}=(0,0)$ of $\mathbb{R}^2$. The boundary $\partial B$ of $B$, then, is given by $g(\phi)\, e^{i \phi}$, $0 \leq \phi < 2 \pi$. Let $d(\phi)$ denote the Euclidean norm of $g(\phi) \, e^{i \phi}$, that is, $d(\phi)$ is the distance of a point $g(\phi) \, e^{i \phi}$ on $\partial B$ from the center $\underline{o}$ of the obstacle $P$.  %Recall also that $B$ is parametrized in polar coordinates as 
%$$B=\lbrace re^{\textbf{i}\phi}, \phi\in [0,2\pi), 0 \leq r < g(\phi)\rbrace.$$ where $g: [0,2 \pi] \rightarrow [0, \infty)$ is a $\mathcal{C}^2$ map with $g(0)= g(2 \pi)$. 
Then, by Lemma \ref{bound_monot}, $d$ is a strictly increasing function of $\phi$ on $[0,\pi]$.
\subsubsection{The initial configuration}\label{a:init_config}
We start with the following initial configuration $\Omega_{\mbox{init}}$ of a domain $\Omega \in \mathcal{F}$. %\begin{assumption}[Initial configuration]
Let $P$ and $B$ be as described in section \ref{stmnt}. Let $\Omega_{\mbox{init}}$  denote the domain $B \setminus P \in \mathcal{F}$ where $P$ is in an OFF position with respect to $B$. { Recall that we assumed, without loss of generality, that (a) The centers of $B$ and $P$ are on the $x_1$-axis, (b) the center of $P$ is at the origin, and (c) the center of $B$ is on the negative $x_1$-axis. % the diameter of $B$ containing the inradius of $P$ is along the $x_1$-axis, and that 
Let $\underline{x}^0:=(-x^0,0)$ be the center of the disk $B$, where $0< x^0<r_1$. }The initial configurations for obstacles with $\mathbb{D}_4$ symmetry are shown in Figure \ref{fig:in_config}. %{ Added the initial configurations for $n$ odd.} 
%\end{assumption}
\begin{figure}[H]\centering
\subfloat{
\includegraphics[width=0.3\textwidth]{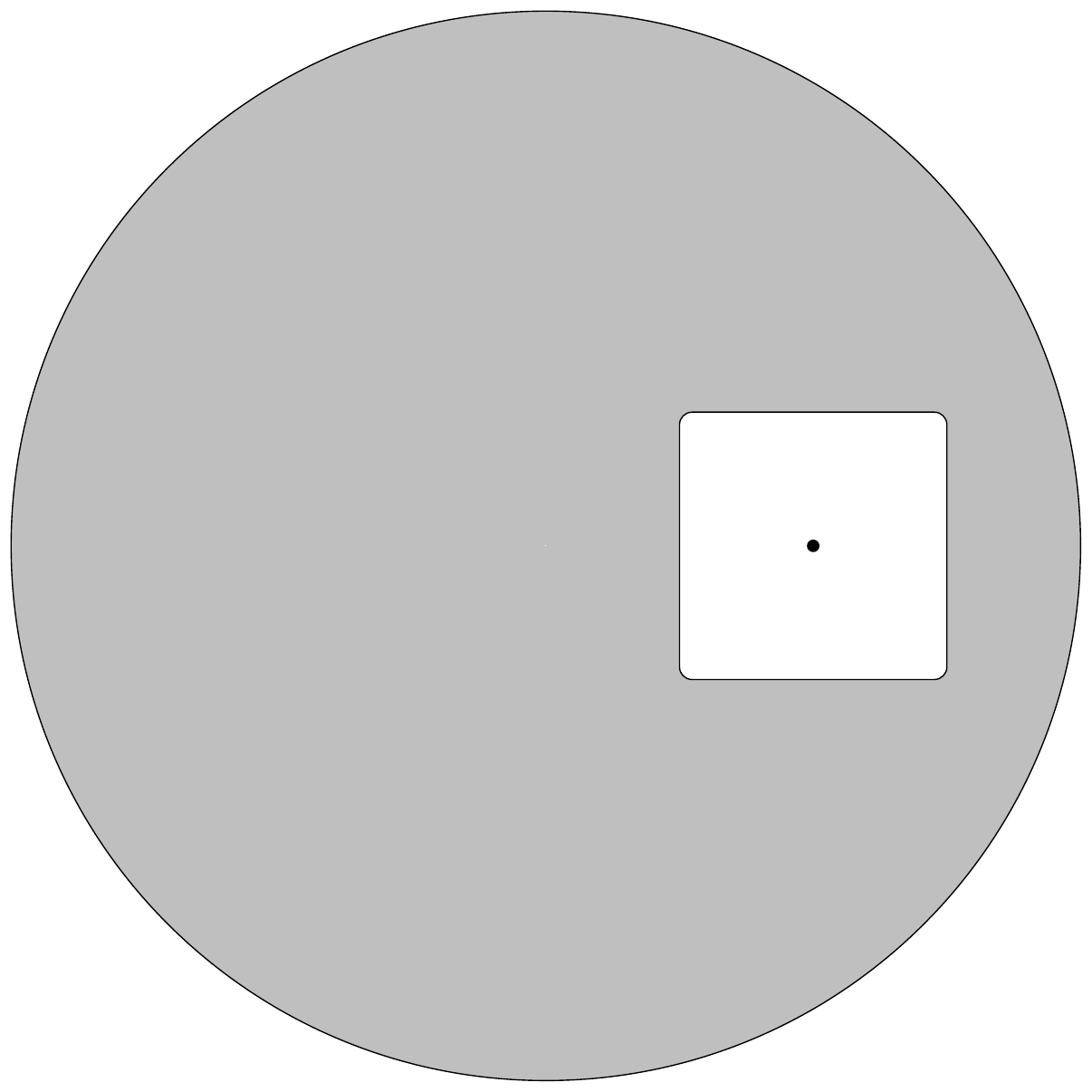}}\hspace{10mm}
\subfloat{
\includegraphics[width=0.3\textwidth]{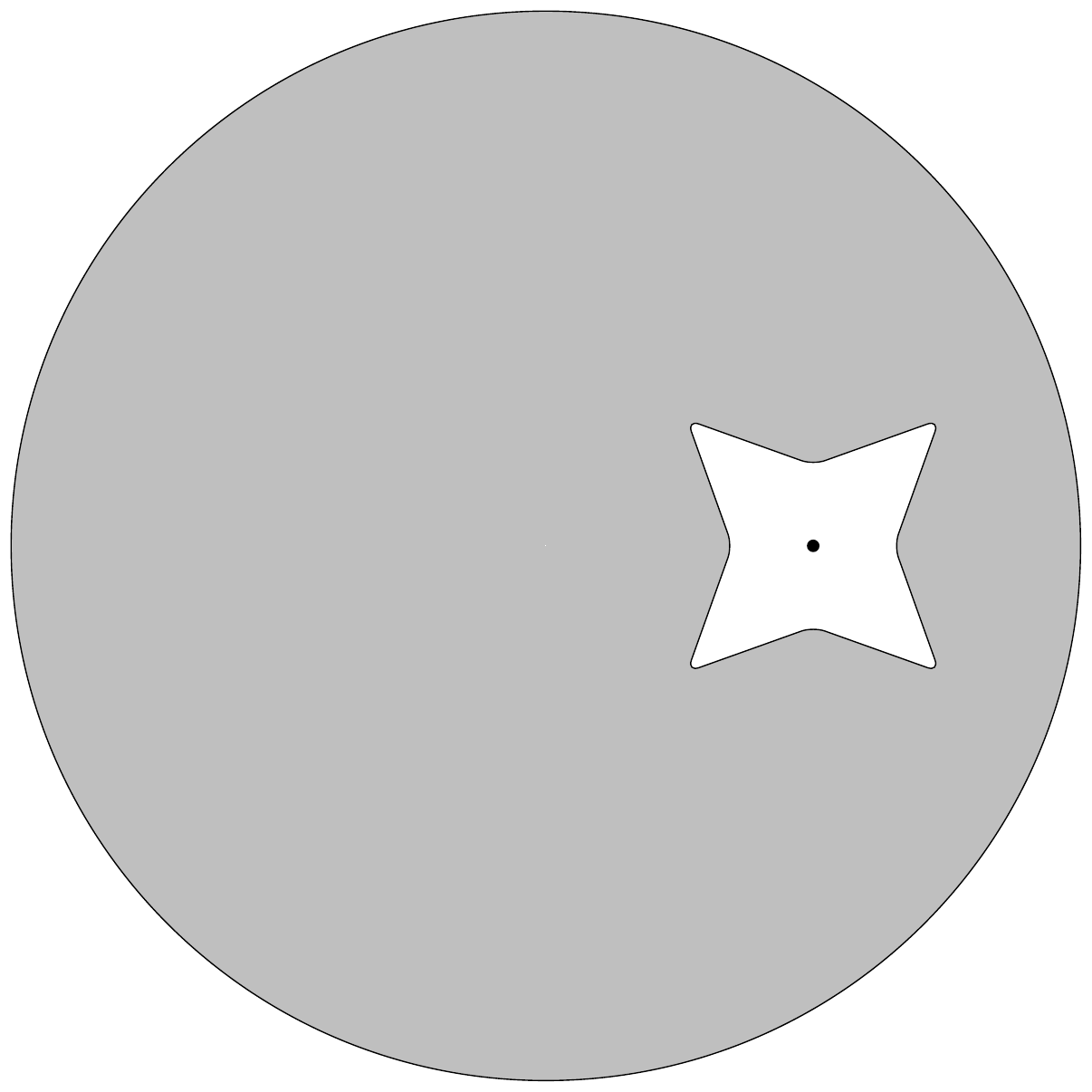}}\\

\subfloat{
\includegraphics[width=0.3\textwidth]{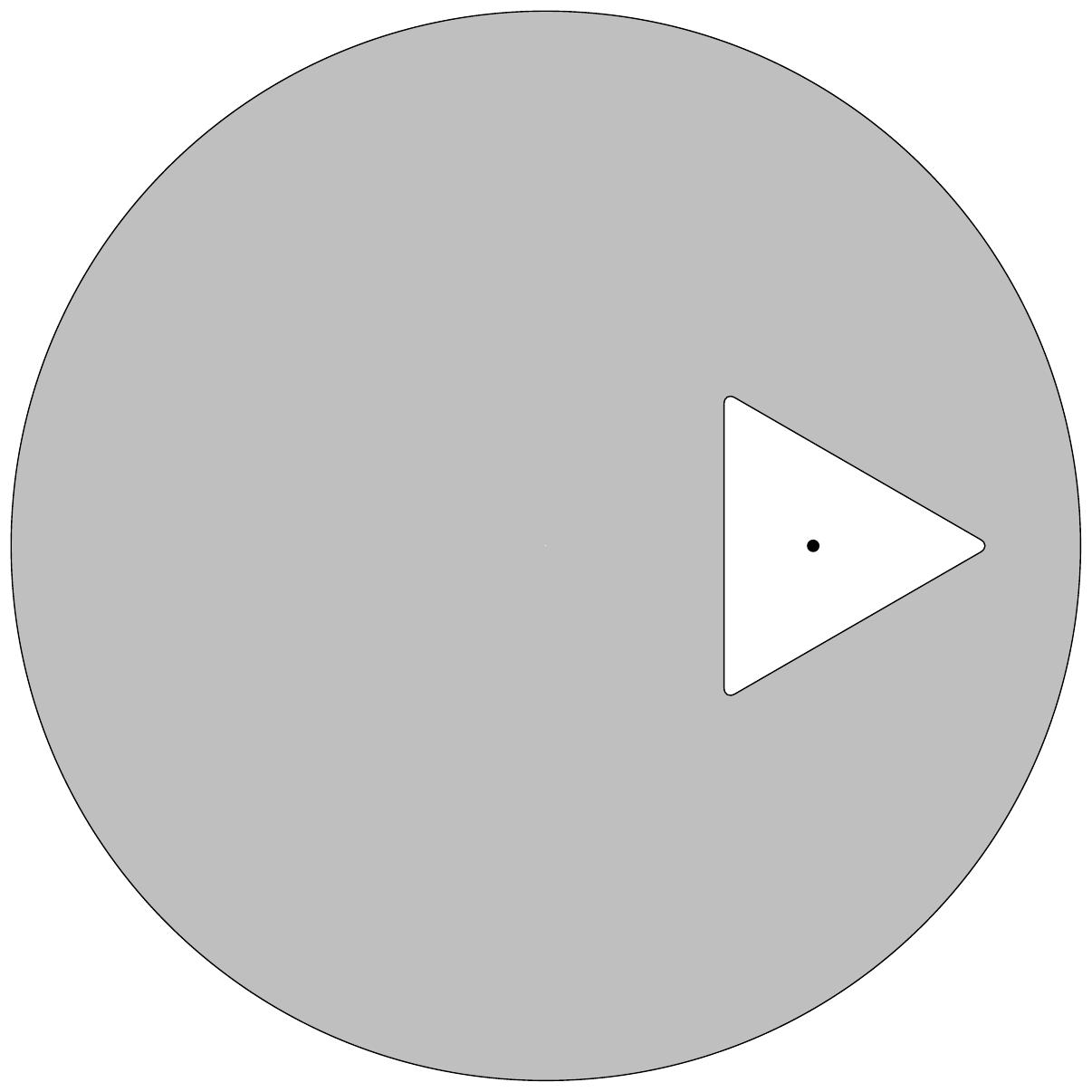}}\hspace{10mm}
\subfloat{
\includegraphics[width=0.3\textwidth]{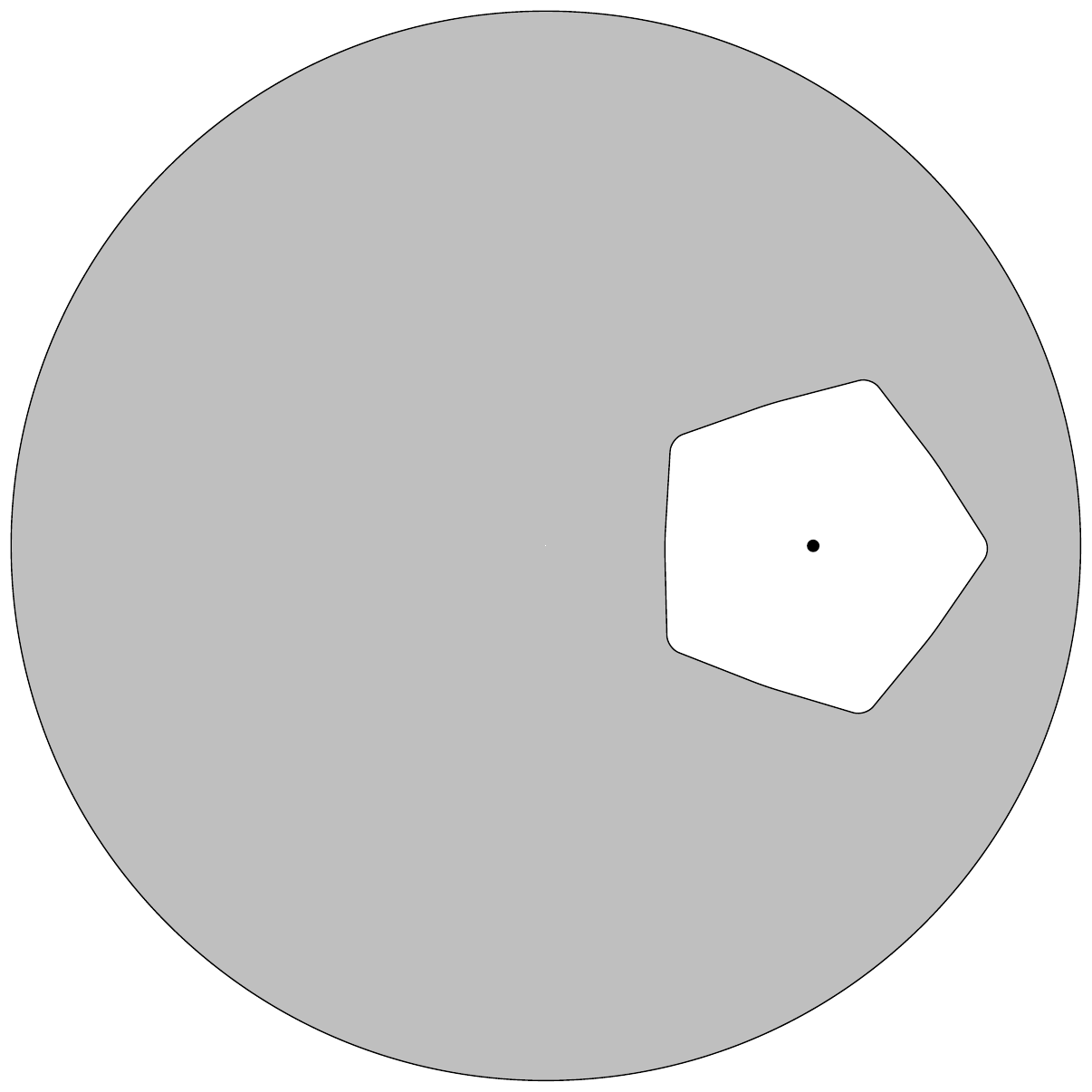}}

\caption{Initial configuration}
\label{fig:in_config}
\end{figure}
%{ Please draw figures for $n$ odd too.}
%Consider the initial configuration $\Omega_{\mbox{init}}= B \setminus P$ of a domain %$\Omega \
%in $\mathcal{F}$, as described in section \ref{a:init_config}

We parametrize $P$ in polar coordinates as follows
\begin{equation}
\begin{aligned}
%&B=\lbrace re^{\textbf{i}\phi}, \phi\in [0,2\pi), 0 \leq r < g(\phi)\rbrace\\
&P=\lbrace re^{\textbf{i}\phi}\, :\, \phi\in [0,2\pi), 0 \leq r < f(\phi)\rbrace,\\
\end{aligned}
\end{equation}
where $f: [0,2 \pi] \rightarrow [0, \infty)$ is a $\mathcal{C}^2$ map with $f(0)=f(2 \pi)$. Because of the initial configuration assumptions on $B \setminus P$, { $f$ %p : [0,2 \pi] \rightarrow (0, \infty)$ is $C^2$ and 
is an increasing function of $\phi$ on $({0,\frac{\pi}{n}})$ for $n$ even, and is a decreasing function of $\phi$ on $({0,\frac{\pi}{n}})$ for $n$ odd. }%; whereas $f$ %p : [0,2 \pi] \rightarrow (0, \infty)$ is $C^2$ and 
%is a decreasing function of $\phi$ on $({0,\frac{\pi}{n}})$ for $n$ odd.} %Recall here that $n \geq 4$ is fixed. Also, $g(0)=g(2 \pi)$ and $f(0)=f(2 \pi)$. We also recall from Lemma \ref{bound_monot} that $g$ is also an increasing function of $\phi$ on $({0,\frac{\pi}{n}})$. 
The condition that the obstacle $P$ can rotate freely around its center $\underline{o}$ inside $B$, i.e. $\rho(P) \subset B~ \forall \rho \in SO(2)$ is guaranteed by assuming that the closure of the convex hull of the circumcircle $C_2(P)$ is contained in $B$. This gives us the following relation:
$$
f\left({\dfrac{\pi}{n}}\right) = \max_{0\leq\phi\leq 2\pi} f(\phi) < \min_{0\leq\phi\leq 2\pi} g(\phi)= g(0).
$$
% Recall the following from section \ref{initconfig} that, for $t\in\mathbb{R}$, $\rho_t \in SO(2, \mathbb{R})$ denotes the rotation in $\mathbb{R}^2$ about the origin $\underline{o}$ in the anticlockwise direction by an angle $t$. That is, for $\zeta\in\mathbb{R}^2 \cong \mathbb{C}$, $\rho_t \zeta :=e^{\textbf{i} t} \zeta$.
\subsubsection{Configuration at time $t$} Now fix $t \in [0, 2\pi)$. We set 
\begin{equation}\label{t_configuration}
P_t := \rho_t(P), \qquad \Omega_t := B \setminus P_t.
\end{equation}
Then, in polar co-ordinates, we have
$\partial P_t : =  \{f(\phi-t) e^{i \phi}\, |\, \phi \in [0, 2\pi)\}.$

\begin{figure}[H]\centering
\subfloat{\includegraphics[width=0.4\textwidth]{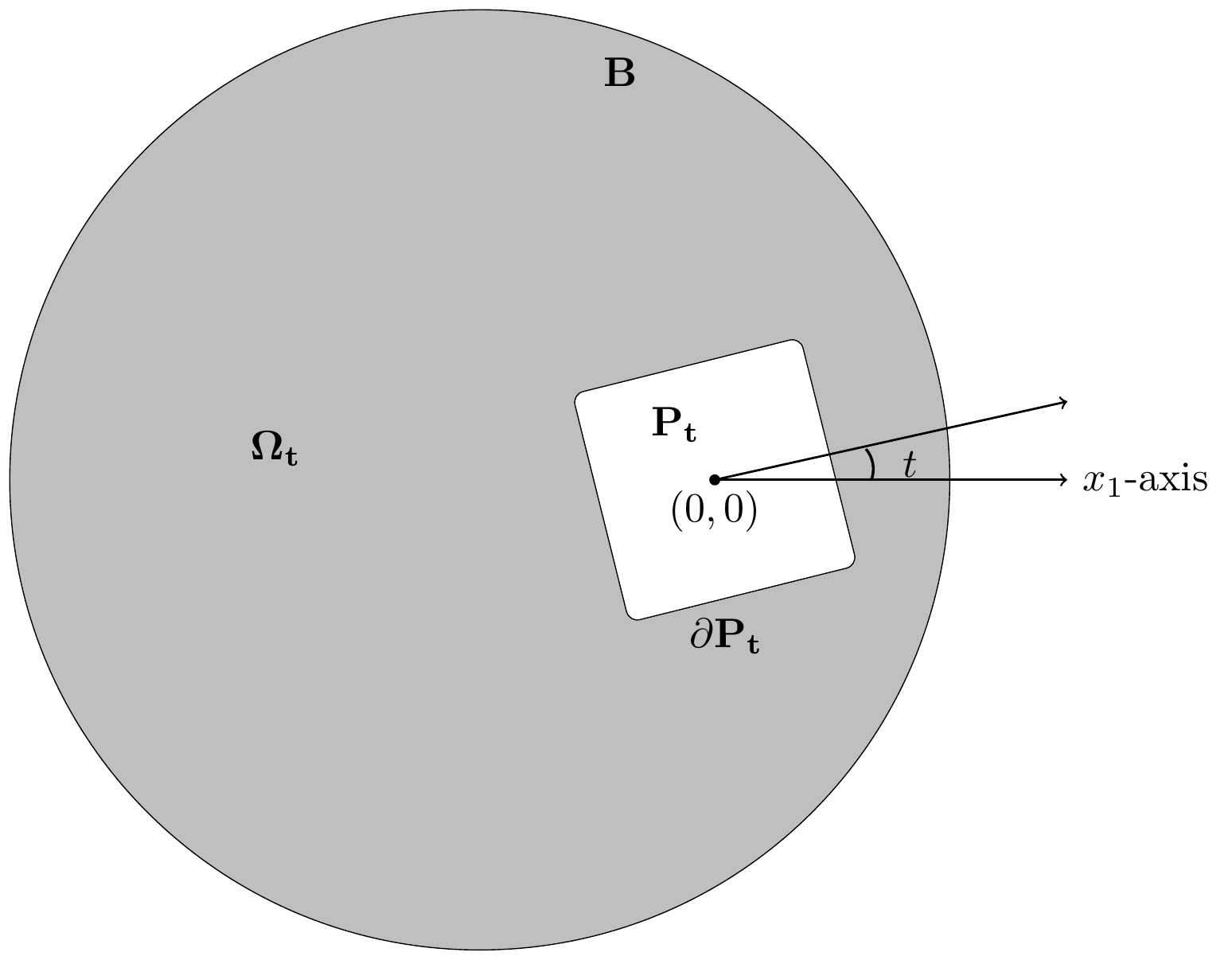}   }\hspace{10mm}
\subfloat{\includegraphics[width=0.4\textwidth]{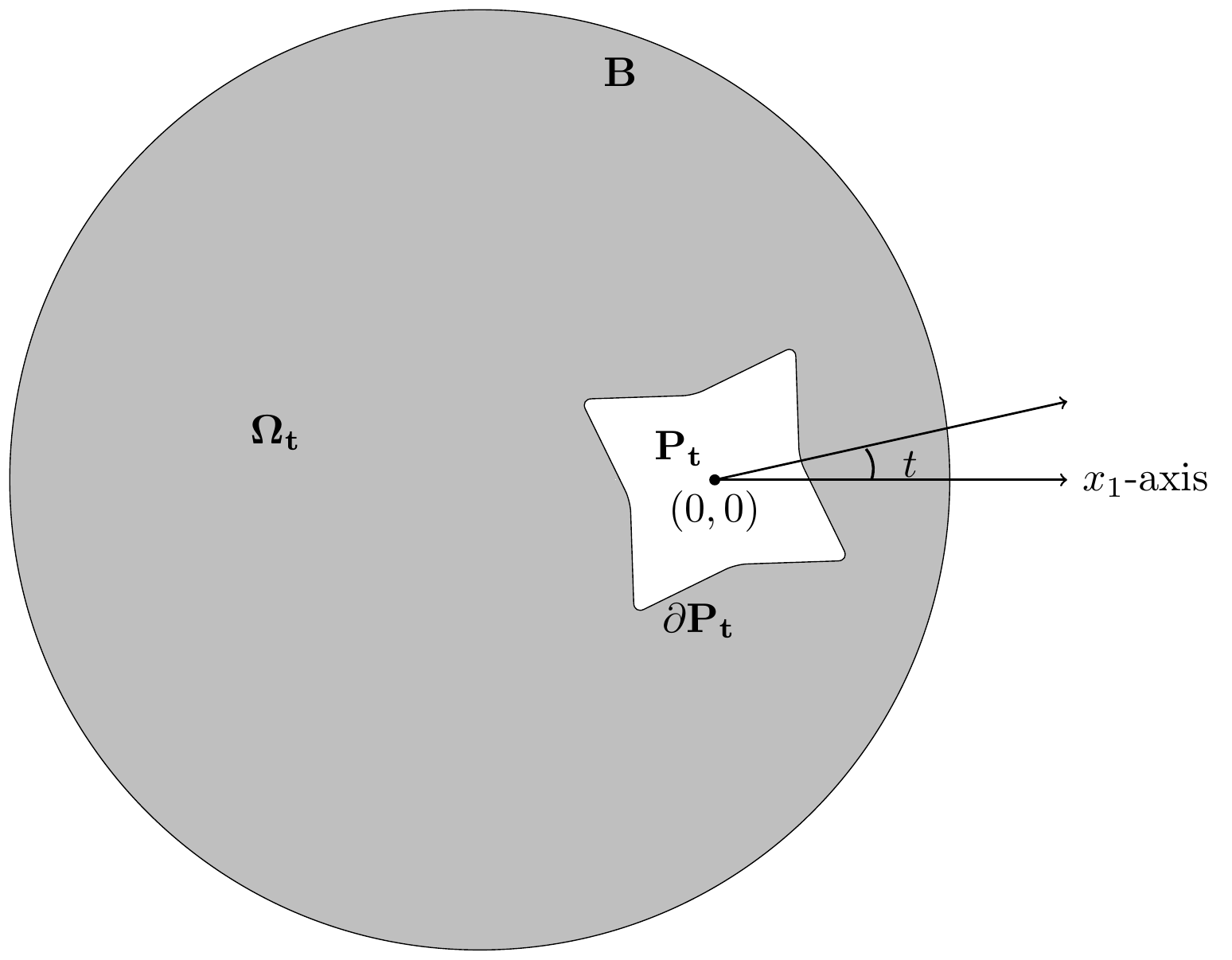}   }\\
\subfloat{\includegraphics[width=0.4\textwidth]{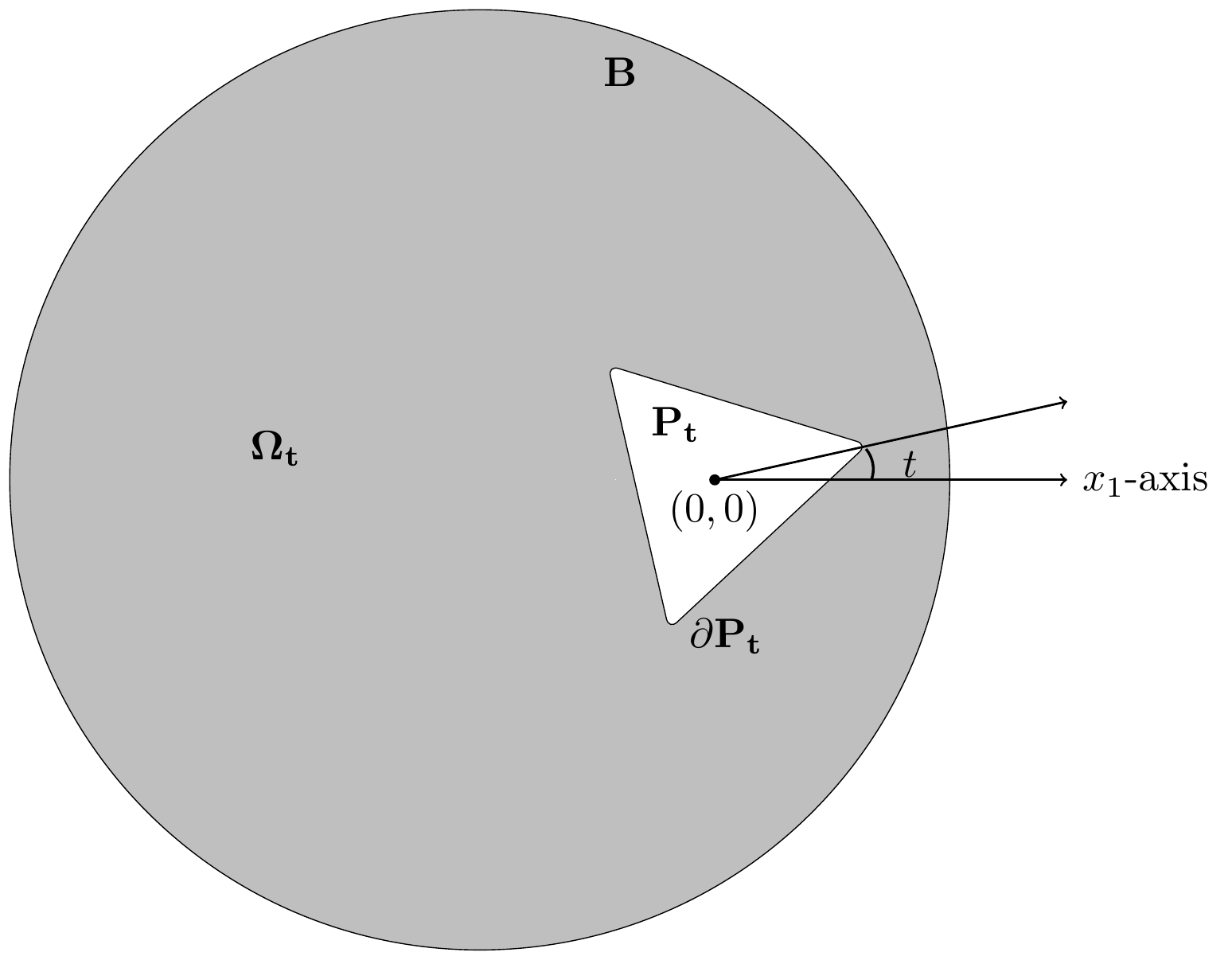}   }\hspace{10mm}
\subfloat{\includegraphics[width=0.4\textwidth]{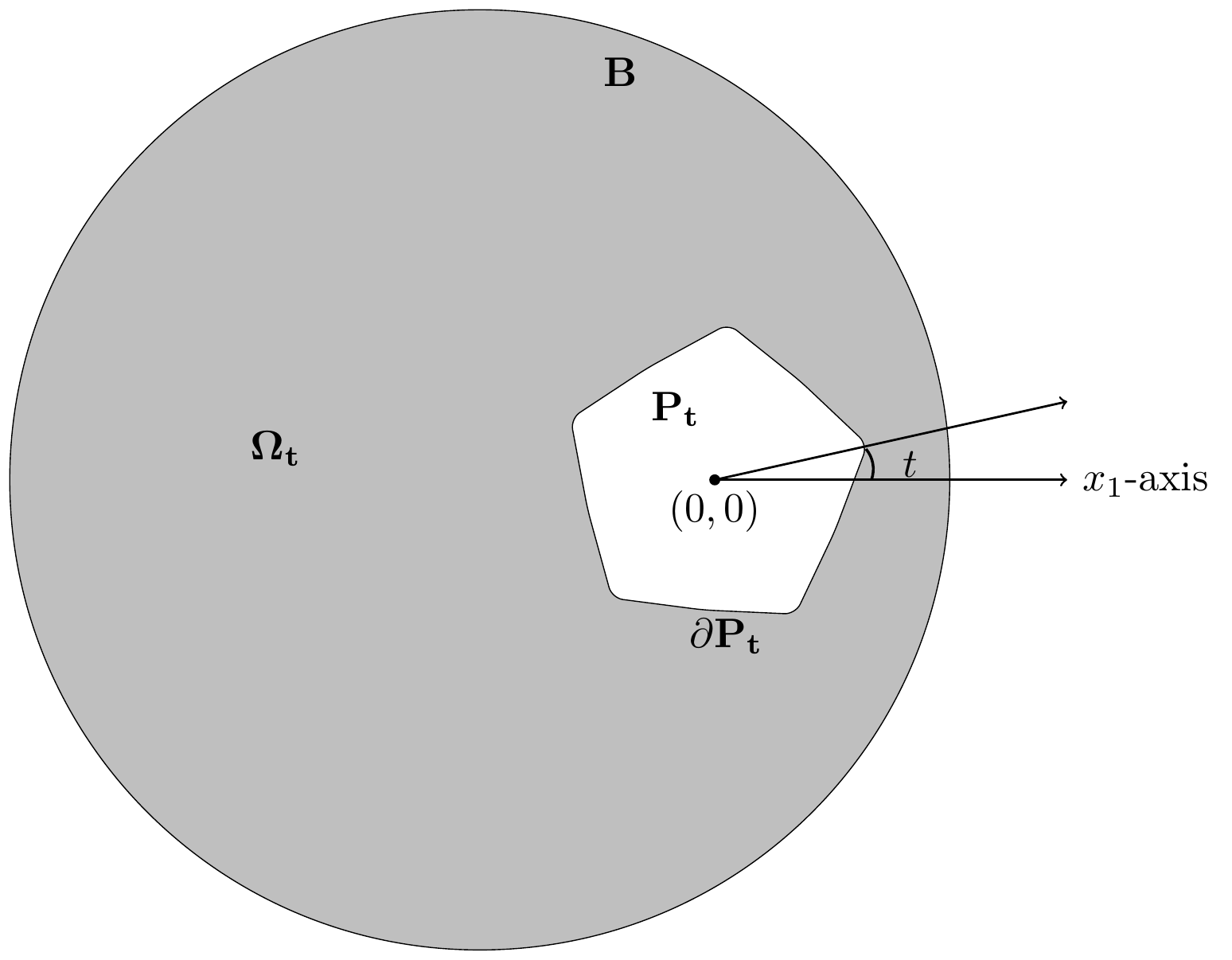}   }
\caption{Configuration at time $t$}\label{fig:t_config}
\end{figure} 
%{ Added the figures for $n$ odd.}

\subsubsection{Hadamard perturbation formula} Let $\lambda_1(t)$ denote the fundamental Dirichlet eigenvalue of the Laplacian on $\Omega_t$ i.e., $\lambda_1(t):= \lambda_1(\Omega_t)$. Let $y_1(t)$ denote the unique positive unit norm principal Dirichlet eigenfunction for the Laplacian on $\Omega_t$, i.e., $y_1(t)$ is the eigenfunction corresponding to $\lambda_1(t)$ on $\Omega_t$ satisfying 
\begin{equation}\label{laplace_equation}
\begin{aligned}
-\Delta u &=\lambda_1(t) \, u & \mbox{ in } & \Omega_t,\\
&u =0 & \mbox{ on } & \partial \Omega_t,\\
\int_{\Omega_t} &u^2(x) \, dx =1, &  & \\
u & > 0 & \mbox{ in } & \Omega_t.
\end{aligned}
\end{equation}
Then, by Proposition 3.1 in \cite{anisa_aithal}, the map $t \longmapsto \lambda_1(t)$ is a $\mathcal{C}^1$ map in $\mathbb{R}$ from a neighborhood of $0$ in $\mathbb{R}$. The same can be said about $\lambda_1(t_0 +t)$ for a fixed $t_0 \in \mathbb{R}$. %: \mathbb{R} \rightarrow (0,\infty)$ is a . %and the derivative $\lambda_1^\prime(t)$ of $\lambda_1$ at a point $t \in \mathbb{R}$ is given by the Hadamard perturbation formula (\ref{Hadamard}).
%Then, $\lambda_1: \mathbb{R} \rightarrow (0,\infty)$ is a $\mathcal{C}^1$ map 
The derivative $\lambda_1^\prime(t)$ of $\lambda_1$ at a point $t \in \mathbb{R}$ is given by the Hadamard perturbation formula, cf. \cite{Had, GS, Sch},
\begin{equation}\label{Hadamard}
\lambda_1^\prime(t) =- \int_{x \in \partial P_t} \left|{\dfrac{\partial y_1(t)(x)}{\partial \eta_t}}\right|^2 ~ \left<\eta_t, v\right>(x)~d\sigma(x)
\end{equation}
where $\eta_t(x)$ is the outward unit normal vector to $\Omega_t$ at $x \in \partial\Omega_t$, and $v \in \mathcal{C}_0^\infty(\Omega_t)$ is the %restriction to $\partial\Omega(t)=\partial D \cup \partial S_t$ of 
 deformation vector field defined as  %In this case, the vector $v$ vanishes on $\partial{D}$ and is given by 
\begin{equation}
v(\zeta) = \rho(\zeta) \, \textbf{i}\zeta, \qquad \forall \, \zeta \in  \mathbb{C} \cong \mathbb{R}^2.
\end{equation}
Here, $\rho: \mathbb{R}^2 \rightarrow [0,1]$ is a smooth function with compact support in $B$ such that $\rho \equiv 1$ in a neighborhood of $\overline{co(C_2(P))}$. 
\begin{remark}
We are interested in the outward unit normal to the domain $\Omega_t$ at points on the boundary $\partial P_t := \{f(\phi) e^{i \phi} \, | \, \phi \in [0, 2 \pi)\}$ of the obstacle $P_t$. Therefore, the outward unit normal with respect to the domain $\Omega_t$ at a point $f(\phi) e^{i \phi}$ on $\partial P_t$ will be the negative of the vector field $\eta(f(\phi) e^{i \phi})$, for $h=f$ in Lemma \ref{properties_eta}. 
\end{remark}
\subsubsection{$\lambda_1$ is an even and periodic function with period $\frac{2\pi}{n}$}
{ Recall that $n \geq 3$ is a fixed integer, even or odd}. Since $P_t$ is invariant under the action of the dihedral group $\mathbb{D}_n$, it follows that $\Omega(t+\frac{2\pi}{n}) = \Omega_t$  for each $t \in \mathbb{R}$. Let $R_0:\mathbb{R}^2 \rightarrow \mathbb{R}^2$ denote the reflection in $\mathbb{R}^2$ about the $x_1$-axis. That is, $R_0((x_1, x_2)) := (x_1, -x_2)$ $\forall (x_1, x_2) \in \mathbb{R}^2$. Then, we have $\rho_{2\pi-t} = R_0 \circ \rho_t \circ R_0$ for each $t \in \mathbb{R}^2$. This gives $P_{2\pi-t} = R_0 (P_t)$ and $\Omega_{2\pi-t}=R_0(\Omega_t)$. In $SO(2, \mathbb{R})$, $\rho_{s+t}= \rho_s \circ \rho_t= \rho_t \circ \rho_s \; \forall s,t \in \mathbb{R}$ and $\rho_{2\pi} =$ Id, the identity map. Therefore, we get  $P_{-t} = R_0 (P_t)$ and $\Omega_{-t}=R_0(\Omega_t)$ for all $t \in \mathbb{R}$. Moreover, since $\rho_{\frac{2 \pi}{n}}(P_t)=P_t$ for all $t \in \mathbb{R}$, $\Omega_{\frac{2 \pi}{n}+t} = \Omega_t$ for all $t \in \mathbb{R}$. This implies that $\lambda_1:  \mathbb{R} \rightarrow (0, \infty)$ is an even and periodic function with period $\frac{2\pi}{n}$. Thus we have, 
\begin{equation}\label{even_lambda}
\lambda_1\left({t+\dfrac{2\pi}{n}}\right)=\lambda_1(t),~ \mbox{ and }~ \lambda_1(-t)=\lambda_1(t) ~~~~ \forall \; t ~ \in\mathbb{R}.
\end{equation}
Therefore, it suffices to study the behavior of $\lambda_1(t)$ only on the interval $\left[{0,\frac{\pi}{n}}\right]$.  
\subsubsection{Sufficient condition for the critical points of $\lambda_1$}
The following theorem states a sufficient condition for the critical points of the function $\lambda_1: \mathbb{R} \rightarrow (0, \infty) $.
\begin{proposition}[Sufficient condition for critical points of $\lambda_1$]\label{critical_points}{  Let $n \geq 3$ be a fixed integer, even or odd}. For each $k = 0, 1, 2, \ldots,2 n-1$, $\lambda_1^\prime\left(k\frac{\pi}{n}\right)=0$. % That is, $ k\frac{\pi}{n},~ k\in\mathbb{Z}$, are critical points of $\lambda_1$ on $\mathbb{R}$.
\end{proposition}
\begin{proof}
Fix $k \in \{0, 1, 2, \ldots,2 n-1\}$. Let $t_k:=k\frac{\pi}{n}$. 
Then, the domain $\Omega_{t_k}$ is symmetric with respect to the $x_1$ axis. The first Dirichlet eigenfunction $y_1\left(t_k \right)$ satisfies
\begin{equation}\label{eigenfunction_symmetry}
u\circ R_0=u,
\end{equation}
where $R_0 \in O(2, \mathbb{R})$ is the reflection about the $x_1$-axis. Clearly, for each $x \in \partial P_{t_k} $ where $\eta$ is defined, $\eta(R_0(x))= DR_0(\eta(x))=R_0(\eta(x))$. Note also that 
\begin{equation}\label{normalderivative}
\dfrac{\partial \left(y_1\left(t_k \right) \circ R_0\right)}{\partial \eta}(x)=\dfrac{\partial \left(y_1\left(t_k\right)\right) }{\partial \eta}\left(R_0(x)\right) \end{equation} 
for each $x$ on $\partial P_{t_k}$ for which the normal derivative makes sense. By the Hadamard perturbation formula (\ref{Hadamard}), we have 
\begin{equation}\label{sum_integral_1newer}
\lambda_1^\prime\left({t_k}\right) =- \int_{\partial P_{t_k}^+} \left|{\dfrac{\partial \left(y_1(t_k )\right)}{\partial \eta_{t_k }}}\right|^2 (x)~ \left<\eta_{t_k}, v \right> (x)~d\sigma(x)
- \int_{\partial P_{t_k }^-} \left|{\dfrac{\partial \left(y_1(t_k )\right)}{\partial \eta_{t_k }}}\right|^2(x) ~ \left<\eta_{t_k },  v\right>(x)~d\sigma(x)
\end{equation}
where $\partial P_{t_k }^+$ and $\partial P_{t_k }^-$ represent the parts of $\partial P_{t_k }$ above the $x_1$-axis and below the $x_1$-axis respectively. 
Therefore we have 
$$
\lambda_1^\prime\left({t_k}\right) = -\int_{\partial P_{t_k }^+} \left|{\dfrac{\partial y_1(t_k)(x)}{\partial \eta_{t_k }}}\right|^2 ~ \left<\eta_{t_k },v \right>(x)~d\sigma(x)
- \int_{R_0\left(\partial P_{t_k }^+\right)} \left|{\dfrac{\partial y_1(t_k )(x)}{\partial \eta_{t_k }}}\right|^2 ~ \left<\eta_{t_k},  v\right>(x)~d\sigma(x).
$$
Using equation (\ref{normalderivative}) and property $(iii)$ of Lemma \ref{properties_eta}, we get 
$\lambda_1^\prime\left({t_k}\right) =0$.
Thus $k\frac{\pi}{n}$, $k \in \{0, 1, 2, \ldots,2 n-1\}$, are the critical points of $\lambda_1$.
\end{proof}
%\subsection{Necessary condition for the critical points of $\lambda_1$}
%In this section, we will show that $k\frac{\pi}{n}$ are the only critical points of $\lambda_1$ using a reflection method. 
%\subsubsection{A reflection method} 
\subsection{The sectors of $\Omega_t$}
{ Fix $n \geq 3$, even or odd.} For a fixed $t \in \mathbb{R}$ and $a, b \in \mathbb{Z}$, $a< b$, let  
$$ \sigma\left( {t+\frac{a\pi}{n},t+\frac{b\pi}{n}}\right):= \left\{ r \, e^{i \phi} \in  \mathbb{R}^2
\, \left| \,  \phi \in \left(t+\frac{a \pi}{n},t+\frac{b \pi}{n}\right), \right. \, r \in \mathbb{R} \right\}.$$
For convenience we will simply write $\sigma_{(a,b)}$ to denote $\sigma\left({t+\frac{a\pi}{n},t+\frac{b\pi}{n}}\right)$. When we write $\sigma_{(k,k+1)}$, $k \in \mathbb{Z}$, we take addition modulo $2n$, that is, $k, k+1 \in \left( \mathbb{Z}_{2n}, + \right)$. %Similarly, $\sigma_{(k,k+m)}$, $k, m \in \mathbb{Z}$, will be taken with respect to the addition modulo $2n$, that is, $k, k+m \in \left( \mathbb{Z}_{2n}, + \right)$. %Let $\sigma_{(a,b)} = \sigma\left({t+\frac{a\pi}{n},t+\frac{b\pi}{n}}\right)$. 
From equation (\ref{Hadamard}), we have 
\begin{equation}\label{integral}
\lambda_1^\prime(t) = -\sum_{k=0}^{2n-1} \int_{\partial P_t \cap \sigma\left({t+\frac{k\pi}{n},t+\frac{(k+1)\pi}{n}}\right)} \left|{\dfrac{\partial y_1(t)(x)}{\partial \eta_t}}\right|^2 ~ \left< \eta_t, v\right>(x)~d\sigma(x)
\end{equation}
Equation (\ref{integral}) can be written as
\begin{equation}\label{sum_integral}
\begin{aligned}
\lambda_1^\prime(t) = &-\sum_{k=0}^{n-1} \int_{\partial P_t \cap \sigma_{(k,k+1)}} \left|{\dfrac{\partial y_1(t)(x)}{\partial \eta_t}}\right|^2 ~ \left< \eta_t,  v \right>(x)~d\sigma(x)
-\sum_{k=n}^{2n-1} \int_{\partial P_t \cap \sigma_{(k,k+1)}} \left|{\dfrac{\partial y_1(t)(x)}{\partial \eta_t}}\right|^2~ \left< \eta_t, v \right>(x)~d\sigma(x).\\
\end{aligned}
\end{equation}
\begin{figure}[H]\centering
\subfloat{
\includegraphics[width=0.5\textwidth]{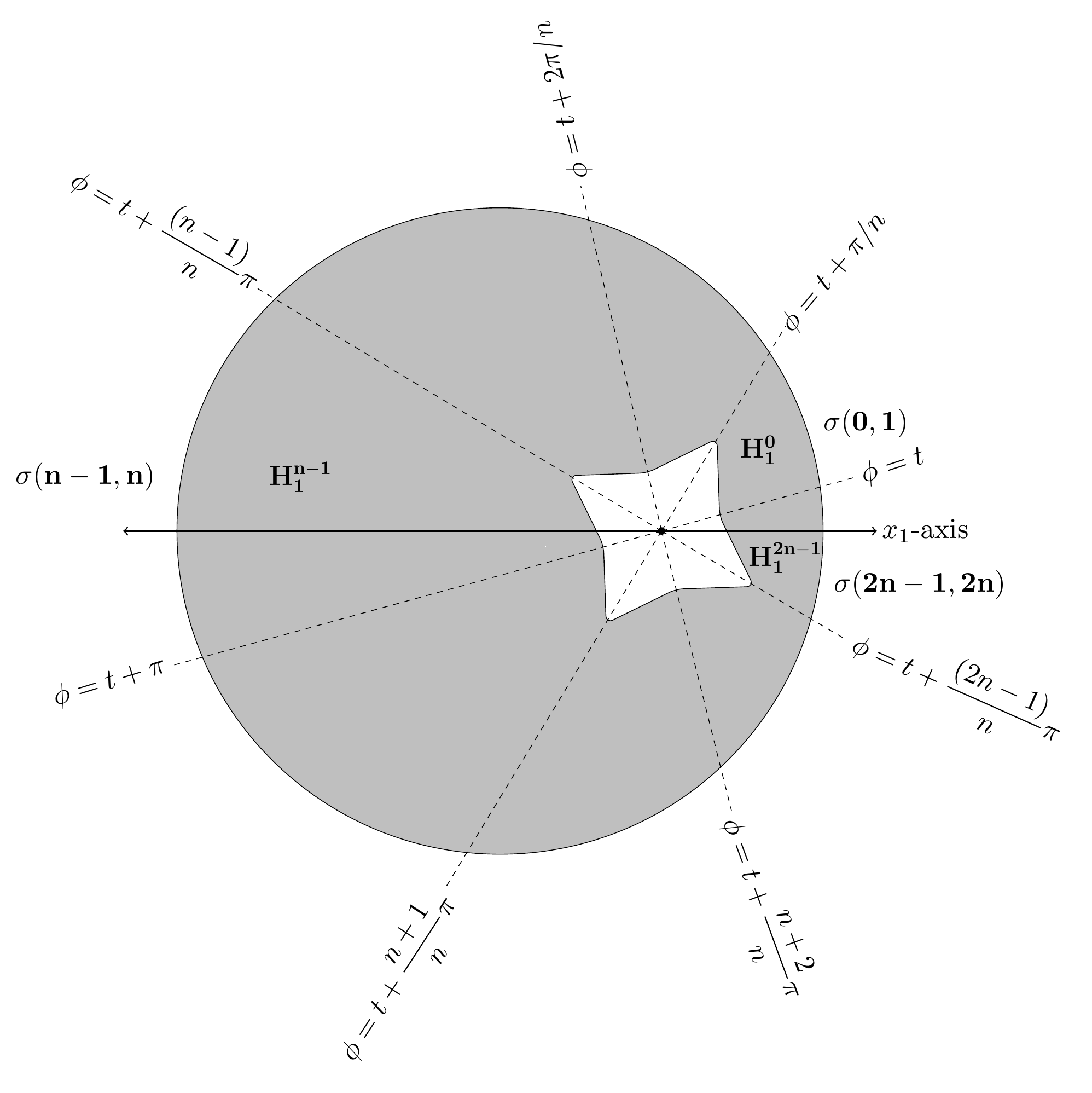}}\caption{Sectors of $\Omega_t$ for $n=4$}
\label{fig:sectors_set}
\end{figure}
%{ Add a figure for $n=5$.}

\begin{figure}[H]\centering
\subfloat{
\includegraphics[width=0.5\textwidth]{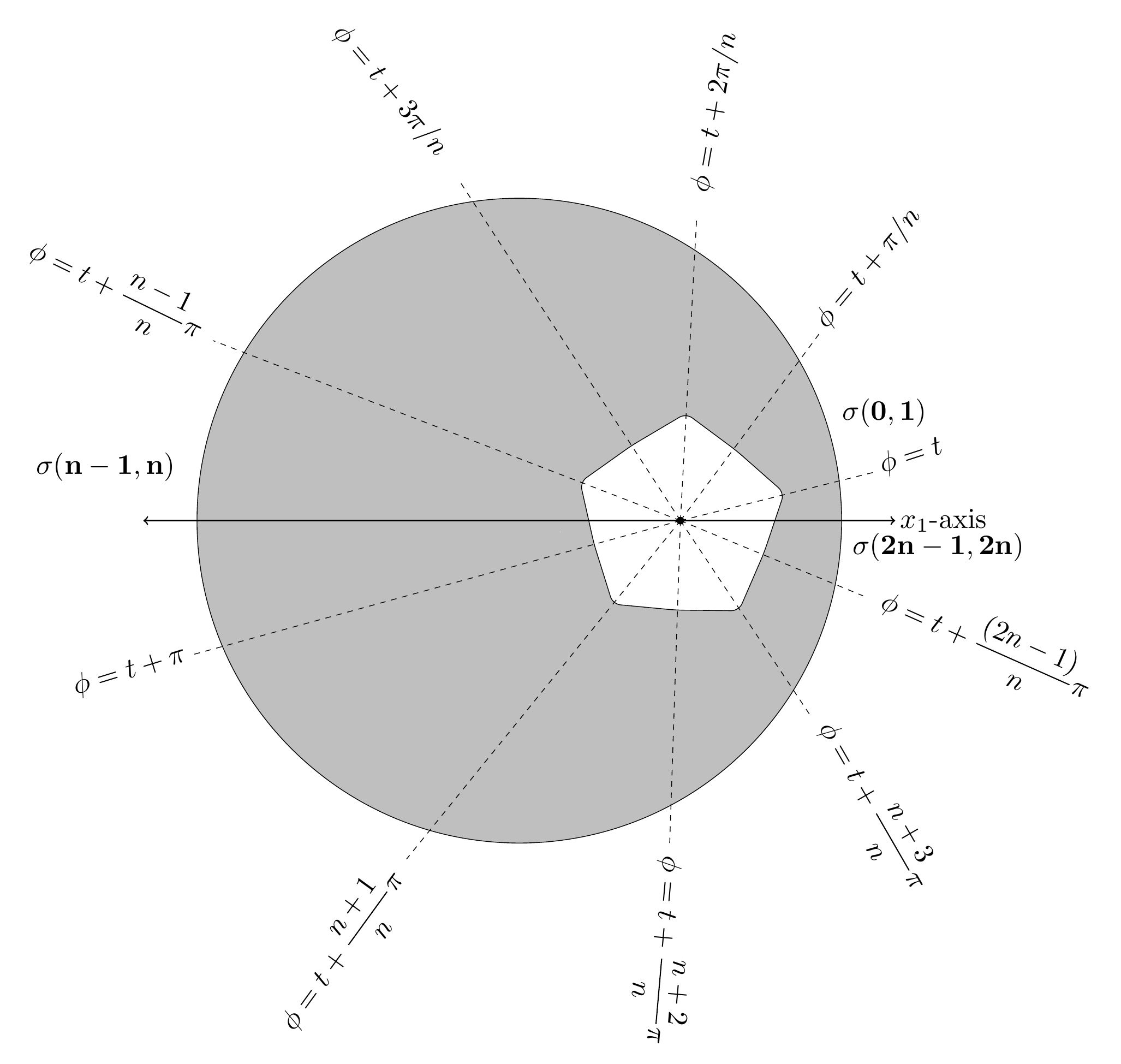}}\caption{Sectors of $\Omega_t$ for $n=5$}
\label{fig:sectors_set_odd}
\end{figure}
%{ Add a figure for $n=5$.}

We now fix a $t \in (0, \frac{2 \pi}{n})$ and note the following properties for the sectors $\sigma_{(k, k+1)}$
\begin{enumerate}
\item For $k = 0, 1, 2, \dots,  n-2$, each of the sectors $\sigma_{(k,k+1)}$ are completely above the $x_1$-axis.
\item For $k = n,\hdots, 2n-2$, the sectors $\sigma_{(k,k+1)}$ are completely below the $x_1$-axis.
\item The sectors $\sigma_{(n-1,n)}$ and $\sigma_{(2n-1,2n)}$ are partially above the $x_1$-axis and partially below it. 
\end{enumerate}
These facts are illustrated in Figure \ref{fig:sectors_set}.
\subsection{A sector reflection technique}
% and if 
% $x^\prime$ %denotes the reflection $R_{t+\frac{(k+1)\pi}{n}}$ of a point $x \in \mathbb{R}^n$ about the line $z_{t+\frac{(k+1)\pi}{n}}$. That is, if 
%$x^\prime:= R_{t+\frac{(k+1)\pi}{n}}(x)$, then for each $k =0,2,4, \ldots , n-2$, 
%$%$x^\prime \in \partial P_t \cap \sigma_{(k+1,k+2)} \; \forall \; x \in  \partial P_t \cap \sigma_{(k,k+1)},$$ 
{ Here onwards, we fix $n\geq 3$, $n$ even.} We recall here from section \ref{K} that, for $\alpha \in [0, 2 \pi]$, $z_\alpha := \{ r e^{i \alpha} \,| \, r \in \mathbb{R}\}$ denotes the line in $\mathbb{R}^2$ corresponding to angle $\phi=\alpha$, represented in polar co-ordinates. Let $R_{\alpha}:\mathbb{R}^2 \rightarrow \mathbb{R}^2$, $\alpha \in \mathbb{R}$, denote the reflection map about the $z_{\alpha}$-axis. For each $t \in \mathbb{R}$, the obstacle $P_t$ is symmetric with respect to the line $z_{t+\frac{(k+1)\pi}{n}}$.  We have, for $k=0,1,2, \ldots,2 n-1$, 
\begin{equation} \label{sectorinclusion}
\displaystyle{R_{t+\frac{(k+1)\pi}{n}}(\partial P_t \cap \sigma_{(k, k+1)}) =\partial P_t \cap \sigma_{( k+1, k+2)}}.
\end{equation}
For $k =0, 1,2, \dots, 2n-1$, let $H_1^k(t):= \Omega_t \cap \sigma_{(k,k+1)}$. Now, let $\tilde{H}_1^k :=\overbar{\Omega}_t \cap \sigma_{(k,k+1)} $, i.e., $\tilde{H}_1^k(t) =H_1^k(t) \cup \left( \overbar{H_1^k(t)} \cap \partial \Omega_t \right)$. %We have,  $R_{t+\frac{(k+1)\pi}{n}}(\partial B \cap \sigma(k, k+1)) =H_1^k(t)$. 

%{ Now, %for a fixed integer $n \geq 3$, $n$ even,} 
We consider pairs of consecutive sectors of $\Omega_t$, namely $\sigma_{(k,k+1)}$ and $\sigma_{(k+1,k+2)}$ for each $k=0, 2, 4, \ldots 2n-2$. We now prove the following lemma:
\begin{lemma}\label{containment_1}
 %For $k =0,2,4, \ldots , n-2$, let $H_1^k(t):= \overbar{\Omega}(t)\cap \sigma_{(k,k+1)}$. Then,
{ Fix $n\geq 3$, $n$ even.} For all $t\in (0,\frac{\pi}{n})$, we have the following
\begin{equation}\label{cont1}
R_{t+\frac{(k+1)\pi}{n}}(H_1^k(t)) \subsetneq H_1^{k+1}(t) ~~\mbox{ for } ~  k =0,2,4, \ldots , n-2.% \overbar{\Omega}(t)\cap \sigma_{(k+1,k+2)}
\end{equation} 
\begin{equation}\label{cont2}
R_{t+\frac{(k+1)\pi}{n}}(\tilde{H}_1^k(t)) \subsetneq \tilde{H}_1^{k+1}(t) \setminus \partial B  ~~\mbox{ for } ~  k =0,2,4, \ldots , n-2.% \overbar{\Omega}(t)\cap \sigma_{(k+1,k+2)}
\end{equation} 
\begin{equation}\label{cont3}
R_{t+\frac{(k+1)\pi}{n}}(H_1^{k+1}(t)) \subsetneq H_1^{k}(t) ~~\mbox{ for } ~  k =n, n+2, \ldots , 2n-2.% \overbar{\Omega}(t)\cap \sigma_{(k+1,k+2)}
\end{equation} 
\begin{equation}\label{cont4}
R_{t+\frac{(k+1)\pi}{n}}(\tilde{H}_1^{k+1}(t)) \subsetneq \tilde{H}_1^{k}(t) \setminus \partial B  ~~\mbox{ for } ~  k =n,n+2, \ldots , 2n-2.% \overbar{\Omega}(t)\cap \sigma_{(k+1,k+2)}
\end{equation} 
%{ May be shift the $\mathcal{O}_{k_0}$ reflection here.}
\end{lemma}
\begin{proof}
We first prove (\ref{cont1}--\ref{cont2}) for $k =0,2,4, \ldots , n-4$, where the pair of sectors $\sigma_{(k,k+1)}$ and $\sigma_{(k+1,k+2)}$ are completely above the $x_1$-axis. A similar technique can be used to prove (\ref{cont3}--\ref{cont4}) for $k =n, n+2, \ldots , 2n-4$, where the sectors $\sigma_{(k,k+1)}$ and $\sigma_{(k+1,k+2)}$ are completely below the $x_1$-axis. We then prove (\ref{cont1}--\ref{cont2}) for $k = n-2$ separately and similarly prove (\ref{cont3}--\ref{cont4}) for $k = 2n-2$ separately.

Let $\beta \in [0, \frac{\pi}{n}]$ be arbitrary. The line $L_1$ containing the center $\underline{o}$ and the point  $$p_1=g\left(t+ (k+1)\frac{\pi}{n}-\beta \right) \left({\cos ({t+ (k+1)\frac{\pi}{n}-\beta}),\sin ({t+ (k+1)\frac{\pi}{n}-\beta})}\right) \in \partial B$$ is reflected about $z_{t+ (k+1)\frac{\pi}{n}}$-axis to the line $L_2$ containing $\underline{o}$ and the point $$p_2=g \left( {t+ (k+1)\frac{\pi}{n}+\beta}\right) \left({\cos ({t+ (k+1)\frac{\pi}{n}+\beta}),\sin ({t+ (k+1)\frac{\pi}{n}+\beta})}\right) \in \partial B,$$
(see Figure \ref{fig:sectors_containment}).

 Since $P_t$ is invariant under this reflection and $B$ is star-shaped with respect to $\underline{o}$, to prove (\ref{cont1}--\ref{cont2}), it suffices to show that $$ g\left({t+\frac{(k+1)\pi}{n}-\beta}\right) < g\left({t+\frac{(k+1)\pi}{n}+\beta}\right) ~~~~ \mbox{ for } ~~  k =0,2,4, \ldots , n-2.$$

\begin{figure}[H]\centering
\subfloat{
\includegraphics[width=0.4\textwidth]{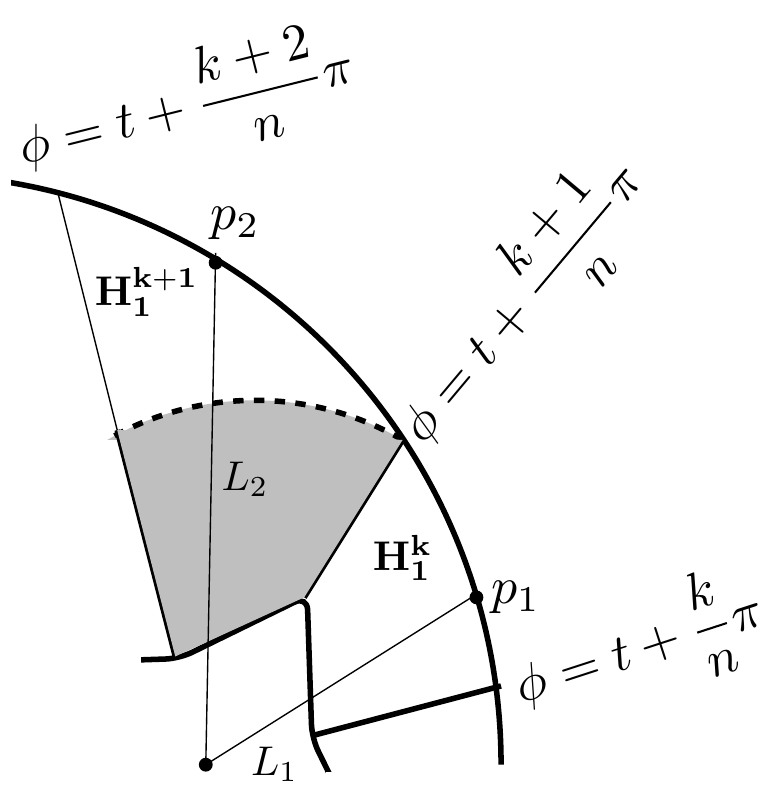}}
\caption{Reflection of sector $H_1^k $ about the axis $z_{t+\frac{(k+1)\pi}{n}}$}
\label{fig:sectors_containment}
\end{figure}

Now, for $k =0,2,4, \ldots , n-4$, $\left({t+\frac{k\pi}{n},t+\frac{(k+2)\pi}{n}}\right) \subset (0, \pi) %\subset (0, \pi)
$. So, by Lemma \ref{bound_monot}, $g$ is a strictly increasing function of the argument in $\left({t+\frac{k\pi}{n},t+\frac{(k+2)\pi}{n}}\right)$ for $k =0,2,4, \ldots , n-4$. Therefore, (\ref{cont1}--\ref{cont2}) for $k =0,2,4, \ldots , n-4$ follow from the fact that $t+\frac{(k+1)\pi}{n}-\beta < t+\frac{(k+1)\pi}{n}+\beta$.

Next we consider the case $k=n-2$. The sector $\sigma_{(n-2,n-1)}$ is completely above the $x_1$-axis whereas the sector $\sigma_{(n-1,n)}$ is partially above and partially below the $x_1$-axis. If the point $p_2$ is above the $x_1$-axis we have $0< t+\frac{(n-1)\pi}{n}-\beta < t+\frac{(n-1)\pi}{n}+\beta < \pi$. Since $g$ is strictly increasing in $[0, \pi]$, we have the desired results  (\ref{cont1}--\ref{cont2}) in this case.

Suppose the point $p_2$ is below the $x_1$-axis. Let $\theta>0$ be the angle between $L_2$ and the positive $x_1$-axis. Then, since $\partial B$ is symmetric with respect to the $x_1$-axis, we get
$
g\left({t+\frac{(n-1)\pi}{n}+\beta}\right)=g\left({t+\frac{(n-1)\pi}{n}+(\beta-2 \theta )}\right)
$.
Now, since $\beta > \theta$, we have $ \left({t+\frac{(n-1)\pi}{n}+(\beta-2 \theta )}\right) > \left({t+\frac{(n-1)\pi}{n}-\beta}\right)$. Clearly, $\left({t+\frac{(n-1)\pi}{n}-\beta}\right)$ $\in (0, \pi)$. Moreover, by the choice of $\theta$, $ \left({t+\frac{(n-1)\pi}{n}+(\beta-2 \theta)}\right) \in (0, \pi)$.  Since $g$ is a strictly increasing function of the argument on $[0, \pi]$, we have the desired results  (\ref{cont1}--\ref{cont2}) in this case.

For $k=2n-2$, we first note that we can write $\sigma_{(2n-2,2n-1)}$ as $\sigma_{(-2,-1)}$
and
$\sigma_{(2n-1,2n)}$ as $\sigma_{(-1,0)}$.
We also note that the sector $\sigma_{(-2,-1)}$ is completely below the $x_1$-axis, whereas  the sector $\sigma_{(-1,0)}$ is partially above and partially below the $x_1$-axis.
The line $L_3$ joining the center $\underline{o}$ of $P_t$ to the point 
$$p_3 = g\left({t-\frac{\pi}{n}+\beta}\right)\left({\cos\left({t-\frac{\pi}{n}+\beta}\right), \sin\left({t-\frac{\pi}{n}+\beta}\right)}\right) \in \partial B$$ 
is reflected about $z_{t-\frac{\pi}{n}}$ to the line $L_4$ joining $\underline{o}$ to the point 
$$p_4 = g\left({t-\frac{\pi}{n}-\beta}\right)\left({\cos\left({t-\frac{\pi}{n}-\beta}\right),\sin\left({t-\frac{\pi}{n}-\beta}\right)}\right)  \in \partial B,$$
(see Figure \ref{fig:sectors_containment2}).

Thus to prove (\ref{cont3}, \ref{cont4}), it suffices to show that $$ g\left({t-\frac{\pi}{n}+\beta} \right) < g\left({t-\frac{\pi}{n}-\beta }\right).$$
\begin{figure}[H]\centering
\subfloat{
\includegraphics[width=0.4\textwidth]{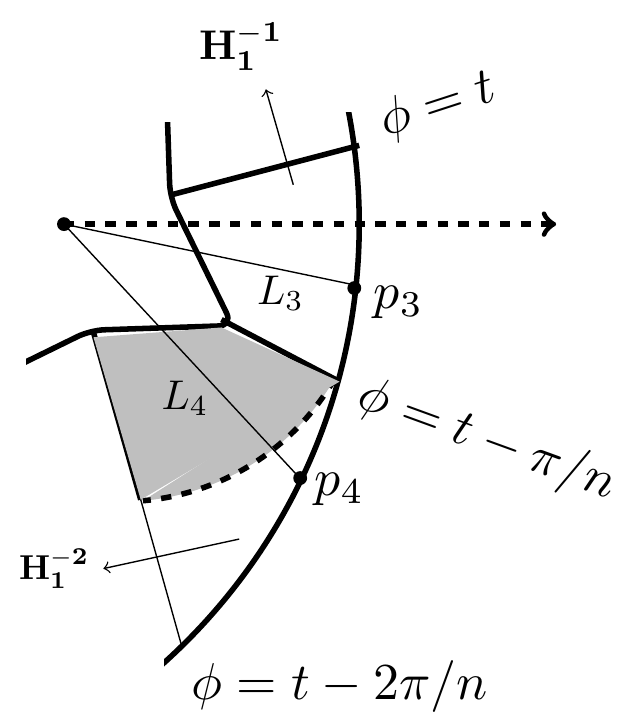}}
\caption{Reflection of sector $H_1^{-1} $ about the axis $z_{t-\frac{\pi}{n}}$}
\label{fig:sectors_containment2}
\end{figure}
 Suppose the point $p_3$ is above the $x_1$-axis. Let $r>0$ be the angle between $L_3$ and the positive $x_1$-axis. Then, $
g\left({t-\frac{\pi}{n}+\beta}\right) = g\left({t-\frac{\pi}{n}+(\beta-2r)}\right)$.
Now, $r< \beta$ implies that $ \left({t-\frac{\pi}{n}+(\beta-2r)}\right) > \left({t-\frac{\pi}{n}-\beta}\right) $. Since $g$ is a strictly decreasing function of the argument in $[\pi, 2 \pi]$ %\partial{D}^-$, . 
we get the desired results (\ref{cont3}, \ref{cont4}) in this case.

If the point $p_3$ is below the $x_1$-axis then
$2 \pi >  \left({t-\frac{\pi}{n}+\beta}\right) > \left({t-\frac{\pi}{n}-\beta}\right) > \pi$, and the fact that $g$ is a strictly decreasing function of the argument in $[\pi, 2 \pi]$ give %\partial{D}^-$, .
the desired results (\ref{cont3}, \ref{cont4}) in this case. 
\end{proof}
%\noindent { $n$ odd case: We consider pairs of consecutive sectors of $\Omega_t$, namely $\sigma_{(k,k+1)}$ and $\sigma_{(k+1,k+2)}$ for each $k=0, 2, 4, \ldots n-3$ and for each $k= n, n+2, \ldots, 2n-3$. 
\subsection{The rotating plane method}%Analysing the sign of $\lambda_1^\prime(t)$, $t \in (0, \frac{\pi}{n})$ in (\ref{sum_integral})}
\noindent { Recall here that $n \geq 3$ is a fixed even integer.} In order to study the behavior of $\lambda_1$ as a function of $t$, we now analyze the two terms appearing on the right hand side of (\ref{sum_integral}) which is an expression for $\lambda_1^\prime(t)$.

For each $\phi \in [0, \pi]$, by Lemma \ref{properties_eta} we have
\begin{equation}\label{ref_0}\left<\eta_t , v \right> \left({t+\frac{(k+1)\pi}{n}+\phi}\right) = -\left< \eta_t,  v\right> \left({t+\frac{(k+1)\pi}{n}-\phi}\right)~~~\mbox{for}~~ k =0,2,4, \ldots , n-2.
\end{equation}
In particular,  (\ref{ref_0}) holds for each $\phi \in [0, \frac{\pi}{n}]$. 
In other words, if $x^\prime:= R_{t+\frac{(k+1)\pi}{n}}(x)$, then by equation (\ref{sectorinclusion}), for each $k =0,2,4, \ldots , n-2$, 
$x^\prime \in \partial P_t \cap \sigma_{(k+1,k+2)}$ %\; \mbox{ 
for each %} \; 
$x \in  \partial P_t \cap \sigma_{(k,k+1)}$, and
$$\left<\eta_t,  v\right> \left( x^\prime \right) = - \left< \eta_t , v \right>\left( x \right) ~~~ \forall \; x \in  \partial P_t \cap \sigma_{(k,k+1)}.$$
Thus, for each $k =0,2,4, \ldots , n-2$, we have the following 
\begin{equation}\label{ref_1}
\begin{aligned}
&\int_{\partial P_t \cap \sigma_{(k,k+1)}} \left|{\dfrac{\partial y_1(t)}{\partial \eta_t}}(x)\right|^2~\left< \eta_t,  v\right>(x)~d\sigma + \int_{\partial P_t \cap \sigma_{(k+1,k+2)}} \left|{\dfrac{\partial y_1(t)}{\partial \eta_t}}(x)\right|^2~ \left<\eta_t , v\right>(x)~d\sigma\\
=&\int_{\partial P_t \cap \sigma_{(k,k+1)}} \left({\left|{\dfrac{\partial y_1(t)}{\partial \eta_t}}(x)\right|^2-\left|{\dfrac{\partial y_1(t)}{\partial \eta_t}}(x^\prime)\right|^2}\right)~ \left<\eta_t,  v \right>(x)~d\sigma.
\end{aligned}
\end{equation}
Now, we know that $f$ is a positive and a strictly increasing function of $ \phi $ in $\left({t+\frac{k\pi}{n},t+\frac{(k+1)\pi}{n}}\right)$  for each $k =0,2,4, \ldots , n-2$. Thus, applying Lemma  \ref{properties_eta} for $\eta_t = -n$ we get
\begin{equation}\label {inrprdctsgn1}\left< \eta_t,  v \right>  >0 ~ \mbox{ on }\partial P_t \cap \sigma_{(k,k+1)}~\mbox{ for each } k =0,2,4, \ldots , n-2.
\end{equation}
Using a similar argument, we have the following: For each $k =n, n+2, \ldots, 2n-2$,
\begin{equation}\label{ref_2}
\begin{aligned}
&\int_{\partial P_t \cap \sigma_{(k,k+1)}} \left|{\dfrac{\partial y_1(t)}{\partial \eta_t}}(x)\right|^2~\left< \eta_t,  v\right>(x)~d\sigma + \int_{\partial P_t \cap \sigma_{(k+1,k+2)}} \left|{\dfrac{\partial y_1(t)}{\partial \eta_t}}(x)\right|^2~ \left<\eta_t , v\right>(x)~d\sigma\\
=&\int_{\partial P_t \cap \sigma_{(k+1,k+2)}} \left({\left|{\dfrac{\partial y_1(t)}{\partial \eta_t}}(x)\right|^2-\left|{\dfrac{\partial y_1(t)}{\partial \eta_t}}(x^\prime)\right|^2}\right)~ \left<\eta_t,  v \right>(x)~d\sigma,
\end{aligned}
\end{equation}
where $x^\prime:= R_{t+\frac{(k+1)\pi}{n}}(x)$. Then, for each $k =n, n+2, \ldots, 2n-2$, $x^\prime \in \partial P_t \cap \sigma_{(k,k+1)}$ %\; \mbox{
for each %} \; 
$x \in  \partial P_t \cap \sigma_{(k+1,k+2)}$.
We note that the function $f$ is a positive and a strictly increasing function of $ \phi $ in  $\left( t+\frac{(k+2)\pi}{n}, t+\frac{(k+1)\pi}{n} \right)$ for each $k =n, n+2, \ldots, 2n-2$. Thus, applying Lemma  \ref{properties_eta} for $\eta_t = -n$ we get \begin{equation}\label {inrprdctsgn2}\left< \eta_t,  v \right>  >0 ~ \mbox{ on }\partial P_t \cap \sigma_{(k+1,k+2)} ~\mbox{ for each } k =n, n+2, \ldots , 2n-2.\end{equation}
\subsection{Necessary condition for the critical points of $\lambda_1$}
{ Recall here that $n \geq 3$ is a fixed even integer.} We finally show that $\left\{\frac{k \pi}{n}\, |\, k =0,1, \ldots n-1\right\}$ are the only critical points of $\lambda_1$, and that, between every pair of consecutive critical points of $\lambda_1$, it is a strictly monotonic function of the argument. In view of Proposition \ref{critical_points} and equation (\ref{even_lambda}), it now suffices to study the behavior of $\lambda_1$ only on the interval $\left({0,\frac{\pi}{n}}\right)$.   
\begin{proposition}[Necessary condition for critical points]\label{complete_critical_points}{  Fix $n \geq 3$, $n$ even.}
 For each $t\in (0,\frac{\pi}{n})$, $\lambda_1^\prime(t) > %{ >} 
0$. %Thus $\lambda(t)$ is strictly increasing on $(0,\frac{\pi}{n})$. As a consequence, $ k\frac{\pi}{n},~ k\in\mathbb{Z}$ are the only critical points of $\lambda$ on $\mathbb{R}$. 
\end{proposition}
\begin{proof} Fix $t\in (0,\frac{\pi}{n})$.
Using (\ref{ref_1}) and (\ref{ref_2}), integral (\ref{sum_integral}) can be written as 
%\begin{equation}\label{sum_integral_1}
%\begin{aligned}
%\lambda_1^\prime(t) = & %%2\sum_{k=0}^{n-1} \int_{\partial P_t \cap \sigma_{(k,k+1)}} \left|{\dfrac{\partial y_1}{\partial \eta_t}}\right|^2(x)~ \left<\eta_t,  v\right>(x)~d\sigma(x)
%\sum_{k=0}^{n-1} \int_{\partial P_t \cap \sigma_{(k,k+1)}} \left( \left|{\dfrac{\partial y_1}{\partial \eta_t}}\right|^2(x)+\left|{\dfrac{\partial y_1}{\partial \eta_t}}\right|^2 (x^*) \right) ~ \left<\eta_t,  v\right> (x)~d\sigma(x)
%\end{aligned}
%\end{equation}

\begin{equation}\label{final_integralnew}
\begin{aligned}
\lambda_1^\prime(t) 
%= & \sum_{0 \leq k \leq n-1}
%\int_{\partial P_t \cap \sigma_{(k,k+1)}} \left( \left| \dfrac{\partial u(x)}{\partial \eta_t} \right|^2 \right) ~ \eta_t \cdot v(x) ~ d\sigma(x) + \sum_{n \leq k \leq 2n-1}
%\int_{\partial P_t \cap \sigma_{(k,k+1)}} \left( \left| \dfrac{\partial u(x)}{\partial \eta_t} \right|^2 \right)~ \eta_t \cdot v(x) ~ d\sigma(x) \\ 
 &= -\sum_{\substack{0\leq k\leq n-2\\k \mbox{ \small{even} }}}
\int_{\partial P_t \cap \sigma_{(k,k+1)}} \left({\left|{\dfrac{\partial y_1(t)(x)}{\partial \eta_t}}\right|^2-\left|{\dfrac{\partial y_1(t)(x^\prime)}{\partial \eta_t}}\right|^2}\right)~ \left<\eta_t, v\right> (x)~d\sigma(x)\\ &- \sum_{\substack{n\leq k\leq 2n-2\\k \mbox{ \small{even} }}} \int_{\partial P_t \cap \sigma_{(k+1,k+2)}} \left({\left|{\dfrac{\partial y_1(t)(x)}{\partial \eta_t}}\right|^2-\left|{\dfrac{\partial y_1(t)(x^\prime)}{\partial \eta_t}}\right|^2}\right)~ \left<\eta_t, v \right> (x)~d\sigma(x) \\
\end{aligned}
\end{equation}
%\begin{equation}\label{final_integral}
%\begin{aligned}
%\lambda_1^\prime(t) = &2\sum_{\substack{0\leq k\leq n-2\\k \mbox{ \small{even} }}}
%%\equiv 0\bmod 2}
%\int_{\partial P_t \cap \sigma_{(k,k+1)}} \left({\left|{\dfrac{\partial u(x)}{\partial \eta_t}}\right|^2-\left|{\dfrac{\partial u(x^\prime)}{\partial \eta_t}}\right|^2}\right)~ \eta_t\cdot v(x)~d\sigma\\
%\end{aligned}
%\end{equation}
Let $\displaystyle H(t):= \bigcup_{\substack{0\leq k\leq n-2\\k \mbox{ \small{even} }}}
%\equiv 0\bmod 2}}
H_1^k(t)$. Let $w(x) := y_1(t)(x) - y_1(t)(x^\prime)$. By Lemma \ref{containment_1}, the real valued function $w$ is well-defined on $H(t)$. Moreover, $w \equiv 0$ on $\partial P_t \cap \partial H(t)$ and also on $\partial H(t) \cap z_{t+ k \frac{\pi}{n}}$ for each $k=1, 3, \ldots n-1$.  That is,
$$
w(x) = 0 ~ \forall \; x\in \partial{H(t)}\bigcap \left({\partial{P_t}\bigcup_{\substack{1\leq k\leq n-1\\k \mbox{ \small{odd} }}}z_{t+\frac{k\pi}{n}}}\right).
$$
Moreover, since $y_1(t)$ vanishes on $\partial{B}$ and is positive inside $\Omega(t)$, and since for each $k=0, 2, \ldots n-2$,  the reflection of $\partial{H_1^k(t)}\cap \partial B$ about the axis $z_{t+( k+1) \frac{\pi}{n}}$ lies completely inside $H_1^{k+1}(t) \subset \Omega(t)$ we have the following 
$$w(x) < 0 ~\forall \; x \in \left( \partial{H(t)}\cap \partial B \right) \setminus \left(\bigcup_{\substack{1\leq k\leq n-1\\k \mbox{ \small{odd} }}} z_{t+\frac{k\pi}{n}}\right).$$ %and $w(x) <0$ for some $x$ in $\partial{H(t)}\cap \partial B$. 
{ Now, we claim that 
\begin{equation}\label{signw}
w(x) < 0 ~\forall \; x \in  \partial H(t) \bigcap \bigcup_{\substack{0 \leq k\leq n-2\\k \mbox{ \small{even} }}} z_{t+\frac{k\pi}{n}}.
\end{equation}
We prove this by proving that for each $k$, $0 \leq k \leq n-2, k$ even, $w(x) < 0 ~\forall \; x \in  \partial H_1^k(t) \cap z_{t+\frac{k\pi}{n}}$. For, let's fix a $k_0$ such that $0 \leq k_0 \leq n-2, k_0$ even. %Consider $z_{t+\frac{(k_0+1)\pi}{n}}$, an axis of symmetry of $P_t$. 
Now, the axis of symmetry $z_{t+\frac{(k_0+1)\pi}{n}}$ divides $\Omega_t$ in two unequal components. Let us denote the smaller component of the two by $\mathcal{O}_{k_0}(t)$. That is, $\mathcal{O}_{k_0}(t):= \Omega_t \cap \sigma_{(-(k_0+1+n), k_0+1)}$. Now, it can be shown that $R_{t+\frac{(k_0+1)\pi}{n}}\left(\mathcal{O}_{k_0}(t)\right) \subset \Omega_t \cap (\overline{\mathcal{O}_{k_0}(t)})^c $. Therefore, if we define $w_{k_0}(x) := y_1(t)(x) - y_1(t)(x^\prime)$, then the real valued function $w_{k_0}$ is well-defined on $\mathcal{O}_{k_0}(t)$. Here, $x^\prime := R_{t+\frac{(k_0+1)\pi}{n}}(x)$ for $x \in \mathcal{O}_{k_0}(t)$. Moreover, $w \equiv 0$ on $\partial P_t \cap \partial \mathcal{O}_{k_0} (t)$ and also on $\partial \mathcal{O}_{k_0}(t) \cap z_{t+ (k_0 +1)\frac{\pi}{n}}$. %for each $k=1, 3, \ldots n-1$.  
That is,
$$
w_{k_0}(x) = 0 ~ \forall \; x\in \partial{\mathcal{O}_{k_0}(t)}\bigcap \left({\partial{P_t}\cup z_{t+\frac{(k_0+1)\pi}{n}}}\right).
$$
Moreover, since $y_1(t)$ vanishes on $\partial{B}$ and is positive inside $\Omega_t$, and since the reflection of $\partial{\mathcal{O}_{k_0}(t)}\cap \partial B$ about the $z_{t+( k_0+1) \frac{\pi}{n}}$-axis lies completely inside $ \Omega_t$ we have the following 
$$w_{k_0}(x) < 0 ~\forall \; x \in \left( \partial \mathcal{O}_{k_0}(t )\cap \partial B \right) \setminus  z_{t+\frac{(k_0+1)\pi}{n}}.$$
Therefore, the non-constant function $w_{k_0}$ satisfies 
\begin{equation}\label{w_{k_0}}
\begin{aligned}
-\Delta w_{k_0} &= \lambda_1(t)\,  w_{k_0} & \mbox{ in } & \mathcal{O}_{k_0}(t),\\
w_{k_0} & \leq 0, & \mbox{ on } & \partial \mathcal{O}_{k_0}(t).
\end{aligned}
\end{equation}
Hence, by the maximum principle, $w_{k_0}<0$ in $\mathcal{O}_{k_0}(t)$. In particular, $w_{k_0}<0$ in $\partial H_1^{k_0}(t) \cap z_{t+\frac{k_0\pi}{n}}$. Now, by definition, $w$ and $w_{k_0}$ coincide in $H_1^{k_0}$. Therefore, by continuity of of both $w, w_{k_0}$ we get, $w <0$ in $\partial H_1^{k_0}(t) \cap z_{t+\frac{k_0\pi}{n}}$. But $k_0$ such that $0 \leq k_0 \leq n-2, k_0$ even, was chosen arbitrarily. This proves our claim (\ref{signw})
 }

\begin{figure}[H]\centering
\subfloat{
\includegraphics[width=0.7\textwidth]{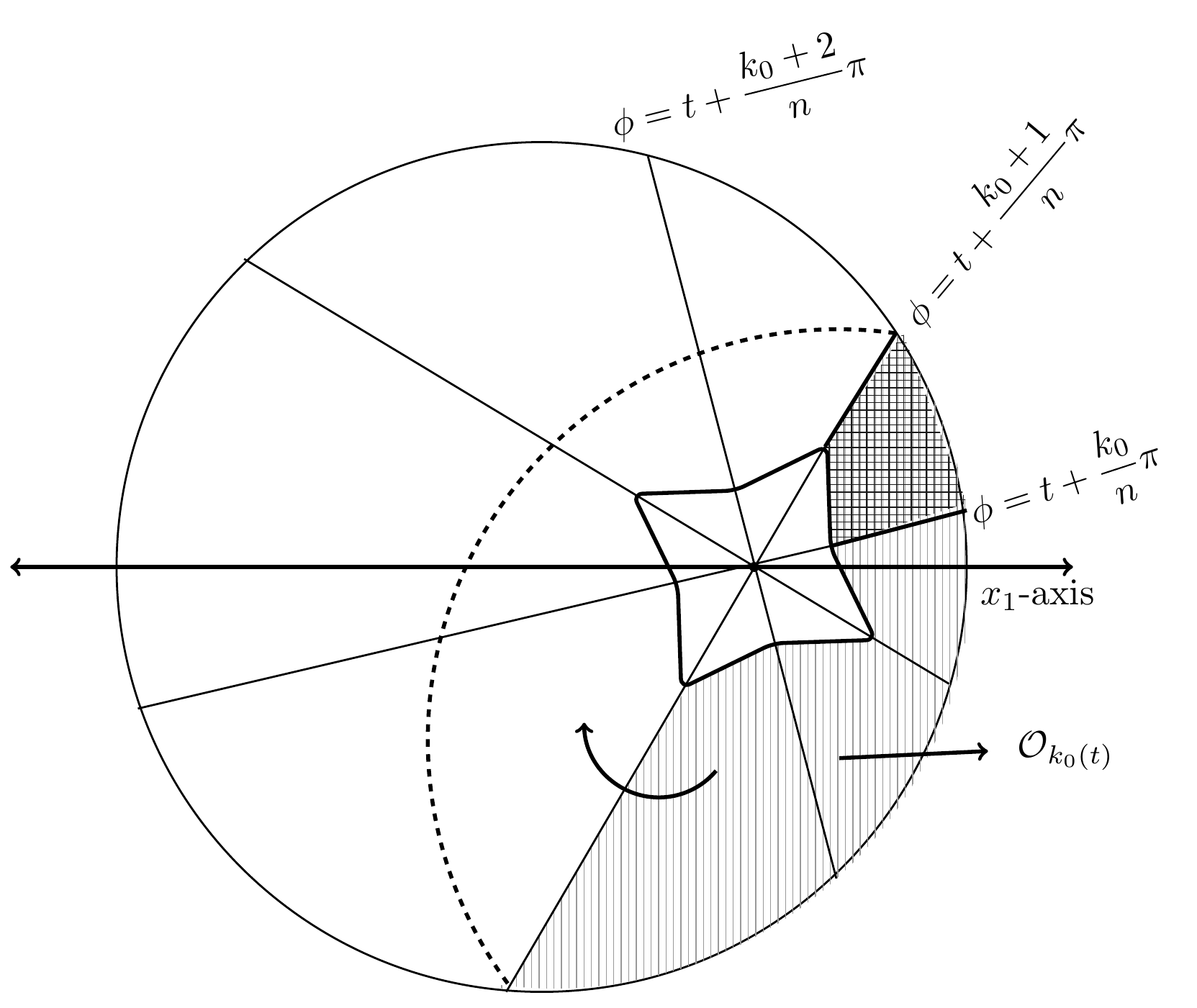}}
\caption{$\mathcal{O}_{k_0}(t)$ for $n=4$}
\label{fig:Ok0_reflection}
\end{figure}

Therefore, the non-constant function $w$ satisfies 
\begin{equation}\label{w}
\begin{aligned}
-\Delta w &= \lambda_1(t)\,  w & \mbox{ in } & H(t),\\
w & \leq 0, & \mbox{ on } & \partial H(t).
\end{aligned}
\end{equation}
Hence, by the maximum principle, $w$ is non-positive on the whole of $H(t)$. Therefore, from (\ref{w}) we have,  $\Delta{w} \geq 0$ in $H(t)$. Since $w$ achieves its maximal value zero on $\bigcup_{\substack{0\leq k \leq n-2\\k\equiv 0\bmod 2}}\left({\partial P_t \cap \sigma_{(k,k+1)}}\right)$ $\subset \partial H(t)$, by the Hopf maximum principle, one has 
$$
\dfrac{\partial w}{\partial \eta_t}(x) %= \dfrac{\partial y_1(t)}{\partial \eta_t}(x) -\dfrac{\partial y_1(t)}{\partial \eta_t}(x^\prime) 
{ >}0 ~~ \forall\; x\in \bigcup_{\substack{0\leq k \leq n-2\\k\equiv 0\bmod 2}}\left({\partial P_t \cap \sigma_{(k,k+1)}}\right).
$$
That is, $$ \dfrac{\partial y_1(t)}{\partial \eta_t}(x) -\dfrac{\partial y_1(t)}{\partial \eta_t}(x^\prime) { >} 0 ~~ \forall\; x\in \bigcup_{\substack{0\leq k \leq n-2\\k\equiv 0\bmod 2}}\left({\partial P_t \cap \sigma_{(k,k+1)}}\right).$$
Also, by the application of the Hopf maximum principle to problem (\ref{laplace_equation}), it follows that $\dfrac{\partial y_1(t)}{\partial \eta_t}(x) <0 ~\forall\; x \in \partial \Omega_t$. Thus,
\begin{equation}\label {drvsgn1}
\left|{\dfrac{\partial y_1(t)}{\partial \eta_t}(x)}\right|^2 -\left|{\dfrac{\partial y_1(t)}{\partial \eta_t}(x^\prime)}\right|^2 { <} 0 ~ \forall \; x \in \bigcup_{\substack{0\leq k \leq n-2\\k\equiv 0\bmod 2}}\left({\partial P_t \cap \sigma_{(k,k+1)}}\right).\end{equation}
Now, from (\ref{drvsgn1}) and (\ref{inrprdctsgn1}), it follows that the first term in (\ref{final_integralnew}) is strictly positive. Similarly, one can prove using (\ref{inrprdctsgn2}) that the second term in (\ref{final_integralnew}) is also strictly positive. %This implies that $\lambda_1^\prime(t) >0$. 
This proves the proposition for $n$ even. 
\end{proof}
\subsection{Proof of the main theorem%\ref{max_min}
} Theorem \ref{max_min}, for $n$ even, now follows from Propositions \ref{critical_points}, %Proposition 
\ref{complete_critical_points}, and equation (\ref{even_lambda}).

%As a consequence, using Th. \ref{critical_points}, we get that $k\pi/n$ are the only critical points of $\lambda'(t)$, which proves the second part of the theorem. 
\subsection{The $n$ odd case}%\ref{max_min}
{ %\begin{remark}\label{remarkpairing}
In the proof of Lemma \ref{containment_1}, we considered two consecutive sectors in each of the two hemispheres of the disk $B$ determined by the $z_t$-axis. We then took the reflection of the smaller sector of this pair into the bigger one about the axis of symmetry separating these two sectors. This was possible because the obstacle $P$ we consider had a $\mathbb{D}_n$ symmetry, where $n \geq 3$ was chosen to be even. As a result, the axes of symmetry of $P$ divide $B$ in even number of sectors in each of these hemispheres. 

When $n$ is odd, the axes of symmetry of $P$ divide $B$ in odd number of sectors in each of the hemispheres. Therefore, unlike the $n$ even case, it's not possible to find a complete pairing of consecutive sectors within each of the hemispheres. That is, if in the upper hemisphere we pair the %following 
consecutive sectors %of $\Omega_t$, namely 
$\sigma_{(k,k+1)}$ and $\sigma_{(k+1,k+2)}$, for each $k=0, 2, 4, \ldots n-3$, $k$ even, %which lie in the upper hemisphere, 
the sector $\sigma_{(n-1,n)}$ of the upper hemisphere remains unpaired. Similarly, if in the lower hemisphere we pair the %following 
consecutive sectors %of $\Omega_t$, namely 
$\sigma_{(k,k+1)}$ and $\sigma_{(k+1,k+2)}$, for each $k= n, n+2, \ldots, 2n-3$, $k$ odd, %which lie in the lower hemisphere,
 the sector $\sigma_{(2n-1,2n)}$ of the lower hemisphere remains unpaired. %We pair up the $i$-th and the $(i+1)$-th sectors for $i =1,3, \ldots n-2 $ in the upper hemisphere. While in the lower hemisphere, we pair the $i$-th and the $(i+1)$-th sectors for $i= n+1, n+3, \ldots, 2n-2$. 
%This is unlike the $n$ even case.%, this leaves the $(n-1)$-th and the $(2n-1)$-th sector unpaired. 
A pairing of these two unpaired sectors (shown in figure \ref{fig:sectors_pairing_odd} in solid black) with each other doesn't help either. %for the sign of the innerproduct $<v,\eta_t>$ on $\partial P_t$ corresponding to the smaller sector is opposite to the ones corresponding to the other pairs.
 %We then pair them to each other and % two consecutive sectors in each of these hemispheres. But since the number of sectors in each of these hemispheres are odd. This will leave one unpaired sector in each of the hemispheres. Now, we will make a pair of these two unpaired sector and 
%For $n$ odd,  This proof, in its current form, doesn't work for $n$ odd.}
%{ So, we tried the following. We paired the $i$-th and the $(i+1)$-th sectors up for $i =1,3, \ldots n-2 $ in the upper hemisphere. While in the lower hemisphere, we paired the $i$-th and the $(i+1)$-th sectors up for $i= n+1, n+3, \ldots, 2n-2$. This leaves the $n$-th and the $2n$-th sector unpaired. We then paired them to each other and % two consecutive sectors in each of these hemispheres. But since the number of sectors in each of these hemispheres are odd. This will leave one unpaired sector in each of the hemispheres. Now, we will make a pair of these two unpaired sector and 
%used an appropriate domain reflection method to carry out the analysis for this pair. But we failed to draw any conclusion about the sign of the derivative of $\lambda_1$. We try many different ways of pairing the sectors but in vain.}
%\end{remark}
%\begin{remark}\label{Remark2} 
For, %$n$ odd case, 
with respect to this pairing of sectors, equation (\ref{sum_integral}) breaks up into a sum of three terms. Here, the first term corresponds to the pairings of two consecutive sectors of the upper hemisphere, the second term corresponds to similar pairings in the lower hemisphere while the third term corresponds to the pairing of the left over sectors one each from each of the two hemispheres. It can be seen that though the first and the second term of this decomposition are positive, the third term turns out to be negative. This is because the inner product $\left< \eta_t,  v \right> $ corresponding to the third term has a different sign than the ones corresponding to the first two terms. { The reason for this is that $f$ is a strictly decreasing function of $\phi$ on $\sigma_{(k,k+1)}$ for $0\leq k\leq n-3, k$ even, and also for $n+1 \leq k \leq 2n-2, k$ even, but is a strictly increasing function of $\phi$ on $\sigma_{(2n-1,2n)}$.}
%not going to have a constant sign for the breakupof the integrals. %\end{remark}} 
As a result, we are unable to arrive at any conclusion about the sign of $\lambda_1^\prime(t) $, $t \in (0, \pi/n)$, for $n$ odd.}
 %\begin{remark} For $n$ odd case, even after the pairings described in Remark \ref{remarkpairing} and a similar reflection technique and rotating plane method fail to yield a consistent sign for all the three terms described in Remark \ref{Remark2}. In fact, first two terms of the sum give the same sign while the last term corresponding to the pairing of the left over unpaired sectors always end up giving the opposite sign.
%$f$ is a decreasing function of $\phi$ on $({0,\frac{\pi}{n}})$ for $n$ odd
% \end{remark}
{ Nevertheless, we provide some numerical evidence that enables us to make a conjecture that Theorem \ref{max_min} holds true for $n$ odd too.}
\begin{figure}[H] \centering 
\subfloat{
\includegraphics[width=0.7\textwidth]{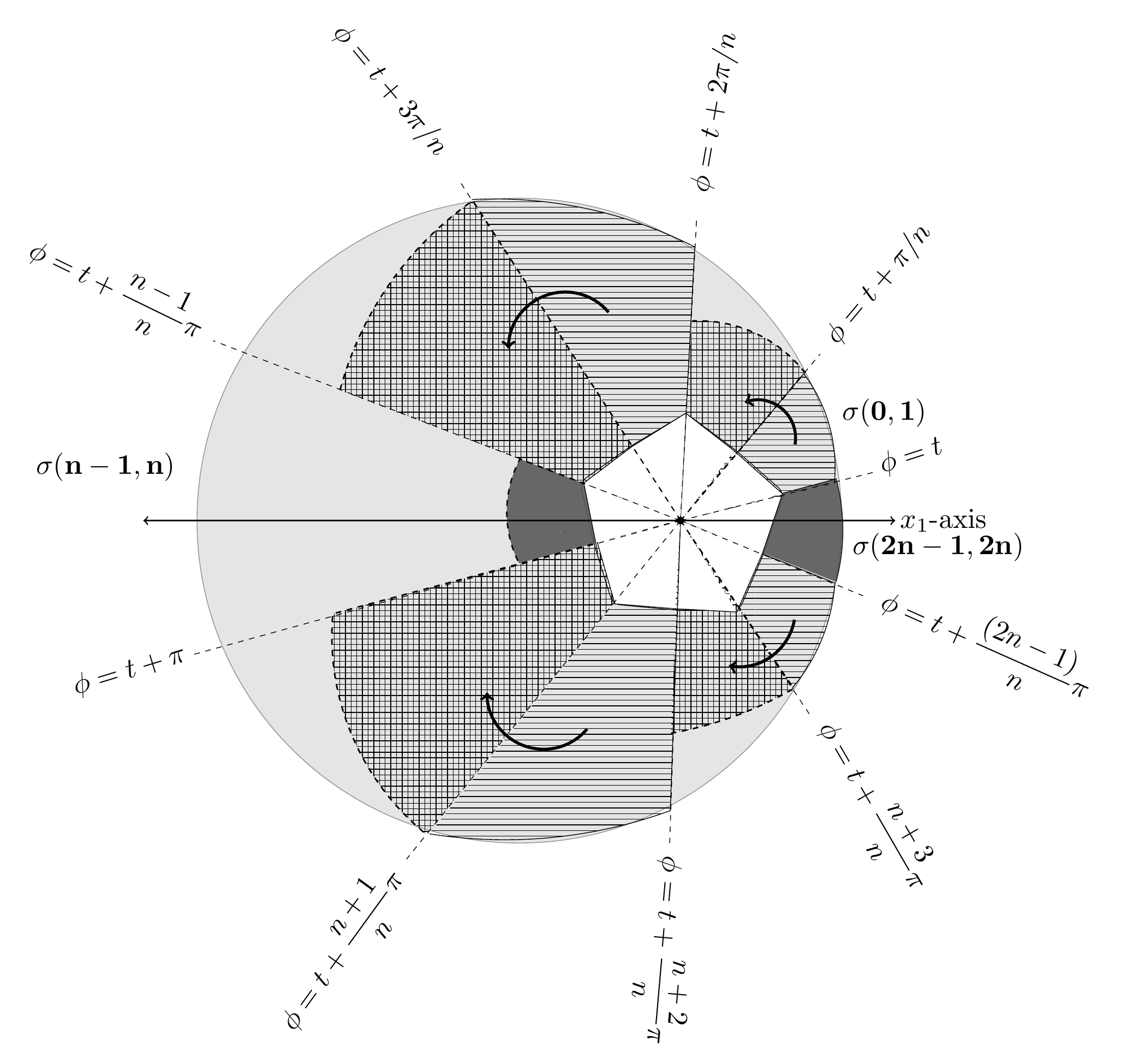}}\caption{Pairing of sectors of $\Omega_t$ for $n=5$}
\label{fig:sectors_pairing_odd}
\end{figure}

\section{Generalizations of Theorem \ref{max_min}}\label{sec:remarks}
Similar to the claims of \cite{elsoufi_kiwan}, extensions of Theorem \ref{max_min} to the following situations can be obtained up to slight
changes in the proof (indeed, only the Hadamard perturbation formula should be replaced by the
variation formula corresponding to the new functional):
\begin{enumerate}
\item Soft obstacles: Instead of considering the Dirichlet Laplacian on $B\setminus P$, we consider the Schr\"{o}dinger-type operator
$$
H(\alpha, P) := \Delta - \alpha \chi_P,
$$
acting on $H_0^1(B),$ where $\alpha >0$ and $\chi_P$ is the indicator function of $P$. For a compact simply connected subset $P$ of $\mathbb{R}^2$ satisfying assumptions \ref{assumption_A} and \ref{assumption_B}, the fundamental eigenvalue of $H(\alpha,P)$ achieves its maximum at an ``ON'' position and minimum at an ``OFF" position. A proof, similar to the one for Theorem \ref{max_min}, works for this case with the Hadamard variation formula replaced by the variational formula corresponding to the new functional.
\item Wells: This case corresponds to the operator $H(\alpha, P)$ with $\alpha <0$. In this case, the fundamental eigenvalue of $H(\alpha,P)$ achieves its maximum at an ``OFF" position and minimum at an ``ON" position.
\item Stationary problem: The problem now is to optimize the Dirichlet energy $E(\Omega):=\int_{\Omega} \|\nabla u\|^2 \, dx$ of the unique solution $u$ of the problem
\begin{equation}\label{stationary}
\begin{aligned}
-\Delta u &= 1 & \mbox{ in } & \Omega,\\
&u =0 & \mbox{ on } & \partial \Omega,\\
%\int_{\Omega_t} &u^2(x) \, dx =1, &  & \\
%u & > 0 & \mbox{ in } & \Omega_t.
\end{aligned}
\end{equation}
This problem was treated in Kesavan \cite{kesavan}  in the case $\Omega = B \setminus P$ where both $P$ and $B$
are disks. Under the assumptions of Theorem \ref{max_min} on $P$ and $B$, one can prove
that $E(B \setminus P)$ achieves its maximum when $P$ is at an “ON” position and its
minimum when $P$ is at an “OFF” position with respect to $B$.
\end{enumerate}
In addition to the list above, we also have the following generalizations. Due to space constraints, we refer to some useful articles for ideas and approach of the proof of these generalizations.
\begin{enumerate}
\item  Planar domains with non-smooth boundary: We know that for any bounded domain $\Omega$ having $\mathcal{C}^2$ 
boundary, the solution $u$ of (\ref{laplace_equation}) lies in $C^\infty(\bar{\Omega}) \subset H^2( \Omega)$. %But the domains 
Let us now consider a closed convex regular polygon $P$ in $\mathbb{R}^2$ enclosing area $A$. That is, $P$ satisfies only conditions (b), (c) and (d) of assumptions \ref{assumption_A} and \ref{assumption_B} and the boundary $\partial P$ of $P$ is a simple closed piecewise linear curve. Let $B$ be an open disc in $\mathbb{R}^2$ such that $B \supset \overline{co (C_2(P))}$.
 Then, the solution of (\ref{laplace_equation}) for $\Omega=B \setminus P$ in this case, is non-smooth and belongs to $H_0^{1+\delta}(\Omega)$, where $\delta\in\left({\dfrac{1}{2},\dfrac{3}{5}}\right)$ \cite{grisvard}. To avoid technical difficulties, in this paper we have worked with domains having $\mathcal{C}^2$ boundaries. Extension of our result to domains with non-smooth boundaries can be done using an approach similar to the one in \cite{aithal_acushla}.
\item  Two-dimensional space forms:
Consider the unit sphere $S^n :=\{x_1, x_2, \ldots, x_{n+1}) \in \mathbb{R}^{n+1} \,|\, \sum_{i=1}^{n+1} x_i^2 =1\}$ with induced Riemannian metric $\left<,\right>$ from the Euclidean space $\mathbb{R}^{n+1}$. Also consider the hyperbolic space $\mathbb{H}^n:= \{ (x_1, x_2, \ldots , x_{n+1}) \in \mathbb{R}^{n+1} \, |\,
\sum_{i=1}^{n}x_i^2
 −x_{n+1}^2= −1 \mbox{ and } x_{n+1} > 0$ with the Riemannian metric induced from the quadratic form $(x,y) := \sum_{i=1}^n x_i \, y_i  - x_{n+1}y_{n+1}$, where $x= (x_1,x_2, \ldots , x_{n+1})$ and $y = (y_1, y_2, \ldots, y_{n+1})$. The Riemannian manifolds $\mathbb{E}^n$, $S^n$ and $\mathbb{H}^n$ are all the space forms , i.e., complete simply connected Riemannian manifolds of constant sectional curvature. For the generalization of Theorem \ref{max_min} to the space forms, we consider space forms of dimension 2. They are denoted by $M_\kappa^2$ in \cite{anisa-aithal1} and \cite{aithal_raut} where $\kappa $ denotes the sectional curvature of the Riemannian manifold $(M,g)$ under consideration. Here, $\kappa =-1, 0$ and $1$ for $\mathbb{H}^2$, $\mathbb{E}^2$ and $S^2$, respectively. 
 Let $B$ be any geodesic ball of radius
$r_1$ in $S^n$, $\mathbb{H}^n$. We choose $r_1< \pi$ for the case of $S^n$.  Let $\kappa \in \{-1, 0 ,1\}$. %B0 be any ball of radius r0 such that B0  B1. Consider the family
%5F D fB1 nB0g of domains in Sn.Hn/.We study the extrema of the following functionals:
\begin{itemize} 
\item obstacle with non-smooth boundary: Let $P$ be a regular polygon of $n$ sides in $M_\kappa^2$ such that $P \subset B$; and $P$, $B$ having distinct centers. For a description of such polygons on $M_\kappa^2$ please refer to \cite{aithal_raut}. Then, Theorem \ref{max_min} of this paper holds for the family of domains $\Omega =B \setminus P$ over $M_\kappa^2$ too. \cite{aithal_raut} will be useful in proving this generalization. 
\item obstacle with smooth boundary: Anisa and Aithal \cite{anisa_aithal} developed a shape calculus on general Riemannian manifolds of dimension $n$, and used it to prove the analogues of the results of Hersch \cite{hersch}, Kesavan \cite{kesavan} and Ramm-Shivakumar \cite{Ramm-Shivakumar} on space-forms. The reflection method worked there just as Euclidean space, because reflection in a hyperplane is an isometry in any space form, and so it commutes with the Laplace-Beltrami operator. One can come up with a description of compact simply connected subset of $M_\kappa^2$ satisfying assumptions \ref{assumption_A}, \ref{assumption_B} of this paper such that $B \supset \overline{co(C_2(P))}$. It can be taken as a small project to generalize the main theorem of \cite{elsoufi_kiwan} and to generalize our main theorem, viz., Theorem \ref{max_min}, for the corresponding family of domains in $M_\kappa^2$.
\end{itemize}
\end{enumerate}
 %In this section, we talk about some generalizations of the Theorem \ref{max_min} from the Euclidean case to some other Riemannian manifolds of dimension 2. We also talk about generalisation to differential equations involving Schr\"{o}dinger-type operators. Theorem \ref{max_min} can be generalized to other types of obstacles and to other topological spaces. 
\section{Numerical results} \label{sec:num_results}

We give some numerical evidence supporting Theorem \ref{max_min}. We take $n=4,5$. That is, we take $P$ to be a compact simply connected subset of $\mathbb{R}^2$ satisfying assumptions \ref{assumption_A}, \ref{assumption_B} for $n=4,5$. Recall that the function $f(\phi)$, the distance of a point $f(\phi)e^{i \phi} \in \partial P$ from the center of $P$, is a decreasing function of $\phi$ for $\phi \in \left({0,\dfrac{\pi}{n}}\right)$. We solve the boundary value problem (\ref{laplace_equation}) in the domain $\Omega= B \setminus P$ using finite element method with $P^1$ elements (see e.g., \cite{SR1,SR2}) on a mesh with element size $h=0.018$. The mesh is shown in Figure \ref{fig:mesh}. 
\begin{figure}[H]
\begin{center}
\includegraphics[width=0.3\textwidth]{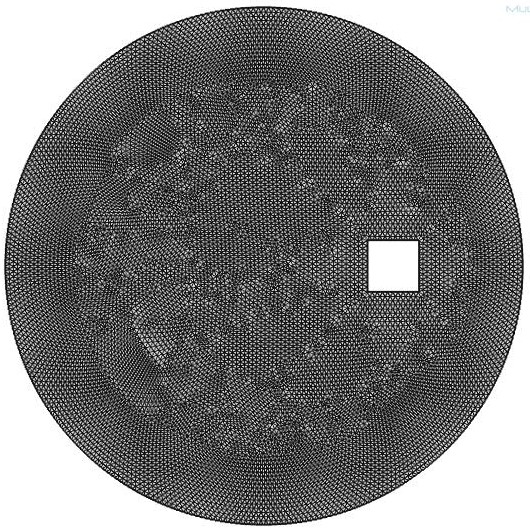}   \hspace{10mm }
\includegraphics[width=0.4\textwidth]{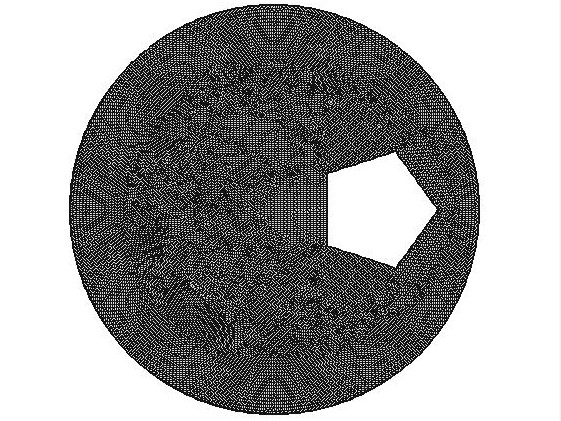}    
\end{center}
\caption{The mesh }\label{fig:mesh}
\end{figure}

\begin{figure}
  \centering
    \subfloat[Initial OFF position]{%
     \label{fig:min} \includegraphics[scale=0.3,keepaspectratio]{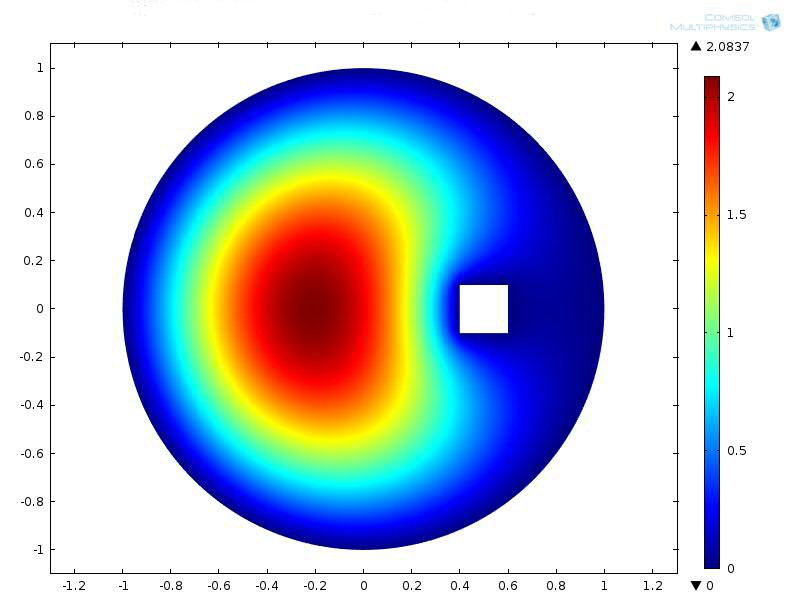}
   }
    \subfloat[Intermediate Position]{%
      \label{fig:middle}\includegraphics[scale=0.3,keepaspectratio]{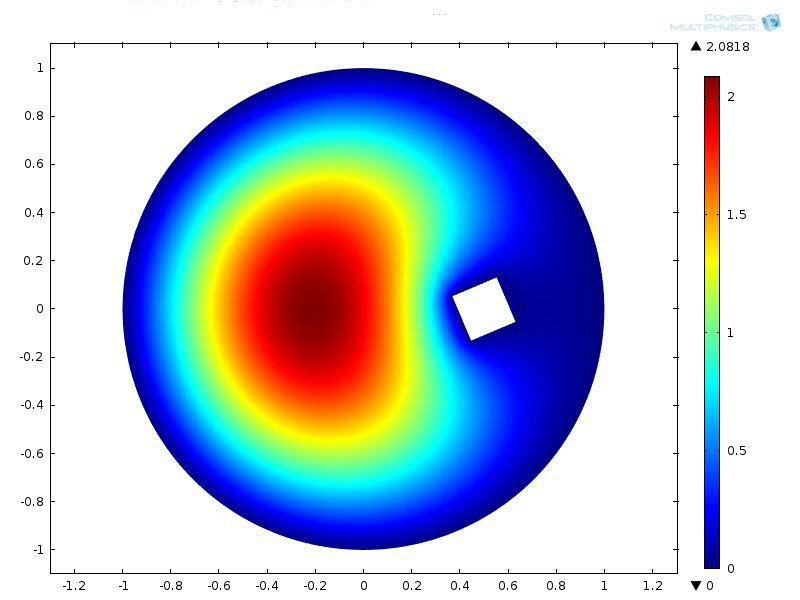}
      }\\
        \subfloat[ON position]{%
      \label{fig:max}\includegraphics[scale=0.3,keepaspectratio]{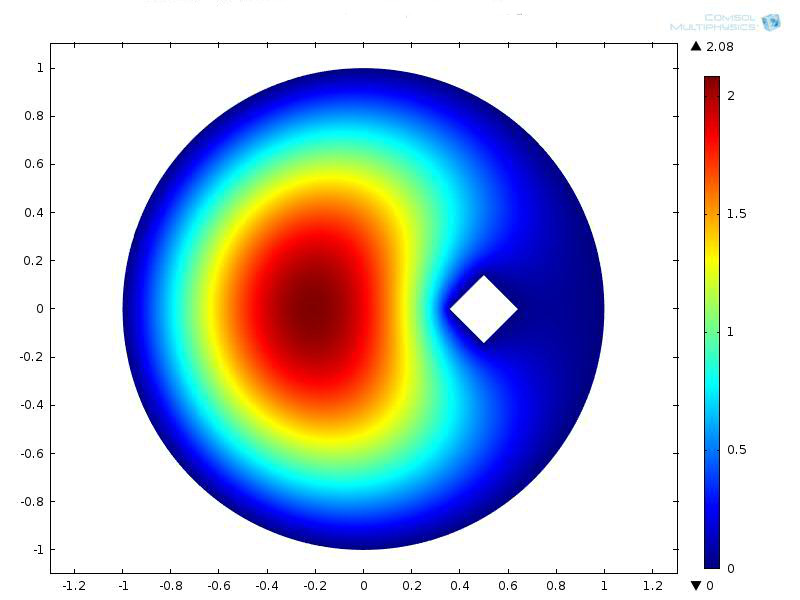}
      }\\
      \subfloat[Another intermediate position]{%
      \label{fig:middle_2}\includegraphics[scale=0.3,keepaspectratio]{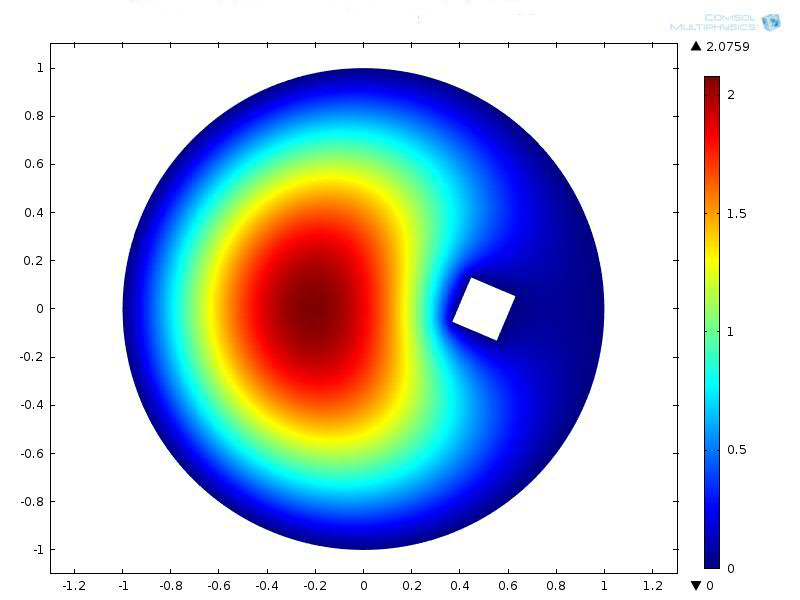}
      }
      \subfloat[OFF position]{%
      \label{fig:min_2}\includegraphics[scale=0.3,keepaspectratio]{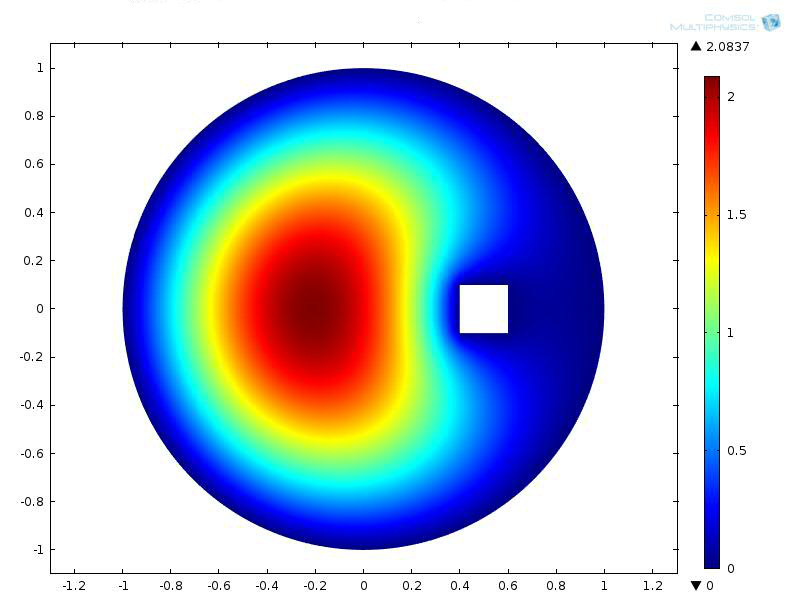}
      }
      \caption{Simulations of ON, OFF and intermediate positions of the square. 
}
      \label{simulation_results}
      \end{figure}
 
We validate Theorem \ref{max_min} for the square obstacle with $n=4$. The initial configuration, given in Figure \ref{fig:min}, is an OFF configuration which is a minimizing configuration according to Proposition \ref{critical_points}, Proposition \ref{complete_critical_points}, and equation (\ref{even_lambda}). This is justified by the numerical value of $\lambda_1=7.5735$ given in Table \ref{table1}. We then rotate $P$ by an angle $\frac{\pi}{8}$ about its center in the anticlockwise direction. This gives an intermediate configuration of the domain $\Omega = B \setminus P \in \mathcal{F}$, cf. Figure \ref{fig:middle} with an increased value of $\lambda_1$. It increases further on rotating by the same angle $\frac{\pi}{8}$ further in the anticlockwise direction. This rotation makes $\Omega$ attain an ON position with respect to $B$, see Figure \ref{fig:max}. One more rotation of $P$ about its center by an angle $\frac{\pi}{8}$ leads to another intermediate configuration, see Figure \ref{fig:middle_2}. This rotation now results in a decrease in the value of $\lambda_1$. A final rotation of $ P$ again about its center by the same angle of $\frac{\pi}{8}$ brings $P$ back to an OFF configuration with respect to the disk $B$, see Figure \ref{fig:min_2}. We note that this time $\lambda_1$ attains its minimum value again. We refer to Table \ref{table1} for the numerical observations.

\begin{table}[H]
\begin{center}
\begin{tabular}{|c| c| c|}
\hline
$\theta$ & $\lambda_1$ & Configuration  \\ [0.5ex]
\hline
0 &7.5735& OFF\\
$\pi/8$ &7.5739&--\\
$\pi/4$ &7.5742&ON\\
$3\pi/8$ &7.5739&--\\
$\pi/2$ &7.5735&OFF\\[1ex]
\hline
\end{tabular}
\end{center}
\caption{Variation of $\lambda_1$ with rotations of the square $P$ about its center by an angle $\theta$ measured with respect to the positive $x_1$-axis.}
\label{table1}
\end{table}

{We next show that Theorem \ref{max_min} is true for odd $n$ too by demonstrating quantitative and qualitative results for an obstacle having pentagonal shape. The initial configuration, given in Figure \ref{fig:pentagon_min}, is an OFF configuration which turns out to be a minimizing configuration. This is justified by the numerical value of $\lambda_1=9.089$ given in Table \ref{table2}. We then rotate $P$ by an angle $\frac{\pi}{10}$ about its center in the anticlockwise direction. This gives an intermediate configuration of the domain $\Omega = B \setminus P \in \mathcal{F}$, cf. Figure \ref{fig:pentagon_middle} with an increased value of $\lambda_1$. It increases further on rotation by the same angle $\frac{\pi}{10}$ further in the anticlockwise direction. This rotation makes $\Omega$ attain an ON position with respect to $B$, see Figure \ref{fig:pentagon_max}. One more rotation of $P$ about its center by an angle $\frac{\pi}{10}$ leads to another intermediate configuration, see Figure \ref{fig:pentagon_middle_2}. This rotation now results in a decrease in the value of $\lambda_1$. A final rotation of $ P$ again about its center by the same angle of $\frac{\pi}{10}$ brings $P$ back to an OFF configuration with respect to the disk $B$, see Figure \ref{fig:pentagon_min_2}. We note that this time $\lambda_1$ attains its minimum value again. We refer to Table \ref{table2} for the numerical observations.}

      \begin{figure}[H]
  \centering
    \subfloat[Initial OFF position]{%
     \label{fig:pentagon_min} \includegraphics[scale=0.4,keepaspectratio]{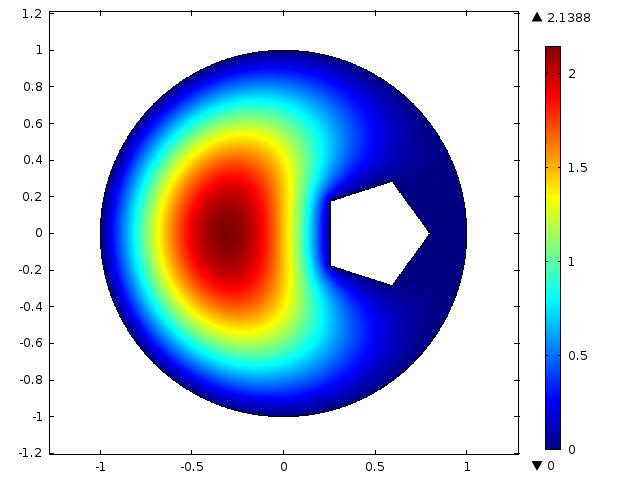}
   }
    \subfloat[Intermediate Position]{%
      \label{fig:pentagon_middle}\includegraphics[scale=0.4,keepaspectratio]{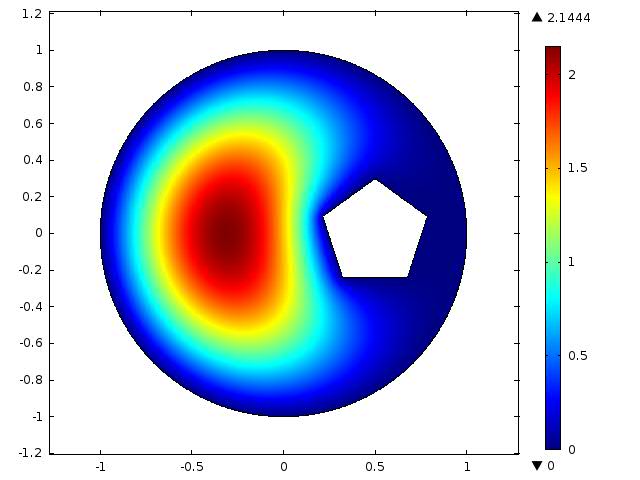}
      }\\
        \subfloat[ON position]{%
      \label{fig:pentagon_max}\includegraphics[scale=0.4,keepaspectratio]{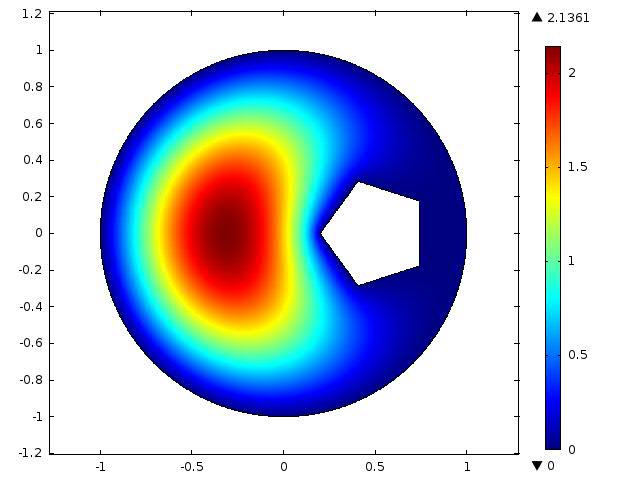}
      }\\
      \subfloat[Another intermediate position]{%
      \label{fig:pentagon_middle_2}\includegraphics[scale=0.4,keepaspectratio]{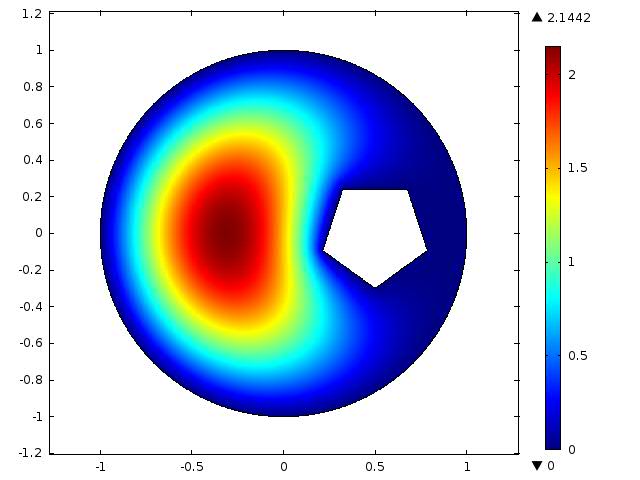}
      }
      \subfloat[OFF position]{%
      \label{fig:pentagon_min_2}\includegraphics[scale=0.4,keepaspectratio]{pentagon_off.jpg}
      }
      \caption{Simulations of ON, OFF and intermediate positions of the pentagon. 
}
      \label{simulation_results_pentagon}
      \end{figure}

\begin{table}[H]
\begin{center}
\begin{tabular}{|c| c| c|}
\hline
$\theta$ & $\lambda_1$ & Configuration  \\ [0.5ex]
\hline
0 &9.089& OFF\\
$\pi/10$ &9.090&--\\
$\pi/5$ &9.092&ON\\
$3\pi/10$ &9.090&--\\
$2\pi/5$ &9.089&OFF\\[1ex]
\hline
\end{tabular}
\end{center}
\caption{Variation of $\lambda_1$ with rotations of the pentagon $P$ about its center by an angle $\theta$ measured with respect to the positive $x_1$-axis.}
\label{table2}
\end{table}
%{ Please add a table for $n$ odd too, $n=5$.}
\section{Conclusion}
Let $P$ be a compact simply connected subset of $\mathbb{R}^2$ satisfying assumptions \ref{assumption_A}, \ref{assumption_B} and let $B$ be an open disk in $\mathbb{R}^2$ of radius $r_1$ such that $B \supset \overline{co(C_2(P))}$. For $t\in\mathbb{R}$, let $\rho_t \in SO(2)$ denote the rotation in $\mathbb{R}^2$ about the origin $\underline{o}$ in the anticlockwise direction by an angle $t$%, i.e., for $\zeta \in \mathbb{C} \cong \mathbb{R}^2  $, we have $\rho_t \zeta :=e^{\textbf{i} t} \zeta$
. Now fix $t \in [0, 2\pi)$. %Let $P_t:=\rho_t(P)$. Then, in polar co-ordinates,
Let $\Omega_t:= B \setminus \rho_t(P)$ and $\mathcal{F}:= \{\Omega_t \, |\, t \in [0,2 \pi)\}$. Then, using a sector reflection technique, rotating plane method and Hadamard perturbation formula, we proved Theorem \ref{max_min} { for $n$ even, $n \geq 3$,} which describes the extremal configurations for the fundamental Dirichlet eigenvalue $\lambda_1(\Omega_t)$ for $\Omega_t \in \mathcal{F}$. This theorem also characterizes all the maximizing and the minimizing configurations for $\lambda_1$ over $\mathcal{F}$. { Equation (\ref{even_lambda}), Propositions \ref{critical_points} and \ref{complete_critical_points} imply Theorem \ref{max_min} for $n$ even, $n \geq 3$. 

Equation (\ref{even_lambda}) and Proposition \ref{critical_points} hold for any $n \geq 3$, even or odd. That is, we are able to identify some of the critical points of the map $t \longmapsto \lambda_1(t)$ and know that now it is enough to study the sign of $\lambda_1^\prime$ only on $(0, \frac{\pi}{n})$. Our proof of Proposition \ref{complete_critical_points} works only for $n$ { even}, $n \geq 3$. We highlight some of the difficulties faced in proving Proposition \ref{complete_critical_points} for $n$ odd.}

We provide some numerical evidence to validate our main theorem, i.e., Theorem \ref{max_min}, for the case where the obstacle $P$ has { $\mathbb{D}_n$} symmetry { for $n=4$. We also provide some numerical evidence for $n=5$ and conjecture that Theorem \ref{max_min} holds true for $n$ odd too.}

We give many different and interesting generalizations of our result in section \ref{sec:remarks}. Soft obstacles and wells for Schr\"{o}dinger-type operator are addressed in the generalizations. Optimal configurations for the energy functional for the stationary problem (\ref{stationary}) can also be obtained in a similar manner. The generalizations also include results for $P$ having non-smooth boundary and also the case where the ambient space for the family of admissible domains $B \setminus P$ is non-Euclidean. %to the case where the ambient space for $B \setminus P$ is non-Euclidean. and for the differential equations involving the Schr\"{o}dinger-type operators.  

%\section{Conclusion} \label{sec:conclusion}
%In this work, we have solved an Dirichlet Laplacian eigenvalue optimization problem related to the placement of an obstacle $P$ in a disk $B\subset \mathbb{R}^2$ in the case when the centres of mass of $P$ and $B$ are non-concentric. We consider the case when the obstacle $P$ is invariant under the action of a dihedral group $\mathbb{D}_n$, with $n \geq 3$, $n$ even. The extremal positions are obtained precisely when the vertices of $P$ are along a diameter of $B$. The proof relies on the Hadamard variation formula and domain reflection methods. We also provided generalizations of our result to two-dimensional manifolds and for Schr\"{o}dinger-type operators. Finally, we validated our theoretical results using numerical experiments with a square-shaped obstacle. 

\end{document}